\documentclass{article}
\usepackage{fullpage}
\usepackage{amsthm}
\usepackage{amsmath}
\usepackage{amssymb}
\usepackage{authblk}
\usepackage{t1enc}
\usepackage[utf8]{inputenc}
\usepackage{colonequals}
\usepackage{float}
\usepackage{tikz}
\usetikzlibrary{arrows.meta, decorations.shapes, shapes.geometric, calc, positioning, backgrounds, shapes, decorations.markings}
\usepackage{caption}
\usepackage{cite}
\usepackage[nobysame]{amsrefs}
\usepackage{thmtools}
\usepackage{enumerate}
\usepackage{enumitem}
\usepackage{multirow}
\usepackage{booktabs}
\usepackage{lscape}
\usepackage{array}
\usepackage{titling}
\thanksmarkseries{arabic}
\usepackage{placeins}

\usepackage{algorithm}
\usepackage[noend]{algpseudocode}

\usepackage[unicode,colorlinks=true,citecolor=green!40!black,linkcolor=red!20!black,urlcolor=blue!40!black,filecolor=cyan!30!black]{hyperref}

\makeatletter
 \newcommand{\linkdest}[1]{\Hy@raisedlink{\hypertarget{#1}{}}}
\makeatother

\renewcommand{\eprint}[1]{#1}
\renewcommand{\PrintYear}[1]{#1.}
\DefineSimpleKey{bib}{volumeadditionalinformation}
\DefineSimpleKey{bib}{issue}
\DefineSimpleKey{bib}{articlenumber}

\BibSpec{article}{
 +{}  {\PrintAuthors}          {author}
 +{,} { \textit}               {title}
 +{,} { }                      {journal}
 +{}  { \textbf}               {volume}
 +{}  {\parenthesize}          {issue}
 +{}  { \parenthesize}         {volumeadditionalinformation}
 +{:} { }                      {pages}
 +{,} { article no.~}          {articlenumber}
 +{,} { \PrintYear}            {year}
 +{}  { Available at \eprint}  {eprint}
 +{}  { \PrintDOI}             {doi}
 +{}  { \parenthesize}         {note}
}

\DeclareFontShape{T1}{cmr}{m}{scit}{<->ssub*cmr/m/sc}{}

\allowdisplaybreaks

\theoremstyle{definition}
\newtheorem{definition}{Definition}[section]
\newtheorem{theorem}[definition]{Theorem}
\newtheorem{lemma}[definition]{Lemma}
\newtheorem{proposition}[definition]{Proposition}

\newtheorem{corollary}[definition]{Corollary}

\newtheorem{remark}[definition]{Remark}

\title{Color-avoiding connected colorings and orientations}
\author{J\'{o}zsef Pint\'{e}r\thanks{Department of Stochastics, Institute of Mathematics, Budapest University of Technology and Economics.} \textsuperscript{,}\thanks{HUN-REN--BME Stochastics Research Group.} \and Kitti Varga\thanks{Department of Computer Science and Information Theory, Faculty of Electrical Engineering and Informatics, Budapest University of Technology and Economics.} \textsuperscript{,}\thanks{HUN-REN--ELTE Egerv\'{a}ry Research Group.}}

\begin{document}

\maketitle

\begin{abstract}
We study network robustness under correlated failures modeled by colors, where each color represents a class of edges or vertices that may fail simultaneously. An edge-colored graph is said to be edge-color-avoiding $k$-edge-connected if it remains $k$-edge-connected after the removal of all edges of any single color. We characterize the graphs that admit such a coloring and show that, when $k=1$, one can determine in polynomial time both the minimum number of colors required and a coloring achieving it; while the problem becomes NP-hard for $k \ge 2$. We also investigate the problem of orienting the edges of a graph so that the resulting digraph remains strongly or rooted connected even after the removal of all arcs of any single color. In addition, we explore generalizations involving vertex-colorings, $k$-vertex-connectivity, simultaneous failures of multiple colors and matroids. 
\end{abstract}

\section{Introduction} \label{section:intro}

In real-world networks, failures often occur in correlated ways. Multiple links may simultaneously fail due to, for example, hardware module failure, or jamming of a shared frequency band. These correlated failures can be modeled by assigning failure labels, often denoted by colors, to the edges or vertices of a network, grouping together those that fail as a unit. Such models naturally arise in diverse areas including cybersecurity~\cite{sheyner02}, infrastructure network design~\cite{coudert16} and survivable routing with shared risk groups~\cite{cdp07}.

\paragraph{Previous works.}
One of the earliest practical instances of correlated failure models can be found in the security-oriented work of Sheyner et al.~\cite{sheyner02}, where the authors use colors on edges to represent different types of attacks. A security analyst might ask for the smallest set of attack types to guard against, such that the attacker is prevented from reaching a critical target --- a problem that is, in essence, finding a minimum label cut.

Since then, a wide variety of related problems have emerged under different names. Despite their common structure --- failure groups defined by sets of edges --- the literature across these domains has remained fragmented. Xu and Farag{\'o}~\cite{xu24} explicitly highlight this disconnection, noting that the same problem has been studied under different names in different communities, often with missing or misleading citations.

The theoretical computer science community developed so-called label cut and colored cut frameworks~\cites{labelcut,morawietz22}. In parallel, network designers studied similar issues under the names Shared Risk Link Groups~\cite{cdp07} or Shared Risk Resource Groups~\cite{coudert16}. In this model, each color represents a group with shared risk: a failure affecting all edges labeled with that color simultaneously.

A broader complexity landscape of these problems has been charted in the work of Coudert et al.~\cite{coudert16}, who analyzed a wide array of survivability problems under correlated edge failures. Their results highlight that while the general case is typically NP-hard, certain structural restrictions -- such as when each color induces a connected subgraph or when all edges of a given color are incident to a single vertex --lead to polynomial-time algorithms. These distinctions are especially relevant in practice, where such structured coloring patterns often arise naturally in layered or modular network architectures.

Building on the label cut framework, Ghaffari et al.~\cite{gkp17} introduced the so-called hedge cut problem, where the goal is to disconnect the graph by removing a minimum number of failure groups. This formulation captures scenarios where resilience against correlated failures is required at a global scale, rather than between specific source-target pairs.

They provided a polynomial-time approximation scheme and a quasi-polynomial exact algorithm using random sampling and contractions. Recently, Jaffke et al.~\cite{jaffke22} proved that this quasi-polynomial complexity is essentially tight: under the randomized Exponential Time Hypothesis, no significantly faster exact algorithm exists for the hedge cut problem. Complementing this result, Fomin et al.~\cite{fomin24} showed that the decision version of the hedge cut problem is fixed-parameter tractable when parameterized by the number~$\ell$ of colors to be removed.

A very similar framework was introduced under the name color-avoiding connectivity by Krause et al.~\cites{article:physicists2,article:physicists3} for vertex-colored graphs, and later extended to edge-colored graphs by Kadovi{\'c} et al.~\cite{article:physicists1}. The central problem is to find a maximum subset of nodes such that there exists a communication path between two nodes that avoids any single color, or more generally, any subset of up to~$\ell$ colors.

From a computational complexity point of view, Molontay and Varga~\cite{article:color-avoiding_components} analyzed the problem of finding color-avoiding connected components, showing that while some variants are solvable in polynomial time, others are NP-hard. In~\cite{article:CA_spanning_subgraphs}, we studied the problem of finding color-avoiding connected spanning subgraphs, proving its NP-hardness and proposing polynomial-time approximation algorithms. We also extended the notion of color-avoiding connectivity to matroids by defining courteously colored matroids, following the terminology of courteous edge-colorings by DeVos et al.~\cite{article:courteous}.

\paragraph{Our contributions.}

In this paper, we present a framework covering various cases such as edge-colored graphs, vertex-colored graphs, simultaneous failures of multiple colors, and stronger connectivity requirements such as $k$-edge- or $k$-vertex-connectivity. Our framework also incorporates directed networks, allowing us to explore color-avoiding strong and rooted connectivity.

 First, we characterize those graphs that can be colored in color-avoiding connected manners. In addition, we investigate the problem of determining the minimum number of colors required for such a coloring and actually finding such a coloring. We also study the problem of finding an orientation of the edges of a colored undirected graph, or finding an orientation and a coloring of the edges of an uncolored undirected graph to achieve color-avoiding strong or rooted connectivity. Furthermore, we extend our findings to encompass $k$-edge- and $k$-vertex-connectivity, as well as scenarios involving the simultaneous elimination of multiple colors from the network. Additionally, in certain cases, we extend our analysis to matroids and the assignment of multiple colors to edges or vertices.

 % \vspace{-2pt}

 \bigskip

 The paper is organized as follows. In Section~\ref{section:preliminaries}, we provide a concise overview of the graph-theoretic and matroid background essential for understanding some parts of the paper. Following this, we define various forms of color-avoiding connectivity and explore interesting relations.

 Section~\ref{section:colorings} delves into the investigation of color-avoiding connected colorings, wherein we give a necessary and sufficient condition for a graph admitting a color-avoiding connected coloring. We also study determining the minimum number of colors required for such colorings. First, we focus on edge-colored graphs, then on vertex-color-avoiding connectivity, and finally, on internally vertex-color-avoiding connectivity along with their respective directed counterparts.

 In Section~\ref{section:orientations}, we study color-avoiding connected orientations, aiming to determine the complexity of deciding whether a colored graph can be oriented to achieve color-avoiding strong or rooted connectivity. Finally, we also consider the problem of simultaneously finding an orientation and a coloring to obtain a color-avoiding strongly or rooted connected digraph.

\section{Preliminaries} \label{section:preliminaries}

In this article, we study color-avoiding connected graphs and courteously colored matroids. The graphs considered in this article might have parallel edges, and they are not necessarily properly colored\footnote{A \emph{proper edge-coloring} (or \emph{proper vertex-coloring}) is an assignment of colors to the edges (or vertices) in which no two incident edges (or no two adjacent vertices) receive the same color.} unless otherwise stated. First, we recall some important definitions and notation.

The set of positive integers is denoted by $\mathbb{Z}_+$. For two sets $X$ and $Y$, the \emph{set difference} of $X$ and $Y$ is denoted by $X-Y$. 
% A graph that may have multiple edges and loops is called a \emph{multigraph}.
For a graph $G$, we denote the \emph{number of its connected components} by $c(G)$. For a graph~$G$ and a set of edges $E' \subseteq E(G)$ or a set of vertices $V' \subseteq V(G)$, we denote by $G-E'$ or by $G-V'$ respectively the graph that is obtained from $G$ by deleting the edges of $E'$ or the vertices of $V'$ from it.

A graph is called \emph{$k$-edge-connected} (or \emph{$k$-vertex-connected}\footnote{The standard definition of $k$-vertex-connectivity requires the graph to have at least $k+1$ vertices and to remain connected whenever fewer than $k$ vertices are removed. This ensures that the \emph{connectivity number} (that is, the smallest $k$ for which the graph is $k$-vertex-connected) can be defined for complete graphs. Under both definitions, the complete graph $K_n$ is $(n-1)$-vertex-connected but not $n$-vertex-connected. However, with the usual definition, a two-vertex graph with $k \ge 2$ parallel edges or a one-vertex graph with $k \ge 1$ loops is not $k$-vertex-connected, while in this article we consider these graphs $k$-vertex-connected.}) if it has at least two vertices and there are at least $k$ pairwise edge-disjoint (or internally vertex-disjoint) paths between any two vertices, or if it is a one-vertex graph with at least $k$ loops, where $k \in \mathbb{Z}_+$.\footnote{A graph is 1-edge-connected if and only if it is 1-vertex-connected; such graphs and the one-vertex graph with no loops, are simply called \emph{connected}.}

An \emph{edge-cut}, \emph{vertex cut}, or \emph{mixed cut} in a connected graph is a set of edges, vertices, or a combination of both, respectively, whose removal disconnects the graph. By Menger’s classical theorems~\cite{article:menger}, a connected graph on at least two vertices is $k$-edge-connected (or $k$-vertex-connected) if and only if it has no edge-cut (or mixed cut) of size at most $k-1$.

A directed graph and a directed path are simply called a \emph{digraph} and a \emph{dipath}, respectively, and a directed edge is called an \emph{arc}. A digraph is called \emph{strongly $k$-arc-connected} (or \emph{strongly $k$-vertex-connected}) if it has at least two vertices and there exist at least $k$ pairwise arc-disjoint (or at least $k$ pairwise internally vertex-disjoint) dipaths from any vertex to any other vertex, or if it is a 1-vertex graph with at least $k$ loops, where $k \in \mathbb{Z}_+$.\footnote{A digraph is strongly 1-arc-connected if and only if it is strongly 1-vertex-connected; such digraphs and the one-vertex digraph with no loops are simply called \emph{strongly connected}.} A digraph with a given vertex $r$ is called \emph{$r$-rooted $k$-arc-connected} (or \emph{$r$-rooted $k$-vertex-connected}) if it has at least two vertices and there exist at least $k$ pairwise arc-disjoint (or at least $k$ pairwise internally vertex disjoint) paths from $r$ to any other vertex, or if it is a one-vertex graph with at least $k$ loops.\footnote{A digraph is rooted 1-arc-connected if and only if it is rooted 1-vertex-connected; such digraphs and the one-vertex digraph with no loops are simply called \emph{rooted connected}.}

A \emph{matroid} $\mathcal{M} = (S, \mathcal{I})$ is a pair formed by a finite \emph{ground set} $S$ and a family of subsets $\mathcal{I} \subseteq 2^S$ called \emph{independent sets} satisfying the \emph{independence axioms}:

\begin{enumerate}[topsep=2pt, itemsep=0pt]
 \item[(I1)] $\emptyset \in \mathcal{I}$,
 \item[(I2)] for any $X,Y \subseteq S$ with $X \subseteq Y$, if $Y \in \mathcal{I}$, then $X \in \mathcal{I}$,
 \item[(I3)] for any $X,Y \in \mathcal{I}$ with $|X| < |Y|$, there exists $e \in Y-X$ such that $X \cup \{e\} \in \mathcal{I}$.
\end{enumerate}

\noindent The maximal independent subsets of $S$ are called \emph{bases}. A \emph{circuit} is an inclusion-wise minimal non-in\-de\-pen\-dent set, and a \emph{loop} is an element which forms a circuit on its own in the matroid. The \emph{rank} of a set $X \subseteq S$ in the matroid, denoted by $r(X)$, is the maximum size of an independent subset of $X$. The \emph{rank of the matroid} is the rank of its ground set. 

A \emph{graphic matroid} is a matroid whose independent sets can be represented as the edge sets of forests of a graph; if $G$ is a graph, then the matroid associated to $G$ is denoted by $\mathcal{M}(G)$. If the underlying graph is connected, then the bases of the graphic matroid are the spanning trees of the graph; for an example, see Figure~\ref{fig:graphic_matroid}.

\begin{figure}[!ht]
\centering
\begin{tikzpicture}[scale=1.2]
 \tikzstyle{vertex}=[draw,circle,fill,minimum size=6,inner sep=0]
 
 \node[vertex] (a1) at (0,0)  {};
 \node[vertex] (b1) at (0,1) {};
 \node[vertex] (c1) at (1,1) {};
 \node[vertex] (d1) at (1,0) {};
 
 \draw[very thick] (a1) -- (b1) node[midway, left] {$e_2$};
 \draw[very thick] (c1) -- (d1) node[midway, right] {$e_4$};
 \draw[very thick] (b1) -- (c1) node[midway, above] {$e_1$};
 \draw[very thick] (a1) -- (d1) node[midway, below] {$e_5$};
 \draw[very thick] (a1) -- (c1) node[midway, xshift=-5pt, yshift=5pt] {$e_3$};
\end{tikzpicture}
\caption{An example for a graphic matroid: the ground set of the matroid is $\{ e_1, e_2, e_3, e_4, e_5 \}$ and the bases are $\{ e_1, e_2, e_4 \}$, $\{ e_1, e_2, e_5 \}$, $\{ e_1, e_3, e_4 \}$, $\{ e_1, e_3, e_5 \}$, $\{ e_1, e_4, e_5 \}$, $\{ e_2, e_3, e_4 \}$, $\{ e_2, e_3, e_5 \}$, and $\{ e_2, e_4, e_5 \}$.}
\label{fig:graphic_matroid}
\end{figure}
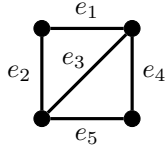

If $\mathcal{M}$ is a matroid on the ground set $S$ and $X \subseteq S$, then the \emph{deletion} of $X$ from $\mathcal{M}$ is the matroid $\mathcal{M} \setminus X \colonequals (S-X, \mathcal{I'})$, where $\mathcal{I}' \colonequals \{ Y \subseteq S-X \mid Y \in \mathcal{I} \}$.

The \emph{dual} of $\mathcal{M}$ is the matroid $\mathcal{M}^* = (S, \mathcal{I}^*)$ where $\mathcal{I}^*$ consists of those sets $X$ for which $S-X$ contains a basis of $\mathcal{M}$. A \emph{cut} in $\mathcal{M}$ is a set which forms a circuit in $\mathcal{M}^*$. A \textit{bridge} is a cut of size 1. Note that if $G$ is a connected graph, then the cuts in $\mathcal{M}(G)$ are exactly the edge cuts of $G$.

A \emph{coloring} of a matroid is an assignment of colors to the elements of its ground set, and a coloring is called \emph{proper} if the elements of any color form an independent set in the matroid (i.e., there are no monochromatic cycles). The \emph{chromatic number} of a loopless matroid $\mathcal{M}$, denoted by $\chi(\mathcal{M})$, is the minimum number of colors in a proper coloring of $\mathcal{M}$. Note that if $\mathcal{M} = \mathcal{M}(G)$ is a loopless graphic matroid, then $\chi(\mathcal{M})$ is the arboricity of $G$ (i.e.\ the minimum number of forests needed to cover all of the edges, and not the ``usual, graph theoretical'' chromatic number of $G$).\footnote{Some authors call the chromatic number of a matroid $\mathcal{M}$ as the \emph{covering number} and denote it by $\beta(\mathcal{M})$. Given two matroids $\mathcal{M}_1$ and $\mathcal{M}_2$ on the same ground set, determining the exact value of $\beta(\mathcal{M}_1 \cap \mathcal{M}_2)$ has been in the focus of research. Putting it in the context of courteous colorings, $\beta(\mathcal{M}_1 \cap \mathcal{M}_2)$ denotes the minimum number of colors needed to color the common ground set of $\mathcal{M}_1$ and $\mathcal{M}_2$ so that both $\mathcal{M}^*_1$ and $\mathcal{M}^*_2$ are courteously colored with respect to this coloring. It is not difficult to see that $\beta(\mathcal{M}_1 \cap \mathcal{M}_2) \ge \min \big\{ \beta(\mathcal{M}_1), \beta(\mathcal{M}_2) \big\}$, and Aharoni and Berger showed that $\beta(\mathcal{M}_1 \cap \mathcal{M}_2) \le 2 \max \big\{ \beta(\mathcal{M}_1), \beta(\mathcal{M}_2) \big\}$, see~\cite{article:AharoniBerger_1}. Aharoni et al.\ conjectured that $\beta(\mathcal{M}_1 \cap \mathcal{M}_2) = \max \big\{ \beta(\mathcal{M}_1), \beta(\mathcal{M}_2) \big\}$ if $\beta(\mathcal{M}_1) \ne \beta(\mathcal{M}_2)$, and $\beta(\mathcal{M}_1 \cap \mathcal{M}_2) \le \max \big\{ \beta(\mathcal{M}_1), \beta(\mathcal{M}_2) \big\} + 1$ if $\beta(\mathcal{M}_1) = \beta(\mathcal{M}_2)$, see~\cite{article:AharoniBerger_2}. B\'{e}rczi and Schwarcz recently proved that determining the exact value of $\beta(\mathcal{M}_1 \cap \mathcal{M}_2)$ is difficult under the rank oracle model~\cite{article:BercziSchwarcz}.}

Since the number of independent sets can be exponential in the size of the ground set, for matroid algorithms it is usually assumed that the matroid is accessed through an \emph{oracle}. An oracle can be thought of as a black box that, given an input set $X \subseteq S$, outputs certain properties of the set, e.g., whether $X$ is independent or not (\emph{independence oracle}), or the rank of $X$ (\emph{rank oracle}). Then the complexity of a matroid algorithm is measured by the number of oracle calls and other conventional elementary steps. There are various types of oracles having the same computational power in the sense that any of them can be simulated by using a polynomial number of calls to any of the others, measured in terms of the size of the ground set. For further details, we refer the interested reader to~\cites{robinson1980computational,hausmann1981algorithmic,coullard1996independence}. Here we assume that a matroid is accessed through an independence oracle that can determine whether a subset of the ground set is independent or not.

\subsection{Color-avoiding connectivity}

Now we are ready to define color-avoiding connectivity for edge- and vertex-colored graphs.

\begin{definition}
 Let $G$ be a graph% on at least two vertices
 , $C$ be a finite color set, $f \colon E(G) \to C$ be a function assigning colors to the edges, and $k, \ell$ be positive integers.
 The edge-colored graph $G$ is called \emph{edge-$\ell$-color-avoiding $k$-edge-connected} (or \emph{edge-$\ell$-color-avoiding $k$-vertex-connected}) if after the removal of the edges of any at most $\ell$ colors from~$G$, the remaining graph is $k$-edge-connected (or $k$-vertex-connected). If $G$ is an edge-$\ell$-color-avoiding $k$-edge-connected (or edge-$\ell$-color-avoiding $k$-vertex-connected) graph with respect to the edge-coloring $f$, then $f$ is called an \emph{edge-$\ell$-color-avoiding $k$-edge-connected} (or \emph{edge-$\ell$-color-avoiding $k$-vertex-connected}) \emph{coloring} of $G$.
\end{definition}

Note that an edge-colored graph % on at least two vertices
is edge-$\ell$-color-avoiding 1-edge-connected if and only if it is edge-$\ell$-color-avoiding 1-vertex-connected; such a graph is simply called \emph{edge-$\ell$-color-avoiding 1-connected}. 

On Figure \ref{fig:ecakecexample}, we can see some examples for the definition of edge-$\ell$-color-avoiding $k$-edge-connectivity.

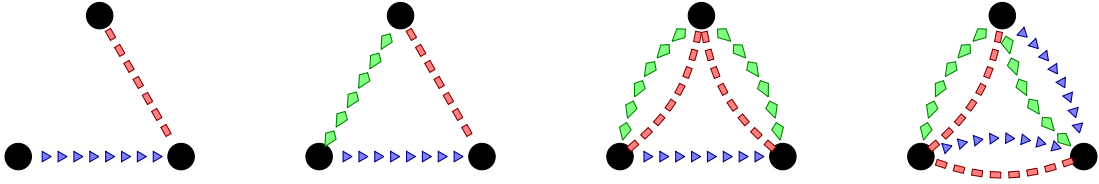
\begin{figure}[!ht]
\centering
\begin{tikzpicture}[scale=1]
 \tikzstyle{vertex}=[draw,circle,fill,minimum size=10,inner sep=0]
 \tikzset{paint/.style={draw=#1!50!black, fill=#1!50}, decorate with/.style = {decorate, decoration={shape backgrounds, shape=#1, shape size = 5pt, shape sep = 6pt}}}
 \tikzset{edge_red/.style={postaction={decorate, decoration={markings, mark=between positions 3pt and 1-2pt step 7pt with {\draw[red!50!black,thin,fill=red!50] (30:0.08) -- (150:0.08) -- (210:0.08) -- (330:0.08) -- (30:0.08);}}}}}
 \tikzset{edge_blue/.style={postaction={decorate, decoration={markings, mark=between positions 5pt and 1-2pt step 6pt with {\draw[blue!50!black,thin,fill=blue!50] (0:0.08) -- (120:0.08) -- (240:0.08) -- (0:0.08);}}}}}
 \tikzset{edge_green/.style={postaction={decorate, decoration={markings, mark=between positions 5pt and 1-2pt step 8.5pt with {\draw[green!50!black,thin,fill=green!50] (0:0.16) -- (90:0.08) -- (180:0.08) -- (270:0.08) -- (0:0.16);}}}}}

 \begin{scope}[shift={(-4,0)}]
 \node[vertex] (a1) at (90:1.25)  {};
 \node[vertex] (b1) at (210:1.25) {};
 \node[vertex] (c1) at (330:1.25) {};

 \path[edge_red] (a1) -- (c1);
 \path[edge_blue] (b1) -- (c1);
 \end{scope}
 
 \node[vertex] (a1) at (90:1.25)  {};
 \node[vertex] (b1) at (210:1.25) {};
 \node[vertex] (c1) at (330:1.25) {};
 
 \path[edge_green] (a1) -- (b1);
 \path[edge_red] (a1) -- (c1);
 \path[edge_blue] (b1) -- (c1);
  
 \begin{scope}[shift={(4,0)}]
 \node[vertex] (a1) at (90:1.25)  {};
 \node[vertex] (b1) at (210:1.25) {};
 \node[vertex] (c1) at (330:1.25) {};
 
 \path[edge_green, bend right=20] (a1) to (b1);
 \path[edge_red, bend left=20] (a1) to (b1);
 \path[edge_red, bend right=20] (a1) to (c1);
 \path[edge_green, bend left=20] (a1) to (c1);
 \path[edge_blue] (b1) -- (c1);
 \end{scope}

 \begin{scope}[shift={(8,0)}]
 \node[vertex] (a1) at (90:1.25)  {};
 \node[vertex] (b1) at (210:1.25) {};
 \node[vertex] (c1) at (330:1.25) {};

 \path[edge_green, bend right=20] (a1) to (b1);
 \path[edge_red, bend left=20] (a1) to (b1);
 \path[edge_green, bend right=20] (a1) to (c1);
 \path[edge_blue, bend left=20] (a1) to (c1);
 \path[edge_red, bend right=20] (b1) to (c1);
 \path[edge_blue, bend left=20] (b1) to (c1);
 \end{scope}
\end{tikzpicture}
\caption{Examples for edge-$\ell$-color-avoiding $k$-edge-connectivity. The first graph is not edge-1-color-avoiding 1-connected: after the removal of the red (denoted by rectangles) edges, the top vertex becomes isolated. The second graph is edge-1-color-avoiding 1-connected but not edge-1-color-avoiding 2-edge-connected: after the removal of edges of any single color, the graph remains connected, but not 2-edge-connected. The third graph is edge-1-color-avoiding 2-edge-connected but not edge-2-color-avoiding 1-connected: after the removal of the red (denoted by rectangles) and green (denoted by deltoids) edges, the top vertex becomes isolated. Finally, the fourth graph is edge-2-color-avoiding 1-connected: after the removal of any two colors, there remains a Hamiltonian path.}
\label{fig:ecakecexample}
\end{figure}

Using the classical theorems of Menger~\cite{article:menger}, we can give equivalent definitions for edge-1-color-avoiding 1-connectivity or, in general, for edge-$\ell$-color-avoiding $k$-edge- and $k$-vertex-connectivity as follows.

\begin{proposition} \label{prop:equiv_def_of_ECA}
 Let $G$ be a graph on at least two vertices, $C$ be a finite color set and $f \colon E(G) \to C$ be a function. Then the following are equivalent.
 
 \begin{enumerate}[label=(\roman*), topsep=2pt, itemsep=0pt]
  \item \label{item:i} The graph $G$ is edge-1-color-avoiding 1-connected.
  \item \label{item:ii} After the removal of the edges of any color, there remains a path between any two vertices.
  \item \label{item:iii} There exist no monochromatic edge cuts in $G$.
  \item \label{item:iv} There exist no inclusion-wise minimal monochromatic edge cuts in $G$.
  \item \label{item:v} \linkdest{item:v_of_prop}{} After the removal of the edges of any color, there remains a spanning tree in the graph.
 \end{enumerate}
\end{proposition}

Given a graph $G$, for an edge-coloring $f$ and a subset $E'$ of edges of $G$, let $f(E')$ denote the set of the colors of the edges in $E'$.

\begin{proposition} \label{prop:equiv_def_of_ECA_2}
 Let $G = (V,E)$ be a graph on at least two vertices, $C$ be a finite color set, and $f \colon E(G) \to C$ be an edge-coloring. Then the following are equivalent.
 
 \begin{enumerate}[label=(\roman*), topsep=2pt, itemsep=0pt]
  \item \label{item2:i} The graph $G$ is edge-$\ell$-color-avoiding $k$-edge-connected (or edge-$\ell$-color-avoiding $k$-vertex-connected).
  \item \label{item2:ii} \linkdest{item:ii_of_equiv_def}{} After the removal of the edges of any at most $\ell$ colors, there remains at least $k$ pairwise edge-disjoint (or at least $k$ pairwise internally vertex-disjoint) paths between any two vertices.
  \item \label{item2:iii} \linkdest{item:iii_of_equiv_def}{} There exists no edge cut $E' \cup E'' \subseteq E$ such that $|E'| \le k-1$ and $\big| f(E'') \big| \le \ell$. (Or there exists no mixed cut $V' \cup E' \cup E''$ with $V' \subseteq V$ and $E' \cup E'' \subseteq E$ such that $|V' \cup E'| \le k-1$ and $|f(E'')| \le \ell$.)
 \end{enumerate}
\end{proposition}

For ease of reference, let us introduce the following terms for paths not using edges of some specific colors. Let $G$ be such a graph on at least two vertices, whose edges are colored with a finite color set $C$, and let $k, \ell$ be positive integers. For a subset of colors $C' \subseteq C$, a path is called an \emph{edge-$C'$-avoiding path} if it does not contain any edges of any color in $C'$. The vertices $u$ and $v$ are called \emph{edge-$\ell$-color-avoiding $k$-edge-connected} (or \emph{edge-$\ell$-color-avoiding $k$-vertex-connected}) if there exist at least $k$ pairwise edge-disjoint edge-$C'$-avoiding (or $k$ pairwise internally vertex-disjoint edge-$C'$-avoiding) $u$-$v$ paths for any subset $C' \subseteq C$ with $|C'| \le \ell$. Then $G$ is edge-$\ell$-color-avoiding $k$-edge-connected (or edge-$\ell$-color-avoiding $k$-vertex-connected) if and only if %it has at least two vertices and any two of them
any two of its vertices are edge-$\ell$-color-avoiding $k$-edge-connected (or edge-$\ell$-color-avoiding $k$-vertex-connected).

The equivalent definition of edge-1-color-avoiding 1-connectivity presented in Proposition~\hyperlink{item:v_of_prop}{\ref*{prop:equiv_def_of_ECA}.\ref*{item:v}} can be extended to matroids as well: in a previous article~\cite{article:CA_spanning_subgraphs}, we called a matroid whose ground set is colored, courteously colored if after the removal of the elements of any single color, at least one basis of the matroid remains intact, i.e, the rank of the matroid does not decrease. Now we extend this definition for the case when more colors can be removed simultaneously.

\begin{definition}
 Let $\mathcal{M} = (S, \mathcal{I})$ be a matroid whose ground set is colored, and let $\ell \in \mathbb{Z}_+$. We say that this coloring of $\mathcal{M}$ is \emph{$\ell$-courteous} if after the deletion of the elements of any at most $\ell$ colors, the rank of the matroid does not change.
\end{definition}

In particular, a graphic matroid is $\ell$-courteously colored if and only if each connected component of the corresponding graph is edge-$\ell$-color-avoiding connected. 

Since a set forms a cut in a matroid if and only if it is minimal for the property that its removal decreases the rank of the matroid~\cite{book:recski}, Proposition~\ref{prop:equiv_def_of_ECA_2} directly implies the following.

\begin{proposition} \label{prop:equiv_def_of_courteous}
 Let $\mathcal{M}$ be a matroid whose ground set is colored, and let $\ell \in \mathbb{Z}_+$. Then the following are equivalent.
 
 \begin{enumerate}[label=(\roman*), topsep=2pt, itemsep=0pt]
  \item \label{item3:i} The matroid $\mathcal{M}$ is $\ell$-courteously colored.
  \item \label{item3:ii} \linkdest{item3:ii_of_prop}{} The matroid $\mathcal{M}$ does not contain a cut consisting of elements colored with at most $\ell$ colors.
  \item \label{item3:iii} \linkdest{item3:iii_of_prop}{} The matroid $\mathcal{M^*}$ does not contain a cycle consisting of elements colored with at most $\ell$ colors.
 \end{enumerate}
\end{proposition}

Both of the equivalent definitions of edge-$\ell$-color-avoiding $k$-edge-connectivity (or $k$-vertex-connectivity) presented in Proposition~\hyperlink{item:ii_of_equiv_def}{\ref*{prop:equiv_def_of_ECA_2}.\ref*{item2:ii}} and ~\hyperlink{item:iii_of_equiv_def}{\ref*{prop:equiv_def_of_ECA_2}.\ref*{item2:iii}} can be naturally modified for vertex-colored graphs, however, unlike for edge-colored graphs, these two concepts for vertex-colored graphs do not coincide. 

\begin{definition}
 Let $G$ be a graph, $C$ be a finite color set, $f \colon V(G) \to C$ be a function assigning colors to the vertices, and $k, \ell$ be positive integers. The vertex-colored graph $G$ is called \emph{vertex-$\ell$-color-avoiding $k$-edge-connected} (or \emph{vertex-$\ell$-color-avoiding $k$-vertex-connected}) if after the removal of the vertices of any at most $\ell$ colors from $G$, the remaining graph is either $k$-edge-connected (or $k$-vertex-connected) or it has at most one vertex. If $G$ is a vertex-$\ell$-color-avoiding $k$-edge-connected (or vertex-$\ell$-color-avoiding $k$-vertex-connected) graph with respect to the vertex-coloring $f$, then $f$ is called a \emph{vertex-$\ell$-color-avoiding $k$-edge-connected} (or \emph{vertex-$\ell$-color-avoiding $k$-vertex-connected}) \emph{coloring}.
\end{definition}

Given a graph $G$, for a vertex-coloring $f$ and a subset $V'$ of vertices of $G$, let $f(V')$ denote the set of the colors of the vertices in $V'$.

\begin{definition}
 Let $G$ be a graph, $C$ be a finite color set, $f \colon V(G) \to C$ be a function assigning colors to the vertices, and $k, \ell$ be positive integers. The vertex-colored graph $G$ is called \emph{internally vertex-$\ell$-color-avoiding $k$-edge-connected} if there exists no mixed cut $V' \cup E'$ with $V' \subseteq V$ and $E' \subseteq E$ such that $|E'| \le k-1$ and $\big| f(V') \big| \le \ell$. The vertex-colored graph $G$ is called \emph{internally vertex-$\ell$-color-avoiding $k$-vertex-connected} if there exists no mixed cut $V' \cup V'' \cup E'$ with $V' \cup V'' \subseteq V$ and $E' \subseteq E$ such that $|V'| + |E'| \le k-1$ and $\big| f(V'') \big| \le \ell$. If $G$ is an internally vertex-$\ell$-color-avoiding $k$-edge-connected (or internally vertex-$\ell$-color-avoiding $k$-vertex-connected) graph with respect to the vertex-coloring $f$, then $f$ is called an \emph{internally vertex-$\ell$-color-avoiding $k$-edge-connected} (or \emph{internally vertex-$\ell$-color-avoiding $k$-vertex-connected}) \emph{coloring}.
\end{definition}

Note that a vertex-colored graph is vertex-$\ell$-color-avoiding 1-edge-connected if and only if it is vertex-$\ell$-color-avoiding 1-vertex-connected; such a graph is simply called \emph{vertex-$\ell$-color-avoiding 1-connected}. Similarly, a vertex-colored graph is internally vertex-$\ell$-color-avoiding 1-edge-connected if and only if it is internally vertex-$\ell$-color-avoiding 1-vertex-connected; such a graph is simply called \emph{internally vertex-$\ell$-color-avoiding 1-connected}. 

For a summary of the different versions of the color-avoiding connectivity of a graph, see Table~\ref{table:summary_def_graph}.

\begin{table}[!ht] 
\begin{center}
 \setlength{\tabcolsep}{5pt}
 \renewcommand{\arraystretch}{1.2}
 \captionsetup{type=table}
 \begin{tabular}{!{\vrule width 1.25pt} c !{\vrule width 1.25pt} c !{\vrule width 1.25pt} c !{\vrule width 1.25pt}} \specialrule{1.25pt}{0pt}{0pt}
   & edge-$\ell$-color-avoiding & vertex-$\ell$-color-avoiding \\ \specialrule{1.25pt}{0pt}{0pt}
   \multirow{5}{*}{\parbox{100pt}{\centering $k$-edge-connectivity}} & \multirow{5}{*}{\parbox{140pt}{\centering after the removal of the edges of any at most $\ell$ colors, the remaining graph is $k$-edge-connected}} & \multirow{5}{*}{\parbox{140pt}{\centering after the removal of the vertices of any at most $\ell$ colors, the remaining graph is either $k$-edge-connected or has at most one vertex}} \\
   & & \\
   & & \\
   & & \\
   & & \\ \hline
   \multirow{5}{*}{\parbox{100pt}{\centering $k$-vertex-connectivity}} & \multirow{5}{*}{\parbox{140pt}{\centering after the removal of the edges of any at most $\ell$ colors, the remaining graph is $k$-vertex-connected}} & \multirow{5}{*}{\parbox{140pt}{\centering after the removal of the vertices of any at most $\ell$ colors, the remaining graph is either $k$-vertex-connected or has at most one vertex}} \\
   & & \\
   & & \\
   & & \\
   & & \\ \specialrule{1.25pt}{0pt}{0pt}
   \multicolumn{3}{c}{} \\ \specialrule{1.25pt}{0pt}{0pt}
   & \multirow{2}{*}{\parbox{100pt}{\centering edge-$\ell$-color-avoiding}} & \multirow{2}{*}{\parbox{100pt}{\centering internally \\ vertex-$\ell$-color-avoiding}}
    \\
    & & \\ \specialrule{1.25pt}{0pt}{0pt}
   \multirow{4}{*}{\parbox{100pt}{\centering $k$-edge-connectivity}} & \multirow{4}{*}{\parbox{140pt}{\centering no subset of the edges of any at most $\ell$ colors along with at most $k-1$ additional edges form an edge cut}} & \multirow{4}{*}{\parbox{140pt}{\centering no subset of the vertices of any at most $\ell$ colors along with at most $k-1$ additional edges form a mixed cut}} \\
   & & \\
   & & \\
   & & \\ \hline
   \multirow{4}{*}{\parbox{100pt}{\centering $k$-vertex-connectivity}} & \multirow{4}{*}{\parbox{140pt}{\centering no subset of the edges of any at most $\ell$ colors along with at most $k-1$ additional edges and vertices form a mixed cut}} & \multirow{4}{*}{\parbox{140pt}{\centering no subset of the vertices of any at most $\ell$ colors along with at most $k-1$ additional edges and vertices form a mixed cut}} \\
   & & \\
   & & \\
   & & \\ \specialrule{1.25pt}{0pt}{0pt}
 \end{tabular}
 \captionof{table}{Summary of the definitions of the color-avoiding connectivity of a graph presented in Section~\ref{section:preliminaries}.}
 \label{table:summary_def_graph}
\end{center}

\vspace{-13pt}
\end{table}

Similarly as in the case of edge-colored graphs, we introduce the following notions. Let $G$ be a graph on at least two vertices which are colored with a finite color set $C$, and let $k, \ell$ be positive integers. For a subset of colors $C' \subseteq C$, a path is called an \emph{internally vertex-$C'$-avoiding path} if it does not contain any internal vertices of any color in $C'$. The vertices $u$ and $v$ are called \emph{internally vertex-$C'$-avoiding $k$-edge-connected} (or \emph{internally vertex-$C'$-avoiding $k$-vertex-connected}) if there exist at least $k$ pairwise edge-disjoint internally vertex-$C'$-avoiding (or $k$ pairwise vertex-disjoint internally vertex-$C'$-avoiding) $u$-$v$ paths. The vertices $u$ and $v$ are called \emph{internally vertex-$\ell$-color-avoiding $k$-edge-connected} (or \emph{internally vertex-$\ell$-color-avoiding $k$-vertex-connected}) if they are internally vertex-$C'$-avoiding $k$-edge-connected (or internally vertex-$C'$-avoiding $k$-vertex-connected) for any subset $C' \subseteq C$ with $|C'| \le \ell$. Then $G$ is internally vertex-$\ell$-color-avoiding $k$-edge-connected (or internally vertex-$\ell$-color-avoiding $k$-vertex-connected) if and only if any two vertices of $G$ are internally vertex-$\ell$-color-avoiding $k$-edge-connected (or internally vertex-$\ell$-color-avoiding $k$-vertex-connected).

The vertices $u$ and $v$ are called \emph{vertex-$C'$-avoiding $k$-edge-connected} (or \emph{vertex-$C'$-avoiding $k$-vertex-connected}) if
\begin{itemize}[topsep=2pt, itemsep=-2pt]
 \item[--] either their colors are not in $C'$ and they are internally vertex-$C'$-avoiding $k$-edge-connected (or internally vertex-$C'$-avoiding $k$-vertex-connected),
 \item[--] or their color is in $C'$ and there exist at least $k$ pairwise edge-disjoint (or vertex-disjoint) $u$-$v$ paths.
\end{itemize}
The vertices $u$ and $v$ are called \emph{vertex-$\ell$-color-avoiding $k$-edge-connected} (or \emph{vertex-$\ell$-color-avoiding $k$-vertex-connected}) if for any subset $C' \subseteq C$ with $|C'| \le \ell$ they are vertex-$C'$-avoiding $k$-edge-connected (or vertex-$C'$-avoiding $k$-vertex-connected). Then $G$ is vertex-$\ell$-color-avoiding $k$-edge-connected (or vertex-$\ell$-color-avoiding $k$-vertex-connected) if and only if any two vertices of $G$ are vertex-$\ell$-color-avoiding $k$-edge-connected (or vertex-$\ell$-color-avoiding $k$-vertex-connected).

Now observe that if a graph is internally vertex-$\ell$-color-avoiding $k$-edge-connected, then it is also vertex-$\ell$-color-avoiding $k$-edge-connected, but not every vertex-$\ell$-color-avoiding $k$-edge-connected graph is internally vertex-$\ell$-color-avoiding $k$-edge-connected, and the same holds for $k$-vertex-connectivity as well; for an example, see Figure~\ref{figure:example_vertex_definitions_difference}.

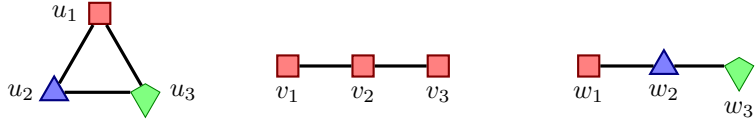
\begin{figure}[!ht]
\centering
\begin{tikzpicture}[scale=1]
 \tikzstyle{vertex_red}=[draw, thick, red!50!black, shape=rectangle, fill=red!50, minimum size=8pt, inner sep=0]
  \tikzstyle{vertex_blue}=[draw, thick, blue!50!black, shape=regular polygon, regular polygon sides=3, fill=blue!50, minimum size=12pt, inner sep=0]
  \tikzstyle{vertex_green}=[draw, green!50!black, kite, fill=green!50, minimum size=11pt, inner sep=0.1pt]

  \node[vertex_red] (u1) at (90:0.7) [label=left: $u_1$] {};
  \node[vertex_blue] (u2) at (210:0.7) [label=left: $u_2$] {};
  \node[vertex_green] (u3) at (330:0.7) [label=right: $u_3$] {};

  \draw[very thick] (u1) -- (u2) -- (u3) -- (u1);
  \begin{scope}[shift={(2.5,0)}]  
  \node[vertex_red] (v1) at (0,0) [label=below: $v_1$] {};
  \node[vertex_red] (v2) at (1,0) [label=below: $v_2$] {};
  \node[vertex_red] (v3) at (2,0) [label=below: $v_3$] {};
  
  \draw[very thick] (v1) -- (v2) -- (v3);
  
  \node[vertex_red] (w1) at (4,0) [label=below: $w_1$] {};
  \node[vertex_blue] (w2) at (5,0) [label=below: $w_2$] {};
  \node[vertex_green] (w3) at (6,0) [label=below: $w_3$] {};
  
  \draw[very thick] (w1) -- (w2) -- (w3);
  \end{scope}
 \end{tikzpicture}
 \caption{Examples for vertex- and internally vertex-$\ell$-color-avoiding $k$-edge-connectivity. A cycle whose vertices are all of distinct colors is both vertex-$\ell$-color-avoiding 1-connected and internally vertex-$\ell$-color-avoiding 1-connected for any $\ell \in \mathbb{Z}_+$. A path on three red vertices (denoted by squares) is vertex-$\ell$-color-avoiding 1-connected, but it is not internally vertex-$\ell$-color-avoiding 1-connected for any $\ell \in \mathbb{Z}_+$ since there exists no $v_1$-$v_3$ path whose internal vertices are not red. On the other hand, if the vertices of the path all have distinct colors, then the graph is neither vertex-1-color-avoiding nor internally vertex-1-color-avoiding 1-connected: after the removal of the blue vertices (denoted by triangles), $w_1$ and $w_3$ are disconnected.}
 \label{figure:example_vertex_definitions_difference}
\end{figure}

Later, we make use of the following straightforward proposition.

\begin{proposition}
 \label{proposition:equiv_def_of_IVCA}
 Let $G$ be a vertex-colored graph on at least two vertices, $C$ be a finite color set and $f \colon V(G) \to C$ be a function. Then the following are equivalent.
 
 \begin{enumerate}[label=(\roman*), topsep=1pt, itemsep=-0.5pt]
  \item The graph $G$ is internally vertex-1-color-avoiding 1-connected.
  \item There exist no monochromatic vertex cuts (including the empty set) in $G$.
  \item \label{item:iii_of_IVCA} \linkdest{item:iii_of_prop_IVCA}{} There exist no inclusion-wise minimal monochromatic vertex cuts (including the empty set) in $G$.
 \end{enumerate}
\end{proposition}

The notion of color-avoiding connectivity can be further generalized by considering a function $f$ assigning a subset of colors to each edge or vertex of the graph. This makes the edges or the vertices of a graph more sensitive or immune against the ``color-attacks'' since an edge or a vertex is destroyed whenever one of the colors assigned to it is attacked; the empty subset of colors means that corresponding edge or vertex cannot be removed.

\begin{remark} \label{remark:spec_colorings}
 Let $G$ be a graph and let $k, \ell \in \mathbb{Z}_+$.
 
 \begin{enumerate}[label=\arabic*., topsep=2pt, itemsep=0pt]
  \item If each edge or each vertex is assigned the empty set of colors, then the edge-, vertex- and internally vertex-$\ell$-color-avoiding $k$-edge-connectivity (or $k$-vertex-connectivity) all give back the definition of the ``usual'' $k$-edge-connectivity (or $k$-vertex-connectivity).
  
  \item If each vertex is colored with the same color, then the vertex-$\ell$-color-avoiding $k$-edge-connectiv\-ity (or $k$-vertex-connectivity) is the same as the ``usual'' $k$-edge-connectivity (or $k$-vertex-connectivity).
  
  \item If each edge is assigned a unique color, then the edge-$\ell$-color-avoiding $k$-edge-connectivity is the same as the ``usual'' $(k + \ell)$-edge-connectivity.
  
  \item If each vertex is assigned a unique color, then the vertex- and the internally vertex-$\ell$-color-avoiding $k$-vertex-connectivity is the same as the ``usual'' $(k+\ell)$-vertex-connectivity.

 \end{enumerate}
\end{remark}

We can naturally extend the definition of color-avoiding connectivity to digraphs by requiring the existence of $u$-$v$ and $v$-$u$ dipaths with the same properties as those in the undirected case, and we call this notion color-avoiding strong connectivity. For a summary of the different versions of the color-avoiding connectivity of two vertices, see Table~\ref{table:summary_def_two_vertices}. By requiring the existence of $r$-$v$ dipaths, having the previously mentioned properties, from a given vertex $r$ to any other vertex $v$, we can obtain the notion of color-avoiding $r$-rooted connectivity as well.

% Itt kellene hogy legyen a Table~\ref{table:summary_def_two_vertices}.

Replacing each edge by two arcs of the same color directed in opposite directions in Figures~\ref{fig:ecakecexample} and~\ref{figure:example_vertex_definitions_difference}, we clearly get a not arc-1-color-avoiding strongly 1-connected, an arc-1-color-avoiding strongly 1-connected but not arc-1-color-avoiding strongly 2-arc-connected, an arc-1-color-avoiding strongly 2-arc-connected but not arc-2-color-avoiding strongly 1-connected, an arc-2-color-avoiding strongly 1-connected, an internally vertex-$\ell$-color-avoiding strongly 1-connected, a vertex-$\ell$-color-avoiding strongly 1-connected but not internally vertex-$\ell$-color-avoiding strongly 1-connected, and a not vertex-1-color-avoiding strongly 1-connected graph, respectively (where $\ell \in \mathbb{Z}_+$). For some examples on color-avoiding rooted connectivity, see Figure~\ref{fig:rootedexample}.

% Itt kellene hogy legyen a Figure~\ref{fig:rootedexample}.

The first statement of the following lemma was already used in a previous work of ours~\cite{article:CA_spanning_subgraphs} for undirected graphs, and it can be naturally extended to directed graphs.

\begin{lemma} \label{lemma:differently_colored_neighbor} 
 $\empty$
 
 \begin{enumerate}[label=(\arabic*), topsep=2pt, itemsep=0pt]
  \item Let $G$ be a vertex-colored graph in which every vertex $v$ has neighbor whose color is different from that of $v$. Then $G$ is vertex-1-color-avoiding 1-connected if and only if it is internally vertex-1-color-avoiding 1-connected.
  \item \label{item:ii_of_differently_colored_neighbor} \linkdest{item:ii_of_lemma_differently_colored_neighbor}{} Let $D$ be a vertex-colored digraph in which every vertex $v$ has an in- and an out-neighbor whose colors are different from that of $v$. Then $D$ is vertex-1-color-avoiding strongly 1-connected if and only if it is internally vertex-1-color-avoiding strongly 1-connected.
  \item \label{item:iii_of_differently_colored_neighbor} \linkdest{item:iii_of_lemma_differently_colored_neighbor}{} Let $D$ be such a vertex-colored digraph with a distinguished vertex $r$, in which every vertex $v \ne r$ has an in-neighbor whose color is different from that of $v$, and $r$ has a unique color. Then $D$ is vertex-1-color-avoiding $r$-rooted 1-connected if and only if it is internally vertex-1-color-avoiding $r$-rooted 1-connected.
 \end{enumerate}
\end{lemma}

\begin{table}[H] 
{\small
\begin{center}
 \setlength{\tabcolsep}{2.75pt}
 \renewcommand{\arraystretch}{1.1}
 \captionsetup{type=table}
 \begin{tabular}{!{\vrule width 1.25pt} c !{\vrule width 1.25pt} c|c|c !{\vrule width 1.25pt}} \specialrule{1.25pt}{0pt}{0pt}
   & \multicolumn{2}{c}{If after the removal of} & \\
   & the edges of & \multicolumn{2}{c !{\vrule width 1.25pt}}{\multirow{2}{*}{the vertices of any at most $\ell$ colors}} \\
   & any at most $\ell$ colors, \ldots & \multicolumn{2}{c !{\vrule width 1.25pt}}{} \\
   &  & $u$ or $v$ is removed, or \ldots & excluding $u$ and $v$, \ldots \\ \specialrule{1.25pt}{0pt}{0pt}
   \ldots there exist at least $k$ pairwise & \multirow{3}{*}{\parbox{0.2\textwidth}{\centering edge-$\ell$-color-avoiding \\ $k$-edge-connected.}} & \multirow{3}{*}{\parbox{0.2\textwidth}{\centering vertex-$\ell$-color-avoiding \\ $k$-edge-connected.}} & \multirow{3}{*}{\parbox{0.2\textwidth}{\centering internally vertex-$\ell$-color-avoiding \\ $k$-edge-connected.}} \\
   edge-disjoint $u$-$v$ path, & & & \\
   then $u$ and $v$ are & & & \\ \hline
   \ldots there exist at least $k$ pairwise  & \multirow{3}{*}{\parbox{0.2\textwidth}{\centering edge-$\ell$-color-avoiding \\ $k$-vertex-connected.}}  & \multirow{3}{*}{\parbox{0.2\textwidth}{\centering vertex-$\ell$-color-avoiding \\ $k$-vertex-connected.}} & \multirow{3}{*}{\parbox{0.2\textwidth}{\centering internally vertex-$\ell$-color-avoiding \\ $k$-vertex-connected.}}\\
   internally vertex-disjoint $u$-$v$ path, & & & \\
   then $u$ and $v$ are & & & \\ \hline
   \ldots there exist at least $k$ pairwise  & \multirow{5}{*}{\parbox{0.2\textwidth}{\centering arc-$\ell$-color-avoiding \\ strongly $k$-arc-connected.}} & \multirow{5}{*}{\parbox{0.2\textwidth}{\centering vertex-$\ell$-color-avoiding \\ strongly $k$-arc-connected.}}& \multirow{5}{*}{\parbox{0.2\textwidth}{\centering internally vertex-$\ell$-color-avoiding \\ strongly $k$-arc-connected.}} \\
   arc-disjoint $u$-$v$ dipath, and & & & \\
   at least $k$ pairwise & & & \\
   arc-disjoint $v$-$u$ dipath, & & & \\
   then $u$ and $v$ are & & &  \\ \hline
   \ldots there exist at least $k$ pairwise  & \multirow{5}{*}{\parbox{0.2\textwidth}{\centering arc-$\ell$-color-avoiding \\ strongly $k$-vertex-connected.}} & \multirow{5}{*}{\parbox{0.2\textwidth}{\centering vertex-$\ell$-color-avoiding \\ strongly $k$-vertex-connected.}} & \multirow{5}{*}{\parbox{0.2\textwidth}{\centering internally vertex-$\ell$-color-avoiding \\ strongly $k$-vertex-connected.}} \\
   internally vertex-disjoint $u$-$v$ dipath, & & & \\
   and at least $k$ pairwise & & & \\
   internally vertex-disjoint $v$-$u$ dipath, & & & \\
   then $u$ and $v$ are & & & \\ \specialrule{1.25pt}{0pt}{0pt}
 \end{tabular}
 \captionof{table}{Summary of the definitions of the color-avoiding connectivity of two vertices $u$ and $v$ presented in Section~\ref{section:preliminaries}.}
 \label{table:summary_def_two_vertices}
\end{center}
}

\vspace{-13pt}
\end{table}

\begin{figure}[!hb]
\centering
\begin{tikzpicture}[scale=0.9]
 \tikzstyle{vertex}=[draw,circle,fill,minimum size=10,inner sep=0]

 \tikzset{paint/.style={draw=#1!50!black, fill=#1!50}, decorate with/.style = {decorate, decoration={shape backgrounds, shape=#1, shape size = 5pt, shape sep = 6pt}}}
 \tikzset{edge_red/.style={postaction={decorate, decoration={markings, mark=between positions 3pt and 1-2pt step 7pt with {\draw[red!50!black,thin,fill=red!50] (30:0.08) -- (150:0.08) -- (210:0.08) -- (330:0.08) -- (30:0.08);}}}}}
 \tikzset{edge_blue/.style={postaction={decorate, decoration={markings, mark=between positions 5pt and 1-2pt step 6pt with {\draw[blue!50!black,thin,fill=blue!50] (0:0.08) -- (120:0.08) -- (240:0.08) -- (0:0.08);}}}}}
 \tikzset{edge_green/.style={postaction={decorate, decoration={markings, mark=between positions 5pt and 1-2pt step 8.5pt with {\draw[green!50!black,thin,fill=green!50] (0:0.16) -- (90:0.08) -- (180:0.08) -- (270:0.08) -- (0:0.16);}}}}}
 
 \tikzstyle{vertex_red}=[draw, thick, red!50!black, shape=rectangle, fill=red!50, minimum size=8pt, inner sep=0]
  \tikzstyle{vertex_blue}=[draw, thick, blue!50!black, shape=regular polygon, regular polygon sides=3, fill=blue!50, minimum size=12pt, inner sep=0]
  \tikzstyle{vertex_green}=[draw, green!50!black, kite, fill=green!50, minimum size=11pt, inner sep=0.1pt]

 \begin{scope}[scale=0.7]
 \node[vertex] (r) at (0,0) [label={left:$r$}] {};
 \node[vertex] (a) at (2.5,1.5) {};
 \node[vertex] (b) at (2.5,-1.5) {};
 \node[vertex] (c) at (5,0) {};
 
 \path[edge_green] (r) -- ($(r)!0.8!(a)$);
 \draw[-{Latex[black,length=10pt, width=10pt]}] ($(r)!0.9!(a)$) -- (a);
 \path[edge_red] (r) -- ($(r)!0.8!(b)$);
 \draw[-{Latex[black,length=10pt, width=10pt]}] ($(r)!0.9!(b)$) -- (b);
 \path[edge_red] (a) -- ($(a)!0.8!(c)$);
 \draw[-{Latex[black,length=10pt, width=10pt]}] ($(a)!0.9!(c)$) -- (c);
 \path[edge_blue] (b) -- ($(b)!0.8!(a)$);
 \draw[-{Latex[black,length=10pt, width=10pt]}] ($(b)!0.9!(a)$) -- (a);
 \path[edge_green] (c) -- ($(c)!0.8!(b)$);
 \draw[-{Latex[black,length=10pt, width=10pt]}] ($(c)!0.9!(b)$) -- (b);
 \end{scope}
 
 \begin{scope}[scale=0.75]
 \clip (-1,2.5) -- (5,2.5) -- (5,0.6) -- (-1,-1) -- cycle;
 \path[edge_blue, bend left=75, looseness=1.5] (r) to (c);
 \end{scope}
 \begin{scope}[scale=0.75,shift={(-0.05,0.15)}]
 \clip (4.9,0.55) -- (4.5,0.48) -- (5,0) -- (5.2,0.65) -- cycle;
 \draw[-{Latex[black,length=10pt, width=10pt]}, bend left=75, looseness=1.5] (r) to (c);
 \end{scope}
 
 \begin{scope}[shift={(10,0)}]
 \node[vertex_red] (v1) at (0,0) [label=below: $r$] {};
 \node[vertex_red] (v2) at (1.5,0) [label=below: $v$] {};
 \node[vertex_red] (v3) at (3,0) [label=below: $w$] {};
  
 \draw[very thick, -{Latex[black,length=7pt, width=7pt]}] (v1) -- (v2);
 \draw[very thick, -{Latex[black,length=7pt, width=7pt]}] (v2) -- (v3);
 \end{scope}
\end{tikzpicture}
\caption{Examples for arc-$\ell$-color-avoiding $r$-rooted $k$-arc-connectivity. The first digraph is arc-1-color-avoiding $r$-rooted 1-connected but not arc-1-color-avoiding $r$-rooted 2-arc-connected: after the removal of arcs of any single color, the digraph remains $r$-rooted connected, but not $r$-rooted 2-arc-connected. The second graph is vertex-1-color-avoiding $r$-rooted 1-connected but not internally vertex-1-color-avoiding $r$-rooted connected: there exist no internally vertex-red-avoiding $r$-$w$ dipath (the color red is denoted by squares).}
\label{fig:rootedexample}
\end{figure}
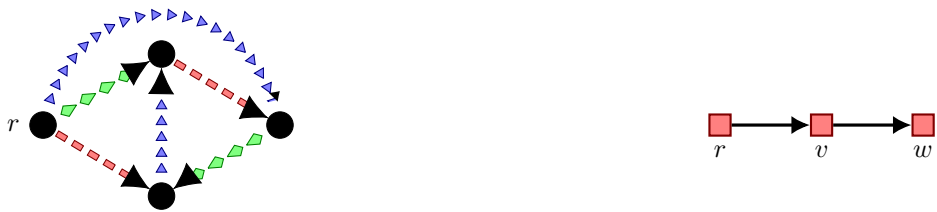

\vspace{-10pt}

\section{Color-avoiding connected colorings} \label{section:colorings}

In this section, we study the problem of finding a color-avoiding connected coloring (if exists) of an uncolored (directed or undirected) graph with minimum number of colors.

\subsection{Edge-color-avoiding connected colorings}

First, let us consider edge-colorings.

\subsubsection{On the existence of edge-color-avoiding connected colorings}

In the following, we give a necessary and sufficient condition for the existence of edge-color-avoiding connected colorings of a graph, and for the existence of arc-color-avoiding strongly or rooted connected colorings of a digraph. Let us start with the case of edge-color-avoiding $k$-edge-connectivity.

\begin{proposition} \label{edge-prop-1}
 Let $k, \ell \in \mathbb{Z}_+$.
 
 \begin{enumerate}[label=(\arabic*), topsep=2pt, itemsep=0pt]
  \item \label{item1:edge-prop-1} \linkdest{linkdest:item1:edge-prop-1}{} A graph $G$ has an edge-$\ell$-color-avoiding $k$-edge-connected coloring if and only if $G$ is $(k+\ell)$-edge-connected.
 
  \item \label{item2:edge-prop-1} A digraph $D$ has an arc-$\ell$-color-avoiding strongly $k$-arc-connected coloring if and only if $D$ is strongly $(k+\ell)$-arc-connected.

  \item \label{item3:edge-prop-1} \linkdest{linkdest:item3:edge-prop-1}{} A digraph $D$ with a given vertex $r \in V(D)$ has an arc-$\ell$-color-avoiding $r$-rooted $k$-arc-connected coloring if and only if $D$ is $r$-rooted $(k+\ell)$-arc-connected.
 \end{enumerate}
\end{proposition}
\begin{proof}
 We only prove statement~\ref{item1:edge-prop-1} since \ref{item2:edge-prop-1} and \ref{item3:edge-prop-1} can be proved analogously.

 \medskip

 First, assume that $G$ is not $(k+\ell)$-edge-connected, i.e., there exists a set of edges $E' \subseteq E$ with $|E'| \le k+\ell-1$ for which $G-E'$ is not connected. Consider an arbitrary edge-coloring of $G$, and let $C'$ denote the set of colors of the edges in $E'$. Let $C'' \subseteq C'$ be an arbitrary subset with $|C''| = \min \big( \ell, |C'| \big)$. Observe that there are at most $k-1$ edges in $E'$ that have different colors from those in $C''$. Thus the removal of all the edges of color $C''$ (which means the removal of at most $\ell$ colors) and the removal of the remaining edges in $E'$ (which means the removal of at most $k-1$ further edges) disconnects $G$, thus $G$ is not edge-$\ell$-color-avoiding $k$-edge-connected.

 \smallskip
 
 Now assume that $G$ is $(k+\ell)$-edge-connected. Then assigning a unique color to each edge results in an edge-$\ell$-color-avoiding $k$-edge-connected graph since in this case the removal of the edges of at most $\ell$ colors removes at most $\ell$ edges.
\end{proof}

The next proposition is the analogue of Proposition~\ref{edge-prop-1} for $k$-vertex-connectivity.

\begin{proposition} \label{edge-prop-2}
  Let $k, \ell \in \mathbb{Z}_+$.
 
 \begin{enumerate}[label=(\arabic*), topsep=2pt, itemsep=0pt]
  \item \label{item1:edge-prop-2} \linkdest{linkdest:item1:edge-prop-2}{} A graph $G = (V,E)$ has an edge-$\ell$-color-avoiding $k$-vertex-connected coloring if and only if $G - E'$ is $k$-vertex-connected for any $E' \subseteq E$ with $|E'| \le \ell$.
 
  \item \label{item2:edge-prop-2} A digraph $D=(V,E)$ has an arc-$\ell$-color-avoiding strongly $k$-vertex-connected coloring if and only if $D - E'$ is strongly $k$-vertex-connected for any $E' \subseteq E$ with $|E'| \le \ell$.

  \item \label{item3:edge-prop-2} A digraph $D=(V,E)$ with a given vertex $r \in V$ has an arc-$\ell$-color-avoiding $r$-rooted $k$-vertex-connected coloring if and only if $D - E'$ is $r$-rooted $k$-vertex-connected for $E' \subseteq E$ with $|E'| \le \ell$.
 \end{enumerate}
\end{proposition}
\begin{proof}
 We only prove statement~\ref{item1:edge-prop-2} since \ref{item2:edge-prop-2} and \ref{item3:edge-prop-2} can be proved analogously.

 \medskip
 
 First, assume that there exists a set of edges $E' \subseteq E$ with $|E'| \le \ell$ for which $G-E'$ is not $k$-vertex-connected, i.e., there exists a mixed cut $V'' \cup E''$ in $G-E'$ such that $|V''| + |E''| \le k-1$. Consider an arbitrary edge-coloring of~$G$, and let $C'$ denote the set of the colors of the edges in $E'$. Then $|C'| \le \ell$ holds, and the removal of all the edges with colors in $C'$ and the removal of the vertices and edges of $V'' \cup E''$ disconnects~$G$, thus $G$ is not edge-$\ell$-color-avoiding $k$-vertex-connected.

 \smallskip
 
 Now assume that $G - E'$ is $k$-vertex-connected for any $E' \subseteq E$ with $|E'| \le \ell$. Then assigning a unique color to each edge results in an edge-$\ell$-color-avoiding $k$-vertex-connected graph.
\end{proof}

As we mentioned earlier, a set forms a cut in a matroid if and only if it is minimal for the property that its removal decreases the rank of the matroid. Using this observation, we can easily generalize the common case $k=1$ of Propositions~\hyperlink{linkdest:item1:edge-prop-1}{\ref*{edge-prop-1}.\ref*{item1:edge-prop-1}} and~\hyperlink{linkdest:item1:edge-prop-2}{\ref*{edge-prop-2}.\ref*{item1:edge-prop-2}} to matroids.

\begin{proposition} \label{prop:colorability_of_matroids}
Let $\ell \in \mathbb{Z}_+$. A matroid $\mathcal{M}$ has an $\ell$-courteous coloring if and only if $\mathcal{M}$ does not contain a cut of size at most $\ell$. 
\end{proposition}
\begin{proof}
First, assume that $\mathcal{M} = (S, \mathcal{I})$ contains a cut $Q$ of size at most $\ell$. Suppose to the contrary that $\mathcal{M}$ has an $\ell$-courteous coloring. Then in this coloring, $Q$ is a cut consisting of elements colored with at most $\ell$ colors. Thus, by Proposition~\ref{prop:equiv_def_of_courteous}, we obtain a contradiction.

\smallskip

Now assume that $\mathcal{M}$ does not contain a cut of size at most $\ell$. Then by Proposition~\ref{prop:equiv_def_of_courteous}, assigning a unique color to each element results in an $\ell$-courteous coloring of $\mathcal{M}$.
\end{proof}

\subsubsection{On the minimum number of colors in courteous colorings}

Since $(k+1)$-edge-connected graphs and strongly $(k+1)$-arc-connected digraphs can be recognized in polynomial time, Propositions~\ref{edge-prop-1} and~\ref{edge-prop-2} imply that we can also decide in polynomial time whether a given graph or digraph has an edge- or arc-1-color-avoiding $k$-edge- or strongly $k$-arc-connected coloring, respectively. In the following, we prove that an edge-1-color-avoiding 1-connected coloring (if exists) with minimum number of colors can also be found in polynomial time using some matroid theoretical techniques by considering the given undirected graph as a graphic matroid. Since this proof naturally extends to nongraphic matroids as well, we only present the general proof for arbitrary matroids. Note that by Proposition~\ref{prop:colorability_of_matroids}, a matroid has a 1-courteous coloring if and only if it is bridgeless.

The following observation establishes a connection between the chromatic number and 1-courteous colorings of matroids.

\begin{proposition}
 Let $\mathcal{M}$ be a bridgeless matroid. The minimum number of colors needed for a 1-courteous coloring of $\mathcal{M}$ is equal to the chromatic number of $\mathcal{M}^*$.
\end{proposition}
\begin{proof}
 To see this, we show that a coloring of $\mathcal{M}$ is 1-courteous if and only if it is a proper coloring of $\mathcal{M}^*$. This directly follows from Proposition~\ref{prop:equiv_def_of_courteous}: a coloring of $\mathcal{M}$ is 1-courteous if and only if there are no monochromatic cycles in $\mathcal{M}^*$, which is equivalent to that this coloring is a proper coloring of~$\mathcal{M}^*$.
\end{proof}

As a direct consequence, we obtain the following result on the minimum number of colors in edge-1-color-avoiding connected colorings of graphs.

\begin{corollary} \label{chromatic_number_of_dual}
 The minimum number of colors in an edge-1-color-avoiding connected coloring of a 2-edge-connected graph $G$ is equal to the chromatic number of the dual matroid $\mathcal{M}^*(G)$ of $G$.
\end{corollary}

Since an optimal coloring of a matroid can be found in polynomial time given an independence oracle for the matroid~\cite{article:matroid_chromatic_number}, we obtain the following.

\begin{theorem} \label{chromatic_number_of_matroid_polytime}
 A 1-courteous coloring with minimum number of colors for a bridgeless matroid can be found in polynomial time given an independence oracle for the matroid. In particular, an edge-1-color-avoiding connected coloring with minimum number of colors for a 2-edge-connected graph $G$ can be found in polynomial time.
\end{theorem}

Note that the minimum number of colors needed in an edge-1-color-avoiding connected coloring of a graph $G$ is
\begin{itemize}[topsep=1.5pt, itemsep=-2pt, label=--]
 \item 0 if and only if $G \simeq K_1$;
 \item 1 if and only if $G$ has only one vertex and at least one loop;
 \item 2 if and only if $G$ has two edge-disjoint spanning trees\footnote{Let us note that a theorem of Tutte~\cite{article:Tutte_edge_disjoint_spanning_trees} and Nash-Williams~\cite{article:NashWilliams_edge_disjoint_spanning_trees} provides a characterization of those graphs that contain $k \in \mathbb{Z}_+$ --- in particular, two --- pairwise edge-disjoint spanning trees: a graph $G$ has $k$ pairwise edge-disjoint spanning trees if and only if $|F| \ge k \cdot \big( c(G-F) - 1 \big)$ for any $F \subseteq E(G)$.}.
\end{itemize}

\medskip

Assuming that a list $L_s$ of allowed colors is assigned to each element $s$ of a loopless matroid $\mathcal{M}$, we say that $\mathcal{M}$ is properly colorable from these lists if there exists a proper coloring of $\mathcal{M}$ in which each element $s$ receives a color from its list $L_s$. The choice number of $\mathcal{M}$, denoted by $\mathrm{ch}(\mathcal{M})$, is the minimum number $k$ such that $\mathcal{M}$ is properly colorable from any lists of size at least $k$. By a theorem of Seymour~\cite{article:seymourlistcolormatroid}, we know that $\mathrm{ch}(\mathcal{M}) = \chi(\mathcal{M})$ holds for any loopless matroid $\mathcal{M}$. Therefore, a matroid can be 1-courteously colored from any lists of length at most $k$ if and only if $\chi(\mathcal{M}^*) \le k$.

We note that for any $\ell \in \mathbb{Z}_+$ with $\ell \ge 2$, the complexity of finding an $\ell$-courteous coloring with minimum number of colors is open~\cite{webpage:MEMMORG}. In particular, the complexity of finding an edge-$\ell$-color-avoiding connected coloring also remains open.

\subsubsection{A weighted version of edge-color-avoiding connected colorings}

Now let us consider the case when the edges of the graph have weights and we require that not only an arbitrary spanning tree remains in the graph after the removal of any color, but a ``cheap'' one (i.e., one whose weight is not larger than a given number). To show that the problem of determining the minimum number of colors in such an edge-coloring is NP-hard, we make use of the following theorem.

\begin{theorem}[Bang-Jensen et al.~\cite{article:min_spanning_trees}] \label{Bang-Jensen}
 Let $G$ be a graph, $w: E(G) \to \mathbb{R}$ be a weight function and $\omega \in \mathbb{R}$ be a number. Deciding whether $G$ contains two edge-disjoint spanning trees both of weight at most $\omega$ is NP-complete. The problem remains NP-complete even if $G$ can be obtained from a path by adding a parallel edge to each edge of this path.
\end{theorem}

By Theorem~\ref{Bang-Jensen}, we obtain the following.

\begin{corollary}\label{cor:weighted}
 Let $G$ be a graph, $w: E(G) \to \mathbb{R}$ be a weight function, let $\omega \in \mathbb{R}$ and $c, \ell \in \mathbb{Z}_+$. Deciding whether $G$ has an edge-$\ell$-color-avoiding 1-connected coloring with at most $c$ colors such that after the removal of any at most $\ell$ colors, there remains a spanning tree of weight at most $\omega$ is NP-complete. The problem remains NP-complete even if $\ell = 1$ and $c = 2$.
\end{corollary}

\begin{proof}
 This problem is obviously in NP and now we show that it is NP-hard even for $\ell=1$ and $c=2$ by reducing the NP-complete problem described in Theorem~\ref{Bang-Jensen} to it. Let $G$ be an arbitrary graph, $\omega$ be an arbitrary number, and let $c=2$. Then $G$ has an edge-1-color-avoiding 1-connected coloring with two colors such that after the removal of any color there remains a spanning tree of weight at most $\omega$ if and only if $G$ contains two edge-disjoint spanning trees both of weight at most $\omega$ --- in which case, we color the edges of one of these spanning trees with one of the colors, and the edges of the other spanning tree with the other color, and we color arbitrarily the remaining edges.
\end{proof}

\subsubsection{On the minimum number of colors in edge-\texorpdfstring{$\boldsymbol{\ell}$}{}-color-avoiding \texorpdfstring{$\boldsymbol{k}$}{}-edge- or \texorpdfstring{$\boldsymbol{k}$}{}-vertex-con\-nected colorings with \texorpdfstring{$\boldsymbol{k \ge 2}$}{}}

Now let us consider the case when $k \ge 2$ and we want to decide whether a graph has an edge-$\ell$-color-avoiding $k$-edge- or $k$-vertex-connected coloring with a given number of colors. For this purpose, we rely on the following result.

\begin{theorem}[Kotzig \cite{article:Hamiltonian_decompositon_1}, Martin \cite{article:Hamiltonian_decompostion_2}] \label{thm:Kotzig}
 Deciding whether a 4-regular graph can be decomposed into two edge-disjoint Hamiltonian cycles is NP-complete.
\end{theorem}

From this, we obtain the following consequence.

\begin{corollary} \label{cor:Kotzig}
 Let $G$ be a graph and let $c, k, \ell \in \mathbb{Z}_+$ with $k \ge 2$. Deciding whether $G$ has an edge-$\ell$-color-avoiding $k$-edge- or $k$-vertex-connected coloring with at most $c$ colors is NP-complete. The problem remains NP-complete even if $G$ is 4-regular, $\ell = 1$ and $c = k = 2$.
\end{corollary}
\begin{proof} 
 These problems are obviously in NP and now we show that they are NP-hard even for $\ell=1$ and $c=k=2$ and for 4-regular graphs by reducing the NP-complete problem described in Theorem~\ref{thm:Kotzig} to it. Let $G$ be an arbitrary 4-regular graph and let $\ell = 1$ and $c=k=2$. Then $G$ has an edge-1-color-avoiding 2-edge- or 2-vertex-connected coloring if and only if $G$ can be decomposed into two edge-disjoint Hamiltonian cycles~---~in which case, we color the edges of one of these Hamiltonian cycles with one of the colors, and the edges of the other Hamiltonian cycles with the other color.
\end{proof}

\subsubsection{On the minimum number of colors in arc-color-avoiding strongly connected colorings}

Now we focus on digraphs, and consider the problem of deciding whether a given digraph has an arc-color-avoiding strongly connected coloring with a given number of colors. For this, we make use of the following result.

\begin{theorem}[Yeo~\cite{book:decomposing_into_two_strong_spanning_subdigraphs}] \label{Yeo}
 Deciding whether a digraph has two arc-disjoint strongly connected spanning subdigraphs is NP-complete.
\end{theorem}

The following is a straightforward consequence of Theorem~\ref{Yeo}.

\begin{corollary} \label{cor:Yeo}
 Let $D$ be a digraph and let $c, k, \ell \in \mathbb{Z}_+$. Deciding whether $D$ has an arc-$\ell$-color-avoiding strongly $k$-arc- or $k$-vertex-connected coloring with at most $c$ colors is NP-complete. The problem remains NP-complete even if $k = \ell = 1$ and $c = 2$.
\end{corollary}
\begin{proof}
 This problem is obviously in NP, and by Theorem~\ref{Yeo} it is also NP-hard: for $k = \ell=1$ and $c=2$, our problem becomes equivalent to deciding whether $G$ has two arc-disjoint strongly connected spanning subdigraphs.
\end{proof}

\subsubsection{On the minimum number of colors in arc-color-avoiding rooted connected colorings}

Still focusing on digraphs, we show that the complexity of finding color-avoiding rooted connected colorings with a minimum number of colors can differ from that of finding color-avoiding strongly connected colorings. In particular, while the latter is always NP-hard, the former may be solvable in certain cases. To establish this, we first make use of the following theorem of Edmonds.

\begin{theorem}[Edmonds~\cite{article:rootededmonds}] \label{thm:rootededmonds}
 Let $D$ be a digraph with a given vertex $r \in V(D)$, and let $k \in \mathbb{Z}_+$. Then $D$ is $r$-rooted $k$-arc-connected if and only if it contains $k$ pairwise arc-disjoint spanning $r$-arborescences.
\end{theorem}

We are now ready to determine the minimum number of colors in an arc-$\ell$-color-avoiding $r$-rooted 1-connected coloring—provided such a coloring exists; that is, by Proposition~\hyperlink{linkdest:item3:edge-prop-1}{\ref*{edge-prop-1}.\ref*{item3:edge-prop-1}}, if and only if the digraph is $r$-rooted and $(\ell + 1)$-arc-connected.

\begin{theorem} \label{thm:rooted_mincolors}
 The minimum number of colors in an arc-$\ell$-color-avoiding $r$-rooted 1-connected coloring of an $r$-rooted $(\ell + 1)$-arc-connected digraph is $\ell + 1$.
\end{theorem}
\begin{proof}
 Let $D$ be an $r$-rooted $(\ell + 1)$-arc-connected digraph. By the definition of arc-$\ell$-color-avoiding 1-con\-nected colorings, we need at least $\ell + 1$ colors. Now we give an arc-$\ell$-color-avoiding 1-connected coloring of $D$ with $\ell + 1$ colors. By Theorem~\ref{thm:rootededmonds}, the digraph $D$ has $\ell + 1$ arc-disjoint spanning $r$-arborescences; let us color all the arcs of the $i$-th $r$-arborescences with the $i$-th color for any $i \in [\ell + 1]$. Then after the removal of the arcs of any $\ell$ colors, at least one $r$-arborescences remains intact, thus the remaining digraph is still $r$-rooted 1-connected. Hence this coloring is indeed an arc-$\ell$-color-avoiding 1-connected coloring of $D$.
\end{proof}

Since we can find $k$ arc-disjoint spanning $r$-arborescences in an $r$-rooted $k$-arc connected digraph in polynomial time, see for example~\cites{article:tarjan,article:bhalgat}, the proof of Theorem~\ref{thm:rooted_mincolors} also implies the following result.

\begin{theorem} \label{thm:rooted_arc_mincoloring}
 An arc-$\ell$-color-avoiding $r$-rooted 1-connected coloring with $\ell + 1$ colors for an $r$-rooted $(\ell + 1)$-arc-connected digraph $D$ can be found in polynomial time.
\end{theorem}

Determining the minimum number of colors in an arc-$\ell$-color-avoiding $k$-edge- or $k$-vertex-connected $r$-rooted coloring of an $r$-rooted $(k + \ell)$-arc-connected digraph for $k \ge 2$ remains open.\footnote{We note that Frank~\cite{article:franksejtes} conjectured that a digraph $D$ is $r$-rooted $k$-vertex-connected if and only if it contains $k$ independent spanning $r$-arborescences; where two spanning $r$-arborescences are called independent if for any vertex $v \in V(D)$, the paths from $r$ to $v$ in the two arborescences are internally vertex-disjoint. The case $k = 2$ was verified by Whitty~\cite{article:whitty}, but for $k \ge 3$ the conjecture was disproved by Huck~\cite{article:huck}.}

\subsubsection{A weighted version of arc-color-avoiding rooted connected colorings}

Finally, let us consider a weighted version of the above problem.

\begin{theorem}
 Let $D$ be a graph with a given vertex $r \in V(D)$, let $w: E(G) \to \mathbb{R}$ be a weight function, let $\omega \in \mathbb{R}$ and $c, \ell \in \mathbb{Z}_+$. Deciding whether $D$ has an arc-$\ell$-color-avoiding $r$-rooted 1-connected coloring with at most $c$ colors such that after the removal of any at most $\ell$ colors, there remains a spanning $r$-arborescence of weight at most $\omega$ is NP-complete. The problem remains NP-complete even if $\ell = 1$ and $c = 2$.
\end{theorem}
\begin{proof}
 This problem is obviously in NP and now we show that it is NP-hard even for $\ell = 1$ and $c=2$ by reducing the NP-complete problem described in Theorem~\ref{Bang-Jensen} to it. Let $G$ be a graph that is obtained from a $u$-$v$ path by adding a parallel edge to each of its edge, let $w \colon E(G) \to \mathbb{R}$ be a weight function, $\omega$ be an arbitrary number, and let $c=2$. Let $D$ be the digraph that can be obtained from $G$ by orienting each edge of $G$ away from $u$ and toward $v$, and let $r = u$ in $D$. Then $D$ has an arc-1-color-avoiding $r$-rooted 1-connected coloring with two colors such that after the removal of any color, there remains a spanning $r$-arborescence of weight at most $\omega$ if and only if $D$ has two arc-disjoint spanning trees both of weight at most $\omega$, which is equivalent to $G$ having two edge-disjoint spanning trees both of weight at most $\omega$. 
\end{proof}

\subsection{Vertex-color-avoiding connected colorings}

The following proposition gives a necessary and sufficient condition for the existence of vertex-color-avoiding connected colorings of a graph, and for the existence of vertex-color-avoiding strongly or rooted connected colorings of a digraph.

\begin{proposition} \label{prop:vertexcoloring}
 Let $k, \ell \in \mathbb{Z}_+$.
 \begin{enumerate}[label=(\arabic*), topsep=2pt, itemsep=0pt]
  \item \label{item1:vertexcoloring} A graph $G$ has a vertex-$\ell$-color-avoiding $k$-edge- or $k$-vertex-con\-nec\-ted coloring if and only if $G$ is $k$-edge- or $k$-vertex-connected, respectively.
  \item \label{item2:vertexcoloring} A digraph $D$ has a vertex-$\ell$-color-avoiding strongly $k$-arc- or $k$-vertex-connected coloring if and only if $D$ is strongly $k$-arc- or $k$-vertex-connected, respectively.
  \item \label{item3:vertexcoloring} A digraph $D$ with a given vertex $r \in V(D)$ has a vertex-$\ell$-color-avoiding $r$-rooted $k$-arc- or $k$-vertex-connected coloring if and only if $D$ is $r$-rooted $k$-arc- or $k$-vertex-connected, respectively.
\end{enumerate}
\end{proposition}
\begin{proof}
 We only prove \ref{item1:vertexcoloring} since \ref{item2:vertexcoloring} and \ref{item3:vertexcoloring} can be proved analogously.

 \smallskip
 
 By the definition of vertex-$\ell$-color-avoiding $k$-edge-connectivity, a graph $G$ must be \mbox{$k$-edge}-connected to have a vertex-$\ell$-color-avoiding $k$-edge-connected coloring, and if $G$ is $k$-edge-connected, then assigning the same color to each vertex results in a vertex-$\ell$-color-avoiding $k$-edge-connected graph.
\end{proof}

Using the above proof, we can also observe the following about the minimum number of colors in vertex-color-avoiding connected colorings of undirected graphs and in vertex-color-avoiding strongly or rooted connected colorings of directed graphs.

\begin{corollary} \label{cor:vertexmincolor}
  Let $k, \ell \in \mathbb{Z}_+$.
 \begin{enumerate}[label=(\arabic*), topsep=2pt, itemsep=0pt]
  \item The minimum number of colors in a vertex-$\ell$-color-avoiding $k$-edge- or $k$-vertex-connected coloring of a $k$-edge- or $k$-vertex-connected graph is 1. 
 
  \item The minimum number of colors in a vertex-$\ell$-color-avoiding strongly $k$-arc- or $k$-vertex-connected coloring of a strongly $k$-arc- or $k$-vertex-connected digraph is 1.

  \item The minimum number of colors in a vertex-$\ell$-color-avoiding $r$-rooted $k$-arc- or $k$-vertex-connected coloring of an $r$-rooted $k$-arc- or $k$-vertex-connected digraph is 1.
 \end{enumerate}
\end{corollary}

\subsection{Internally color-avoiding connected colorings}

In the following, we give a necessary and sufficient condition for the existence of internally vertex-color-avoiding connected colorings of a graph and for the existence of internally vertex-color-avoiding strongly or rooted connected colorings of a digraph. First, let us focus on the cases with $k$-edge-connectivity.

\begin{proposition} \label{prop:ivertexcoloring-edge}
 Let $k, \ell \in \mathbb{Z}_+$.
 \begin{enumerate}[label=(\arabic*), topsep=2pt, itemsep=0pt]
  \item \label{item1:ivertexcoloring-edge} A graph $G=(V,E)$ has an internally vertex-$\ell$-color-avoiding $k$-edge-connected coloring if and only if $G - V'$ is $k$-edge-connected for any $V' \subseteq V$ with $|V'| \le \ell$.
 
  \item \label{item2:ivertexcoloring-edge} A digraph $D=(V,E)$ has an internally vertex-$\ell$-color-avoiding strongly $k$-arc-connected coloring if and only if $D-V'$ is strongly $k$-arc-connected for any $V' \subseteq V$ with $|V'| \le \ell$.

  \item \label{item3:ivertexcoloring-edge} A digraph $D=(V,E)$ with a given vertex $r \in V$ has an internally vertex-$\ell$-color-avoiding $r$-rooted $k$-arc-connected coloring if and only if $D-V'$ is $r$-rooted $k$-arc-connected for any $V' \subseteq V - \{ r \}$ with $|V'| \le \ell$.
 \end{enumerate}
\end{proposition}
\begin{proof}
 We only prove \ref{item1:ivertexcoloring-edge} since \ref{item2:ivertexcoloring-edge} and \ref{item3:ivertexcoloring-edge} can be proved analogously.

 \medskip
 
 First, assume that there exists a set of vertices $V' \subseteq V$ with $|V'| \le \ell$ for which $G - V'$ is not $k$-edge-connected, i.e., there exists a set $E'$ of at most $k$ edges whose removal disconnects the graph. Suppose to the contrary that $G$ has an internally vertex-$\ell$-color-avoiding $k$-edge-connected coloring $f \colon V \to C$, where $C$ is a finite color set. Then $V' \cup E'$ is a mixed cut with $\big| f(V') \big| \le \ell$ and $|E'| \le k-1$, a contradiction.

 \smallskip
 
 Now assume that $G - V'$ is $k$-edge-connected for all $V' \subseteq V$ with $|V'| \le \ell$. Then assigning a unique color to each vertex results in an internally vertex-$\ell$-color-avoiding $k$-edge-connected graph.
\end{proof}

Now let us consider the corresponding cases of $k$-vertex-connectivity.

\begin{proposition} \label{prop:ivertexcoloring-vertex}
 Let $k, \ell \in \mathbb{Z}_+$.
 \begin{enumerate}[label=(\arabic*), topsep=2pt, itemsep=0pt]
  \item \label{item1:ivertexcoloring-vertex} A graph $G$ has an internally vertex-$\ell$-color-avoiding $k$-vertex-connected coloring if and only if $G$ is $(k+\ell)$-vertex-connected.
 
  \item \label{item2:ivertexcoloring-vertex} A digraph $D$ has an internally vertex-$\ell$-color-avoiding strongly $k$-vertex-connected coloring if and only if $D$ is strongly $(k+\ell)$-vertex-connected.

  \item \label{item3:ivertexcoloring-vertex} A digraph $D$ with a given vertex $r \in V$ has an internally vertex-$\ell$-color-avoiding $r$-rooted $k$-vertex-connected coloring if and only if $D$ is $r$-rooted $(k+\ell)$-vertex-connected.
 \end{enumerate}
\end{proposition}
\begin{proof}
 We only prove \ref{item1:ivertexcoloring-vertex} since \ref{item2:ivertexcoloring-vertex} and \ref{item3:ivertexcoloring-vertex} can be proved analogously.

 \medskip
 
 First, assume that $G$ is not $(k+\ell)$-vertex-connected, i.e., there exists a set of vertices $V' \subset V$ with $|V'| \le k+\ell-1$ for which $G-V'$ is not connected. Consider an arbitrary vertex-coloring of $G$, and let $C'$ denote the set of colors of the vertices in $V'$. Let $C'' \subseteq C'$ be an arbitrary subset with $|C''| = \min \big( \ell, |C'| \big)$. Observe that there are at most $k-1$ vertices in $V'$ that have different colors from those in $C''$. Thus $V'$ is a vertex cut consisting of vertices of at most $\ell$ colors with the exception of at most $k-1$ additional vertices. Therefore, $G$ is not internally vertex-$\ell$-color-avoiding $k$-vertex connected.

 \smallskip
 
 Now assume that $G$ is $(k+\ell)$-vertex-connected. Then assigning a unique color to each vertex results in an internally vertex-$\ell$-color-avoiding $k$-vertex-connected graph.
\end{proof}

In contrast to the cases of edge- and vertex-1-color-avoiding 1-connectivity, the following theorem says that deciding whether a graph has an internally vertex-1-color-avoiding 1-connected coloring with a given number of colors is NP-complete.

\begin{theorem} \label{thm:ivckecvc}
 Let $G$ be a graph and let $c, k, \ell \in \mathbb{Z}_+$. Deciding whether $G$ has an internally vertex-$\ell$-color-avoiding $k$-edge-connected coloring with at most $c$ colors is NP-complete. Similarly, deciding whether $G$ has an internally vertex-$\ell$-color-avoiding $k$-vertex-connected coloring with at most $c$ colors is NP-complete. These problems remain NP-complete even if $k = \ell = 1$ and $c = 2$.
\end{theorem}
\begin{proof} 
 The problem is clearly in NP. Now we show that the problem is NP-hard even for $k = \ell = 1$ and $c=2$ by reducing the NP-complete problem of recognizing 2-colorable (loopless) hypergraphs~\cite{article:hypergraph_2-colorability} to it. (A hypergraph is called 2-colorable if its vertices can be colored with two colors without creating monochromatic hyperedges.)
 
 Let $\mathcal{H}$ be an arbitrary hypergraph without loops and let $G$ be defined as follows. Let

 \vspace{-10pt}
 
 \[ V(G) \colonequals \big\{ v \bigm| v \in V(\mathcal{H}) \big\} \cup \big\{ w_e \bigm| e \in E(\mathcal{H}) \big\} \]

 \vspace{-5pt}
 
 \noindent and
 
 \vspace{-10pt}
 
 \[ E(G) \colonequals \big\{ \{ u, v \} \bigm| u, v \in V(\mathcal{H}) \big\} \cup \big\{ \{ v, w_e \} \bigm| v \in V(\mathcal{H}), \, e \in E(\mathcal{H}), \, v \in e \big\} \text{.} \]

 \vspace{-2pt}
 
 By Proposition~\hyperlink{item:iii_of_prop_IVCA}{\ref*{proposition:equiv_def_of_IVCA}.\ref*{item:iii_of_IVCA}}, it is not difficult to see that if $\mathcal{H}$ has a proper 2-vertex-coloring with red and blue colors, then using the same vertex-coloring for the corresponding vertices in $G$ and coloring the remaining vertices of $G$ red yields an internally vertex-1-color-avoiding 1-connected coloring of $G$; for an example, see Figure \ref{fig:hypergraphproof}.

 %\vspace{-65pt}
 
 \begin{figure}[H]
 \centering
 \begin{tikzpicture}[scale=2.5]
 \tikzstyle{vertex_red}=[draw, shape=regular polygon, regular polygon sides=3, fill=red, minimum size=12pt, inner sep=0]
 \tikzstyle{vertex_blue}=[draw, shape=rectangle, fill=blue, minimum size=8pt, inner sep=0]
 \tikzstyle{vertex_green}=[draw, shape=diamond, fill=green, minimum size=11pt, inner sep=0]
 \tikzstyle{vertex_yellow}=[draw, shape=star, fill=yellow, minimum size=11pt, inner sep=0]
  
  \node[vertex_red] (v1) at (0,0) [label=below left: $v_1$] {};
  \node[vertex_blue] (v2) at (60:0.5) [label=left: $v_2$] {};
  \node[vertex_red] (v3) at (60:1) [label=above: $v_3$] {};
  \node[vertex_blue] (v4) at (30:0.8660254) [label=right: $v_4$] {};
  \node[vertex_red] (v5) at (1,0) [label=below right: $v_5$] {};
  \node[vertex_blue] (v6) at (0.5,0) [label=below: $v_6$] {};
  
 \draw[thick, shift={(65:0.525)}, rotate=-30] (0,0) ellipse (0.3 and 0.85);
 \node[label=above:{$e_1$}] at (-0.15,0.633) {};
 \draw[thick, shift={(31:0.9160254)}, rotate=30] (0,0) ellipse (0.3 and 0.85);
 \node[label=above:{$e_2$}] at (1.15,0.633) {};
 \draw[thick, shift={(0.5,-0.05)}] (0,0) ellipse (0.85 and 0.3);
 \node[label=below:{$e_3$}] at (0.5,-0.35) {};
 
 \begin{scope}[shift={(3,0)}]
 
  \node[vertex_red] (v1) at (0,0) [label=below left: $v_1$] {};
  \node[vertex_blue, shift={(-0.25,0.144)}] (v2) at (60:0.5) [label=left: $v_2$] {};
  \node[vertex_red] (v3) at (60:1) [label=above: $v_3$] {};
  \node[vertex_blue, shift={(0.25,0.144)}] (v4) at (30:0.8660254) [label=right: $v_4$] {};
  \node[vertex_red] (v5) at (1,0) [label=below right: $v_5$] {};
  \node[vertex_blue, shift={(0,-0.3)}] (v6) at (0.5,0) [label={[xshift=8pt, yshift=-19pt]:$v_6$}] {};
  
  \node[vertex_red, shift={(-1.1,0.635)}] (we1) at (60:0.5) [label=above left: $w_{e_1}$] {};
  \node[vertex_red, shift={(1.1,0.635)}] (we2) at (30:0.8660254) [label=above right: $w_{e_2}$] {};
  \node[vertex_red] (we3) at (0.5,-0.45) [label=below: $w_{e_3}$] {};
  
  \draw[very thick] (v1)--(v2);
  \draw[very thick] (v1)--(v3);
  \draw[very thick] (v1)--(v4);
  \draw[very thick] (v1)--(v5);
  \draw[very thick] (v1)--(v6);
  \draw[very thick] (v2)--(v3);
  \draw[very thick] (v2)--(v4);
  \draw[very thick] (v2)--(v5);
  \draw[very thick] (v2)--(v6);
  \draw[very thick] (v3)--(v4);
  \draw[very thick] (v3)--(v5);
  \draw[very thick] (v3)--(v6);
  \draw[very thick] (v4)--(v5);
  \draw[very thick] (v4)--(v6);
  \draw[very thick] (v5)--(v6);
  \draw[very thick] (we1)--(v1);
  \draw[very thick] (we1)--(v2);
  \draw[very thick] (we1)--(v3);
  \draw[very thick] (we2)--(v3);
  \draw[very thick] (we2)--(v4);
  \draw[very thick] (we2)--(v5);
  \draw[very thick] (we3)--(v5);
  \draw[very thick] (we3)--(v6);
  \draw[very thick] (we3)--(v1);
 \end{scope}
 \end{tikzpicture}
 \caption{A hypergraph $\mathcal{H}$ with a proper 2-coloring and the corresponding internally vertex-1-color-avoiding 1-connected graph $G$ constructed in the proof of Theorem~\ref{thm:ivckecvc}.}
 \label{fig:hypergraphproof}
\end{figure}
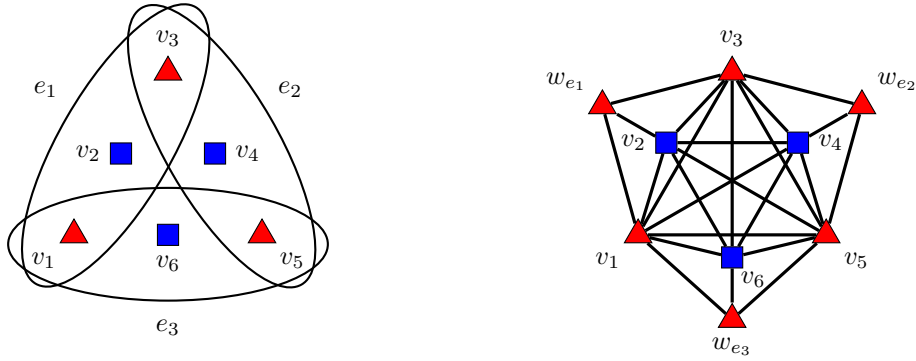
 
 Conversely, if $G$ has an internally vertex-1-color-avoiding 1-connected coloring with red and blue colors, then assigning the same colors of the corresponding vertices in $\mathcal{H}$ (regardless the colors of the remaining vertices of $G$) yields a proper 2-vertex-coloring of $\mathcal{H}$.
\end{proof}

Let us note that a graph has an internally vertex-1-color-avoiding $k$-edge-connected coloring with one color if and only if the graph is a complete graph on at least $k+1$ vertices.

The following corollary is a straightforward consequence of Theorem~\ref{thm:ivckecvc} for internally vertex-color-avoiding strongly connected colorings.

\begin{corollary} \label{cor:ivckecvc}
 Let $D$ be a digraph and let $c, k, \ell \in \mathbb{Z}_+$. Deciding whether $D$ has an internally vertex-$\ell$-color-avoiding strongly $k$-arc-connected coloring with at most $c$ colors is NP-complete. The problem remains NP-complete even if $k = \ell = 1$ and $c=2$.
\end{corollary}

Finally, let us consider the analogous problem for internally vertex-color-avoiding rooted connected colorings.

\begin{theorem} \label{thm:rooted_iv_mincolor}
 Let $D$ be a digraph with a given vertex $r \in V(D)$ and let $c, k, \ell \in \mathbb{Z}_+$. Deciding whether $D$ has an internally vertex-$\ell$-color-avoiding $r$-rooted $k$-edge-connected coloring with at most $c$ colors is NP-complete. Similarly, deciding whether $G$ has an internally vertex-$\ell$-color-avoiding $k$-vertex-connected coloring with at most $c$ colors is NP-complete. These problems remain NP-complete even if $k = \ell = 1$ and $c = 2$.
\end{theorem}
\begin{proof} 
 The problem is clearly in NP. Now we show that the problem is NP-hard even for $k = \ell = 1$ and $c=2$ by reducing the NP-complete problem of recognizing 2-colorable (loopless) hypergraphs~\cite{article:hypergraph_2-colorability} to it.
 
 Let $\mathcal{H}$ be an arbitrary hypergraph without loops and let $D$ be defined as follows. Let
 \[ V(D) \colonequals \{ r \} \cup \big\{ v \bigm| v \in V(\mathcal{H}) \big\} \cup \big\{ w_e \bigm| e \in E(\mathcal{H}) \big\} \]
 and
 \[ E(D) \colonequals \big\{ uv, vu \bigm| u, v \in V(\mathcal{H}) \big\} \cup \big\{ v w_e, w_e v \bigm| v \in V(\mathcal{H}), \, e \in E(\mathcal{H}), \, v \in e \big\} \cup \big\{ rv \bigm| v \in V(\mathcal{H}) \big\} \text{.} \]
 
 By Proposition~\hyperlink{item:iii_of_prop_IVCA}{\ref*{proposition:equiv_def_of_IVCA}.\ref*{item:iii_of_IVCA}}, it is not difficult to see that if $\mathcal{H}$ has a proper 2-vertex-coloring with red and blue colors, then using the same vertex-coloring for the corresponding vertices in $D$ and coloring the remaining vertices of $D$ red yields an internally vertex-1-color-avoiding $r$-rooted 1-connected coloring of $D$. Conversely, if $D$ has an internally vertex-1-color-avoiding $r$-rooted 1-connected coloring with red and blue colors, then assigning the same colors of the corresponding vertices in $\mathcal{H}$ (regardless the colors of the remaining vertices of $D$) yields a proper 2-vertex-coloring of $\mathcal{H}$.
\end{proof}

For a summary of the results presented in this section, see Table~\ref{table:summary_results}.

\begin{table}[p]
{\small
\begin{center}
 \setlength{\tabcolsep}{1.4pt}
 \captionsetup{type=table}
 \rotatebox{90}{
 \begin{tabular}{!{\vrule width 1.25pt} c !{\vrule width 1.25pt} c|c !{\vrule width 1.25pt} c|c !{\vrule width 1.25pt} c|c !{\vrule width 1.25pt}} \specialrule{1.25pt}{0pt}{0pt}
   & \multicolumn{2}{c !{\vrule width 1.25pt}}{edge-$\ell$-color-avoiding} & \multicolumn{2}{c !{\vrule width 1.25pt}}{vertex-$\ell$-color-avoiding} & \multicolumn{2}{c !{\vrule width 1.25pt}}{internally vertex-$\ell$-color-avoiding} \\
   & $k$-edge-connected & $k$-vertex-connected & $k$-edge-connected & $k$-vertex-connected & $k$-edge-connected & $k$-vertex-connected \\ \specialrule{1.25pt}{0pt}{0pt}
   \multirow{5}{*}{\parbox{80pt}{\centering Necessary and \\ sufficient condition \\ on the existence of \\ a $\ldots$ coloring of $G$}} & \multirow{5}{*}{\parbox{80pt}{\centering $G$ is $(k + \ell)$-edge- \\connected \\ see Prop.~\ref{edge-prop-1}}} & \multirow{5}{*}{\parbox{80pt}{\centering $G - E'$ is \\ $k$-vertex-connected \\ for any $E' \subseteq E(G)$ \\ with $|E'| \le \ell$, \\ see Prop.~\ref{edge-prop-2}}} & \multirow{5}{*}{\parbox{80pt}{\centering $G$ is $k$-edge- \\ connected, \\ see Prop.~\ref{prop:vertexcoloring}}} & \multirow{5}{*}{\parbox{80pt}{\centering $G$ is $k$-vertex- \\ connected, \\ see Prop.~\ref{prop:vertexcoloring}}} & \multirow{5}{*}{\parbox{80pt}{\centering $G - V'$ is \\ $k$-edge-connected \\ for any $V' \subseteq V(G)$ \\ with $|V'| \le \ell$, \\ see Prop.~\ref{prop:ivertexcoloring-edge}}} & \multirow{5}{*}{\parbox{80pt}{\centering $G$ is $(k + \ell)$-vertex- \\ connected, \\ see Prop.~\ref{prop:ivertexcoloring-vertex}}} \\
   & & & & & & \\
   & & & & & & \\
   & & & & & & \\
   & & & & & & \\ \specialrule{1.25pt}{0pt}{0pt}
   \multirow{8}{*}{\parbox{80pt}{\centering Minimum number \\ of colors in a $\ldots$ \\ coloring of $G$ \\ (if exists)}} & \multicolumn{2}{c !{\vrule width 1.25pt}}{\multirow{4}{*}{\parbox{150pt}{\centering $k = \ell = 1$: \\ $\chi \big( \mathcal{M}^*(G) \big)$, see Cor.~\ref{chromatic_number_of_dual}; \\
   can be determined in \\ polynomial time, see Thm.~\ref{chromatic_number_of_matroid_polytime}}}} & \multicolumn{2}{c !{\vrule width 1.25pt}}{\multirow{8}{*}{\parbox{80pt}{\centering 1, \\ see Cor.~\ref{cor:vertexmincolor}}}} & \multicolumn{2}{c !{\vrule width 1.25pt}}{\multirow{8}{*}{\parbox{80pt}{\centering NP-complete, \\ see Thm.~\ref{thm:ivckecvc}}}} \\
   & \multicolumn{2}{c !{\vrule width 1.25pt}}{} & \multicolumn{2}{c !{\vrule width 1.25pt}}{} & \multicolumn{2}{c !{\vrule width 1.25pt}}{} \\
   & \multicolumn{2}{c !{\vrule width 1.25pt}}{} & \multicolumn{2}{c !{\vrule width 1.25pt}}{} & \multicolumn{2}{c !{\vrule width 1.25pt}}{} \\
   & \multicolumn{2}{c !{\vrule width 1.25pt}}{} & \multicolumn{2}{c !{\vrule width 1.25pt}}{} & \multicolumn{2}{c !{\vrule width 1.25pt}}{} \\ \cline{2-3}
   & \multicolumn{2}{c !{\vrule width 1.25pt}}{\multirow{2}{*}{\parbox{150pt}{\centering $k = 1$, $\ell \ge 2$: \\ open}}} & \multicolumn{2}{c !{\vrule width 1.25pt}}{} & \multicolumn{2}{c !{\vrule width 1.25pt}}{} \\
   & \multicolumn{2}{c !{\vrule width 1.25pt}}{} & \multicolumn{2}{c !{\vrule width 1.25pt}}{} & \multicolumn{2}{c !{\vrule width 1.25pt}}{} \\ \cline{2-3}
   & \multicolumn{2}{c !{\vrule width 1.25pt}}{\multirow{2}{*}{\parbox{150pt}{\centering $k \ge 2$: \\ NP-complete, see Cor.~\ref{cor:Kotzig}}}} & \multicolumn{2}{c !{\vrule width 1.25pt}}{} & \multicolumn{2}{c !{\vrule width 1.25pt}}{} \\
   & \multicolumn{2}{c !{\vrule width 1.25pt}}{} & \multicolumn{2}{c !{\vrule width 1.25pt}}{} & \multicolumn{2}{c !{\vrule width 1.25pt}}{} \\ \specialrule{1.25pt}{0pt}{0pt}
   \multicolumn{7}{c}{\vphantom{\Big|}} \\ \specialrule{1.25pt}{0pt}{0pt}
   & \multicolumn{2}{c !{\vrule width 1.25pt}}{arc-$\ell$-color-avoiding strongly} & \multicolumn{2}{c !{\vrule width 1.25pt}}{vertex-$\ell$-color-avoiding strongly} & \multicolumn{2}{c !{\vrule width 1.25pt}}{internally vertex-$\ell$-color-avoiding strongly} \\
   & $k$-arc-connected & $k$-vertex-connected & $k$-arc-connected & $k$-vertex-connected & $k$-arc-connected & $k$-vertex-connected \\ \specialrule{1.25pt}{0pt}{0pt}
   \multirow{5}{*}{\parbox{80pt}{\centering Necessary and \\ sufficient condition \\ on the existence of \\ a $\ldots$ coloring of $D$}} & \multirow{5}{*}{\parbox{80pt}{\centering $D$ is strongly \\ $(k + \ell)$-arc- \\ connected, \\ see Prop.~\ref{edge-prop-1}}} & \multirow{5}{*}{\parbox{80pt}{\centering $D - E'$ is strongly \\ $k$-vertex-connected \\ for any $E' \subseteq E(D)$ \\ with $|E'| \le \ell$, \\ see Prop.~\ref{edge-prop-2}}} & \multirow{5}{*}{\parbox{80pt}{\centering $D$ is strongly \\ $k$-arc-connected, \\ see Prop.~\ref{prop:vertexcoloring}}} & \multirow{5}{*}{\parbox{80pt}{\centering $D$ is strongly \\ $k$-vertex-connected, \\ see Prop.~\ref{prop:vertexcoloring}}} & \multirow{5}{*}{\parbox{80pt}{\centering $D - V'$ is strongly \\ $k$-arc-connected \\ for any $V' \subseteq V(D)$ \\ with $|V'| \le \ell$, \\ see Prop.~\ref{prop:ivertexcoloring-edge}}} & \multirow{5}{*}{\parbox{80pt}{\centering $D$ is strongly $(k + \ell)$-vertex- \\ connected, \\ see Prop.~\ref{prop:ivertexcoloring-vertex}}} \\
   & & & & & & \\
   & & & & & & \\
   & & & & & & \\
   & & & & & & \\ \specialrule{1.25pt}{0pt}{0pt}
   \multirow{4}{*}{\parbox{80pt}{\centering Minimum number \\ of colors in a $\ldots$ \\ coloring of $D$ \\ (if exists)}} & \multicolumn{2}{c !{\vrule width 1.25pt}}{\multirow{4}{*}{\parbox{150pt}{\centering NP-complete, \\ see Cor.~\ref{cor:Yeo}}}} & \multicolumn{2}{c !{\vrule width 1.25pt}}{\multirow{4}{*}{\parbox{80pt}{\centering 1, \\ see Cor.~\ref{cor:vertexmincolor}}}} & \multicolumn{2}{c !{\vrule width 1.25pt}}{\multirow{4}{*}{\parbox{80pt}{\centering NP-complete, \\ see Cor.~\ref{cor:ivckecvc}}}} \\
   & \multicolumn{2}{c !{\vrule width 1.25pt}}{} & \multicolumn{2}{c !{\vrule width 1.25pt}}{} & \multicolumn{2}{c !{\vrule width 1.25pt}}{} \\
   & \multicolumn{2}{c !{\vrule width 1.25pt}}{} & \multicolumn{2}{c !{\vrule width 1.25pt}}{} & \multicolumn{2}{c !{\vrule width 1.25pt}}{} \\
   & \multicolumn{2}{c !{\vrule width 1.25pt}}{} & \multicolumn{2}{c !{\vrule width 1.25pt}}{} & \multicolumn{2}{c !{\vrule width 1.25pt}}{} \\ \specialrule{1.25pt}{0pt}{0pt}
   \multicolumn{7}{c}{\vphantom{\Big|}} \\ \specialrule{1.25pt}{0pt}{0pt}
   & \multicolumn{2}{c !{\vrule width 1.25pt}}{arc-$\ell$-color-avoiding rooted} & \multicolumn{2}{c !{\vrule width 1.25pt}}{vertex-$\ell$-color-avoiding rooted} & \multicolumn{2}{c !{\vrule width 1.25pt}}{internally vertex-$\ell$-color-avoiding rooted} \\
   & $k$-arc-connected & $k$-vertex-connected & $k$-arc-connected & $k$-vertex-connected & $k$-arc-connected & $k$-vertex-connected \\ \specialrule{1.25pt}{0pt}{0pt}
   \multirow{5}{*}{\parbox{80pt}{\centering Necessary and \\ sufficient condition \\ on the existence of \\ a $\ldots$ coloring of $D$}} & \multirow{5}{*}{\parbox{80pt}{\centering $D$ is $r$-rooted \\ $(k + \ell)$-arc- \\ connected, \\ see Prop.~\ref{edge-prop-1}}} & \multirow{5}{*}{\parbox{80pt}{\centering $D - E'$ is $r$-rooted \\ $k$-vertex-connected \\ for any $E' \subseteq E(D)$ \\ with $|E'| \le \ell$, \\ see Prop.~\ref{edge-prop-2}}} & \multirow{5}{*}{\parbox{80pt}{\centering $D$ is $r$-rooted \\ $k$-arc-connected, \\ see Prop.~\ref{prop:vertexcoloring}}} & \multirow{5}{*}{\parbox{80pt}{\centering $D$ is $r$-rooted \\ $k$-vertex-connected, \\ see Prop.~\ref{prop:vertexcoloring}}} & \multirow{5}{*}{\parbox{80pt}{\centering $D - V'$ is $r$-rooted \\ $k$-arc-connected for \\ any $V' \! \subseteq \! V(D) \! - \! \{ r \}$ \\ with $|V'| \le \ell$, \\ see Prop.~\ref{prop:ivertexcoloring-edge}}} & \multirow{5}{*}{\parbox{80pt}{\centering $D$ is $r$-rooted \\ $(k + \ell)$-vertex- \\ connected, \\ see Prop.~\ref{prop:ivertexcoloring-vertex}}} \\
   & & & & & & \\
   & & & & & & \\
   & & & & & & \\
   & & & & & & \\ \specialrule{1.25pt}{0pt}{0pt}
   \multirow{4}{*}{\parbox{80pt}{\centering Minimum number \\ of colors in a $\ldots$ \\ coloring of $D$ \\ (if exists)}} & \multicolumn{2}{c !{\vrule width 1.25pt}}{\multirow{2}{*}{\parbox{150pt}{\centering $k=1$: \\ $\ell+1$, see Thm.~\ref{thm:rooted_mincolors}}}} & \multicolumn{2}{c !{\vrule width 1.25pt}}{\multirow{4}{*}{\parbox{80pt}{\centering 1, \\ see Cor.~\ref{cor:vertexmincolor}}}} & \multicolumn{2}{c !{\vrule width 1.25pt}}{\multirow{4}{*}{\parbox{80pt}{\centering NP-complete, \\ see Thm.~\ref{thm:rooted_iv_mincolor}}}} \\
   & \multicolumn{2}{c !{\vrule width 1.25pt}}{} & \multicolumn{2}{c !{\vrule width 1.25pt}}{} & \multicolumn{2}{c !{\vrule width 1.25pt}}{} \\ \cline{2-3}
   & \multicolumn{2}{c !{\vrule width 1.25pt}}{\multirow{2}{*}{\parbox{150pt}{\centering $k \ge 2$: \\ open}}} & \multicolumn{2}{c !{\vrule width 1.25pt}}{} & \multicolumn{2}{c !{\vrule width 1.25pt}}{} \\
   & \multicolumn{2}{c !{\vrule width 1.25pt}}{} & \multicolumn{2}{c !{\vrule width 1.25pt}}{} & \multicolumn{2}{c !{\vrule width 1.25pt}}{} \\ \specialrule{1.25pt}{0pt}{0pt}
 \end{tabular}
 } \vspace{-8pt}
 \captionof{table}{Summary of the results in Section~\ref{section:colorings}.}
\label{table:summary_results}
\end{center}
}
\end{table}

\section{Color-avoiding connected orientations} \label{section:orientations}

In this section, we study the problem of deciding whether the edges of a given edge-colored graph can be oriented so that the obtained digraph is arc-1-color-avoiding strongly or rooted 1-connected. A famous result for the uncolored, strongly connected version is from Nash-Williams~\cite{article:strongly_k-arc-connected_orientations}, saying that a graph $G$ has a strongly $k$-arc-connected orientation if and only if $G$ is $2k$-edge-connected (where $k \in \mathbb{Z}_+$), but the case $k=1$ was previously settled by Robbins~\cite{article:robbins}. Thomassen~\cite{article:thomassen} proved that a graph $G$ has a strongly 2-vertex-connected orientation if and only if $G$ is 4-edge-connected and $G-v$ is 2-edge-connected for every vertex $v \in V(G)$. However, it is not known whether a strongly 2-vertex-connected orientation can be found in polynomial time. Durand de Gevigney showed that for any integer $k \ge 3$, the problem of deciding whether a given graph has a strongly $k$-vertex-connected orientation is NP-complete~\cite{article:durand_de_gevigney}. On the other hand, a strongly $k$-arc-connected orientation can be found in polynomial time, see for example~\cites{frank3,gabow2,iwata}.

It is not difficult to see that if an edge-colored graph has an arc-1-color-avoiding strongly 1-connected orientation, then the graph must be edge-1-color-avoiding 2-edge-connected. However, there exist edge-1-color-avoiding 2-edge-connected graphs that do not have an arc-1-color-avoiding strongly 1-connected orientation; for an example, see the first graph in Figure~\ref{figure:ECA_Strong_Orientation_Proof}.

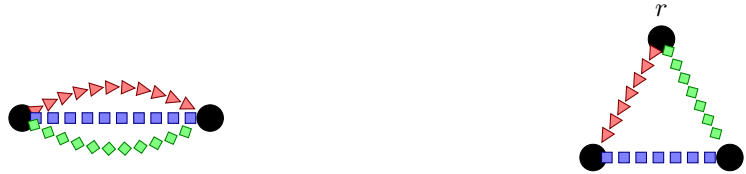
\begin{figure}[H]
 \centering
 \begin{tikzpicture}[scale=2]
 \tikzstyle{vertex}=[draw,circle,fill,minimum size=10,inner sep=0]
 \tikzset{paint/.style={draw=#1!50!black, fill=#1!50}, decorate with/.style = {decorate, decoration={shape backgrounds, shape=#1, shape size = 5pt, shape sep = 6pt}}}
 \tikzstyle{edge_red}=[draw, decorate with = isosceles triangle, paint = red]
 \tikzstyle{edge_blue}=[draw, decorate with = rectangle, decoration = {shape size = 4pt, shape sep = 6.5pt}, paint = blue]
 \tikzstyle{edge_green}=[draw, decorate with = diamond, paint = green]
 \tikzstyle{edge_yellow}=[draw, decorate with = star, paint = yellow]
 
 \begin{scope}[shift={(4.25,0)}, scale=0.75]
  \node[vertex] (u1) at (90:0.7) [label=above: $r$] {};
  \node[vertex] (u2) at (210:0.7)  {};
  \node[vertex] (u3) at (330:0.7)  {};

  \draw[edge_red] (u1) -- (u2);
  \draw[edge_blue] (u2) -- (u3);
  \draw[edge_green] (u1) -- (u3);
 \end{scope}
  \node[vertex] (a1) at (0,0) {};
  \node[vertex] (b1) at (1.25,0) {};
  
  \draw[edge_red] (a1) [bend left=30] to (b1);
  \draw[edge_blue] (a1) -- (b1);
  \draw[edge_green] (a1) [bend right=30] to (b1);
 \end{tikzpicture}
 \caption{The first graph is an edge-1-color-avoiding 2-edge-connected graph that does not have an arc-1-color-avoiding strongly 1-connected orientation: in each orientation, either all three edges are oriented in the same direction -- which is clearly not an arc-1-color-avoiding strongly 1-connected orientation --, or one of them, let us say the red edge (denoted by triangles), is oriented in the opposite direction than the other two -- in which case, after the removal of the red arc(s), the remaining digraph is not strongly connected anymore. The second graph is an edge-1-color-avoiding graph that does not have an arc-1-color-avoiding $r$-rooted 1-connected orientation: without loss of generality, we can assume that the edges incident to $r$ are oriented away from $r$ (otherwise we can re-orient these edges), but neither orientation of the remaining edge yields an arc-1-color-avoiding $r$-rooted 1-connected orientation.}
 \label{figure:ECA_Strong_Orientation_Proof}
\end{figure}

It is not difficult to see that an uncolored graph $G$ graph has a rooted connected orientation if and only if $G$ is connected: clearly, if $G$ has a rooted connected orientation then it must be connected, and if $G$ is connected then it has a spanning tree, and by picking an arbitrary vertex of $G$ and orienting every edge of this spanning tree away from $r$ and orienting the remaining edges arbitrarily, we get a rooted connected orientation. Combining the previously mentioned Tutte~\cite{article:Tutte_edge_disjoint_spanning_trees} and Nash-William~\cite{article:NashWilliams_edge_disjoint_spanning_trees} with Theorem~\ref{thm:rootededmonds}, a graph has a rooted $k$-arc-connected orientation (where $k \in \mathbb{Z}_+$) if and only if $|F| \ge k \cdot \big( c(G-F) - 1 \big)$ for any $F \subseteq E(G)$. Moreover, an actual rooted $k$-arc-connected orientation and a rooted 2-vertex-connected orientation can be found in polynomial time, see for example~\cites{frank2,gabow} and~\cites{article:whitty, thesis:plehn, article:cheriyan}, respectively. On the other hand, Durand de Gevigney showed that for any integer $k \ge 3$, the problem of deciding whether a given graph has a rooted $k$-vertex-connected orientation is NP-complete~\cite{thesis:de_gevigney}.

Clearly, if an edge-colored graph has an arc-1-color-avoiding rooted 1-connected orientation, then the graph must be edge-1-color-avoiding 1-connected. However, there exist edge-1-color-avoiding 1-connected graphs that do not have an arc-1-color-avoiding rooted 1-connected orientation; for an example, see the second graph in Figure~\ref{figure:ECA_Strong_Orientation_Proof}.

\subsection{Arc-color-avoiding connected orientations}

First, let us consider arc-color-avoiding strongly connected orientations. We prove that it is NP-complete to decide whether an edge-colored graph admits an arc-color-avoiding strongly connected orientation. 

\begin{theorem} \label{thm:ECA_orientability_strong}
 Given an edge-colored graph $G$ and numbers $k, \ell \in \mathbb{Z}_+$, deciding whether $G$ has an arc-$\ell$-color-avoiding strongly $k$-arc- or $k$-vertex-connected orientation is NP-complete. This problem remains NP-complete even if $k = \ell = 1$.
\end{theorem}
\begin{proof}
 This problem is obviously in NP, and to show that it is NP-hard even for the case $k = \ell = 1$, we reduce the NP-complete problem \textsc{Positive-Linear-NAE-3SAT-Exact4}~\cite{Darmann2020Simple} to it. Our goal in the problem \textsc{Positive-Linear-NAE-3SAT-Exact4} is to decide whether a Boolean formula in conjunctive normal form where each clause contains 3 literals, each literal is not negated and appears in exactly four distinct clauses, each pair of distinct clauses shares at most one variable can be satisfied in such a way that for each clause, not all literals have the same truth value.
 
 First, we assign a list of colors to each edge, and then show how to modify the construction so that every edge is assigned a single color.

 Let $\varphi$ be an an instance of \textsc{Positive-Linear-NAE-3SAT-Exact4} with $n$ variables $x_1, \ldots, x_n$ and with $m \ge 11$ clauses $C_1, \ldots, C_m$. Let $G$ be defined as follows. Let
 \[ V(G) \colonequals \{ s, s' \} \cup \big\{ u_i, v_i \bigm| i \in[n] \big\} \cup \big\{ w_j \bigm| j \in [m] \big\} . \]
 For each $i \in [n]$, connect the vertices $u_i$ and $v_i$ with an edge whose list of colors is $\emptyset$, and connect the vertices $v_i$ and $s'$ with two parallel edges, both having a list of colors $\emptyset$. For each $j \in [m]$, connect the vertices $s$ and $w_j$ with two parallel edges, both having a list of colors $\{j\}$.
 
 For each $i \in [n]$, let $A_i$ denote the set of indices of the clauses containing the variable $x_i$.
 For each $i \in [n]$ and $j \in [m]$, if the variable $x_i$ is contained in the clause $C_j$, then connect $w_j$ and $u_i$ with two parallel edges, both having a list of colors $A_i \setminus \{j\}$. Clearly, $G$ can be constructed in polynomial time. For an example, see Figure~\ref{fig:ECA_orientability_1}.

 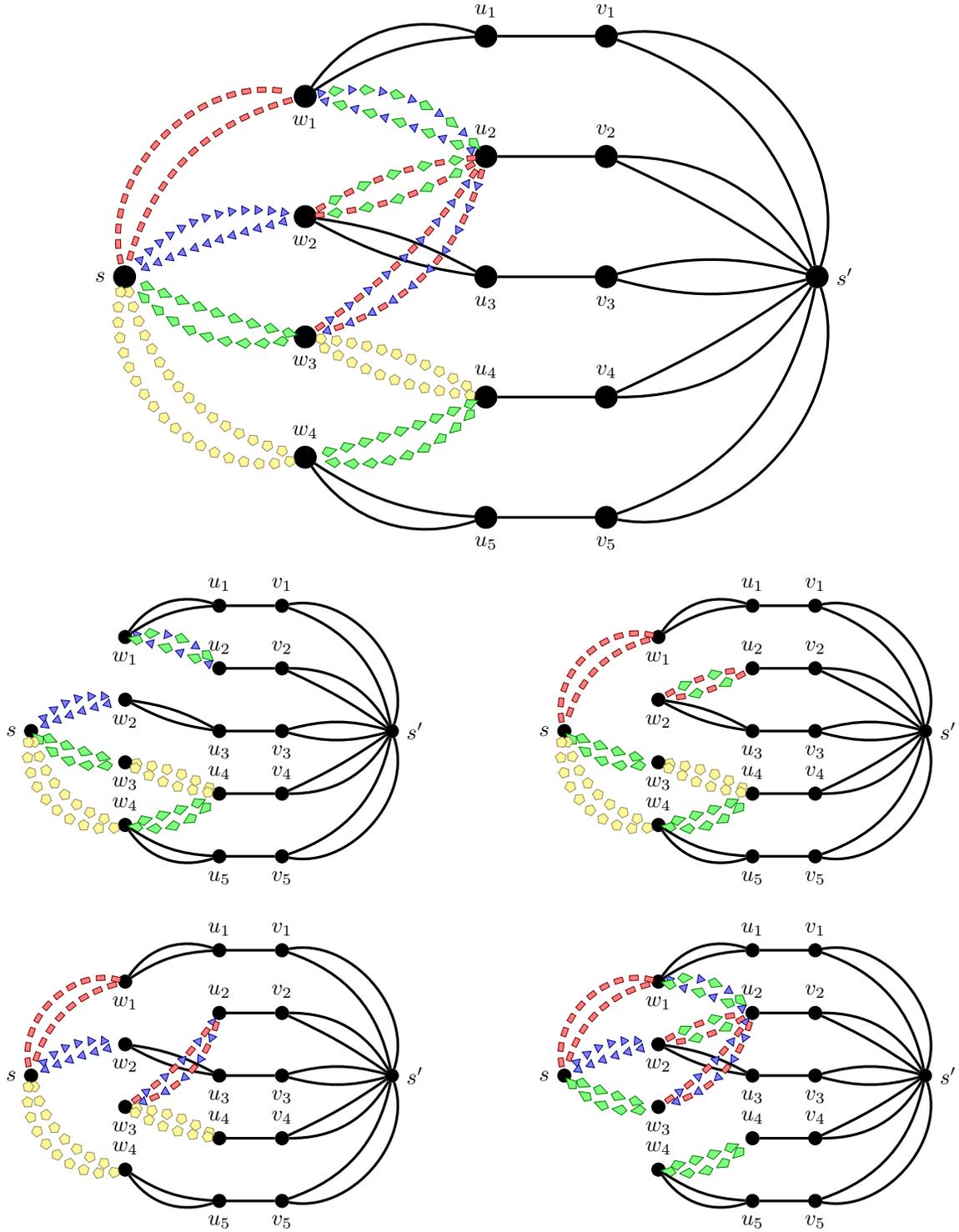
\begin{figure}[p]
 \centering
 \begin{tikzpicture}
    % First row: Original graph
    \begin{scope}[scale=0.96, shift={(0,0)}]
        % Original graph code
        \tikzstyle{vertex}=[draw,circle,fill,minimum size=10,inner sep=0]
        \tikzset{paint/.style={draw=#1!50!black, fill=#1!50}, decorate with/.style = {decorate, decoration={shape backgrounds, shape=#1, shape size = 5pt, shape sep = 6pt}}}

        \tikzset{edge_red/.style={postaction={decorate, decoration={markings, mark=between positions 3pt and 1-2pt step 7pt with {\draw[red!50!black,thin,fill=red!50] (30:0.08) -- (150:0.08) -- (210:0.08) -- (330:0.08) -- (30:0.08);}}}}}

        \tikzset{edge_blue/.style={postaction={decorate, decoration={markings, mark=between positions 5pt and 1-2pt step 6pt with {\draw[blue!50!black,thin,fill=blue!50] (0:0.08) -- (120:0.08) -- (240:0.08) -- (0:0.08);}}}}}
 
        \tikzset{edge_green/.style={postaction={decorate, decoration={markings, mark=between positions 5pt and 1-2pt step 8.5pt with {\draw[green!50!black,thin,fill=green!50] (0:0.16) -- (90:0.08) -- (180:0.08) -- (270:0.08) -- (0:0.16);}}}}}

        \tikzset{edge_yellow/.style={postaction={decorate, decoration={markings, mark=between positions 2pt and 1-2pt step 7pt with {\draw[yellow!50!black,thin,fill=yellow!50] (0:0.09) -- (72:0.09) -- (144:0.09) -- (216:0.09) -- (288:0.09) -- (0:0.09);}}}}}

        \tikzset{edge_red_and_blue/.style={postaction={decorate, decoration={markings, mark=between positions 3pt and 1-2pt step 11pt with {\draw[red!50!black,thin,fill=red!50] (30:0.08) -- (150:0.08) -- (210:0.08) -- (330:0.08) -- (30:0.08);}}}, postaction={decorate, decoration={markings, mark=between positions 8.5pt and 1-2pt step 11pt with {\draw[blue!50!black,thin,fill=blue!50] (0:0.08) -- (120:0.08) -- (240:0.08) -- (0:0.08);}}}}}

        \tikzset{edge_red_and_green/.style={postaction={decorate, decoration={markings, mark=between positions 3pt and 1-2pt step 14.5pt with {\draw[red!50!black,thin,fill=red!50] (30:0.08) -- (150:0.08) -- (210:0.08) -- (330:0.08) -- (30:0.08);}}}, postaction={decorate, decoration={markings, mark=between positions 9.5pt and 1-8pt step 14.5pt with {\draw[green!50!black,thin,fill=green!50] (0:0.16) -- (90:0.08) -- (180:0.08) -- (270:0.08) -- (0:0.16);}}}}}

        \tikzset{edge_blue_and_green/.style={postaction={decorate, decoration={markings, mark=between positions 2pt and 1-2pt step 14pt with {\draw[blue!50!black,thin,fill=blue!50] (0:0.08) -- (120:0.08) -- (240:0.08) -- (0:0.08);}}}, postaction={decorate, decoration={markings, mark=between positions 8pt and 1-2pt step 14pt with {\draw[green!50!black,thin,fill=green!50] (0:0.16) -- (90:0.08) -- (180:0.08) -- (270:0.08) -- (0:0.16);}}}}}

        % Vertices
        \node[vertex, black, label={left:$s$}] (B) at (-6, 0) {};
        \node[vertex, black, label={below:$w_1$}] (C1) at (-3, 3) {};
        \node[vertex, black, label={below:$w_2$}] (C2) at (-3, 1) {};
        \node[vertex, black, label={below:$w_3$}] (C3) at (-3, -1) {};
        \node[vertex, black, label={above:$w_4$}] (C4) at (-3, -3) {};
        \node[vertex, black, label={above:$u_1$}] (u1) at (0, 4) {};
        \node[vertex, black, label={above:$v_1$}] (v1) at (2, 4) {};
        \node[vertex, black, label={above:$u_2$}] (u2) at (0, 2) {};
        \node[vertex, black, label={above:$v_2$}] (v2) at (2, 2) {};
        \node[vertex, black, label={below:$u_3$}] (u3) at (0, 0) {};
        \node[vertex, black, label={below:$v_3$}] (v3) at (2, 0) {};
        \node[vertex, black, label={above:$u_4$}] (u4) at (0, -2) {};
        \node[vertex, black, label={above:$v_4$}] (v4) at (2, -2) {};
        \node[vertex, black, label={below:$u_5$}] (u5) at (0, -4) {};
        \node[vertex, black, label={below:$v_5$}] (v5) at (2, -4) {};
        \node[vertex, black, label={right:$s'$}] (Bprime) at (5.5, 0) {};

        % Edges
        % uv
        \draw[very thick] (u1) to (v1);
        \draw[very thick] (v2) to (u2);
        \draw[very thick] (u3) to (v3);
        \draw[very thick] (u4) to (v4);
        \draw[very thick] (v5) to (u5);

        % BC1, BC2, BC3
        \path[edge_red] (B) to[bend right=-60] (C1);
        \path[edge_red] (C1) to[bend right=30] (B);
        \path[edge_blue] (B) to[bend right=-30] (C2);
        \path[edge_blue] (C2) to[bend right=5] (B);
        \path[edge_green] (B) to[bend right=5] (C3);
        \path[edge_green] (C3) to[bend right=-30] (B);
        \path[edge_yellow] (B) to[bend right=30] (C4);
        \path[edge_yellow] (C4) to[bend right=-60] (B);
        
        % Parallel edges from C1 
        \draw[very thick] (C1) to[bend right=-37] (u1);
        \draw[very thick] (u1) to[bend right=15] (C1);
        \path[edge_blue_and_green] (C1) to[bend right=-30] (u2);
        \path[edge_blue_and_green] (u2) to[bend right=5] (C1);
        
        % Parallel edges from C2
        \path[edge_red_and_green] (C2) to[bend right=-12] (u2);
        \path[edge_red_and_green] (u2) to[bend right=-12] (C2);
        \draw[very thick] (C2) to[bend right=-10] (u3);
        \draw[very thick] (u3) to[bend right=-10] (C2);
        
        % Parallel edges from C3
        \path[edge_red_and_blue] (C3) to[bend right=10] (u2);
        \path[edge_red_and_blue] (u2) to[bend right=-30] (C3);
        \path[edge_yellow] (C3) to[bend right=-12] (u4);
        \path[edge_yellow] (u4) to[bend right=-12] (C3);

        % Parallel edges from C4
        \path[edge_green] (C4) to[bend right=5] (u4);
        \path[edge_green] (u4) to[bend right=-30] (C4);
        \draw[very thick] (C4) to[bend right=15] (u5);
        \draw[very thick] (u5) to[bend right=-37] (C4);
        
        % Parallel edges to B'
        \draw[very thick] (v1) to[bend right=-60] (Bprime);
        \draw[very thick] (Bprime) to[bend right=30] (v1);
        \draw[very thick] (v2) to[bend right=-30] (Bprime);
        \draw[very thick] (Bprime) to[bend right=5] (v2);
        \draw[very thick] (v3) to[bend right=-15] (Bprime);
        \draw[very thick] (Bprime) to[bend right=-15] (v3);
        \draw[very thick] (v4) to[bend right=5] (Bprime);
        \draw[very thick] (Bprime) to[bend right=-30] (v4);
        \draw[very thick] (v5) to[bend right=30] (Bprime);
        \draw[very thick] (Bprime) to[bend right=-60] (v5);
    \end{scope}

    % Second row: Smaller graphs
    \begin{scope}[shift={(0, -7.25)}, scale=0.5]
        % First smaller graph
        \tikzstyle{vertex}=[draw,circle,fill,minimum size=6,inner sep=0]
        \tikzset{paint/.style={draw=#1!50!black, fill=#1!50}, decorate with/.style = {decorate, decoration={shape backgrounds, shape=#1, shape size = 5pt, shape sep = 6pt}}}

        \tikzset{edge_red/.style={postaction={decorate, decoration={markings, mark=between positions 3pt and 1-2pt step 7pt with {\draw[red!50!black,thin,fill=red!50] (30:0.08) -- (150:0.08) -- (210:0.08) -- (330:0.08) -- (30:0.08);}}}}}

        \tikzset{edge_blue/.style={postaction={decorate, decoration={markings, mark=between positions 5pt and 1-2pt step 6pt with {\draw[blue!50!black,thin,fill=blue!50] (0:0.08) -- (120:0.08) -- (240:0.08) -- (0:0.08);}}}}}
 
        \tikzset{edge_green/.style={postaction={decorate, decoration={markings, mark=between positions 5pt and 1-2pt step 8.5pt with {\draw[green!50!black,thin,fill=green!50] (0:0.16) -- (90:0.08) -- (180:0.08) -- (270:0.08) -- (0:0.16);}}}}}

        \tikzset{edge_yellow/.style={postaction={decorate, decoration={markings, mark=between positions 2pt and 1-2pt step 7pt with {\draw[yellow!50!black,thin,fill=yellow!50] (0:0.09) -- (72:0.09) -- (144:0.09) -- (216:0.09) -- (288:0.09) -- (0:0.09);}}}}}

        \tikzset{edge_red_and_blue/.style={postaction={decorate, decoration={markings, mark=between positions 3pt and 1-2pt step 11pt with {\draw[red!50!black,thin,fill=red!50] (30:0.08) -- (150:0.08) -- (210:0.08) -- (330:0.08) -- (30:0.08);}}}, postaction={decorate, decoration={markings, mark=between positions 8.5pt and 1-2pt step 11pt with {\draw[blue!50!black,thin,fill=blue!50] (0:0.08) -- (120:0.08) -- (240:0.08) -- (0:0.08);}}}}}

        \tikzset{edge_red_and_green/.style={postaction={decorate, decoration={markings, mark=between positions 3pt and 1-2pt step 14.5pt with {\draw[red!50!black,thin,fill=red!50] (30:0.08) -- (150:0.08) -- (210:0.08) -- (330:0.08) -- (30:0.08);}}}, postaction={decorate, decoration={markings, mark=between positions 9.5pt and 1-8pt step 14.5pt with {\draw[green!50!black,thin,fill=green!50] (0:0.16) -- (90:0.08) -- (180:0.08) -- (270:0.08) -- (0:0.16);}}}}}

        \tikzset{edge_blue_and_green/.style={postaction={decorate, decoration={markings, mark=between positions 2pt and 1-2pt step 14pt with {\draw[blue!50!black,thin,fill=blue!50] (0:0.08) -- (120:0.08) -- (240:0.08) -- (0:0.08);}}}, postaction={decorate, decoration={markings, mark=between positions 8pt and 1-2pt step 14pt with {\draw[green!50!black,thin,fill=green!50] (0:0.16) -- (90:0.08) -- (180:0.08) -- (270:0.08) -- (0:0.16);}}}}}

        \begin{scope}[shift={(-8.5, 0)}]
        % Vertices
        \node[vertex, black, label={left:$s$}] (B) at (-6, 0) {};
        \node[vertex, black, label={below:$w_1$}] (C1) at (-3, 3) {};
        \node[vertex, black, label={below:$w_2$}] (C2) at (-3, 1) {};
        \node[vertex, black, label={below:$w_3$}] (C3) at (-3, -1) {};
        \node[vertex, black, label={above:$w_4$}] (C4) at (-3, -3) {};
        \node[vertex, black, label={above:$u_1$}] (u1) at (0, 4) {};
        \node[vertex, black, label={above:$v_1$}] (v1) at (2, 4) {};
        \node[vertex, black, label={above:$u_2$}] (u2) at (0, 2) {};
        \node[vertex, black, label={above:$v_2$}] (v2) at (2, 2) {};
        \node[vertex, black, label={below:$u_3$}] (u3) at (0, 0) {};
        \node[vertex, black, label={below:$v_3$}] (v3) at (2, 0) {};
        \node[vertex, black, label={above:$u_4$}] (u4) at (0, -2) {};
        \node[vertex, black, label={above:$v_4$}] (v4) at (2, -2) {};
        \node[vertex, black, label={below:$u_5$}] (u5) at (0, -4) {};
        \node[vertex, black, label={below:$v_5$}] (v5) at (2, -4) {};
        \node[vertex, black, label={right:$s'$}] (Bprime) at (5.5, 0) {};

        % Edges
        % uv
        \draw[very thick] (u1) to (v1);
        \draw[very thick] (v2) to (u2);
        \draw[very thick] (u3) to (v3);
        \draw[very thick] (u4) to (v4);
        \draw[very thick] (v5) to (u5);

        % BC1, BC2, BC3
        \path[edge_blue] (B) to[bend right=-30] (C2);
        \path[edge_blue] (C2) to[bend right=5] (B);
        \path[edge_green] (B) to[bend right=5] (C3);
        \path[edge_green] (C3) to[bend right=-30] (B);
        \path[edge_yellow] (B) to[bend right=30] (C4);
        \path[edge_yellow] (C4) to[bend right=-60] (B);
        
        % Parallel edges from C1 
        \draw[very thick] (C1) to[bend right=-37] (u1);
        \draw[very thick] (u1) to[bend right=15] (C1);
        \path[edge_blue_and_green] (C1) to[bend right=-30] (u2);
        \path[edge_blue_and_green] (u2) to[bend right=5] (C1);
        
        % Parallel edges from C2
        \draw[very thick] (C2) to[bend right=-10] (u3);
        \draw[very thick] (u3) to[bend right=-10] (C2);
        
        % Parallel edges from C3
        \path[edge_yellow] (C3) to[bend right=-12] (u4);
        \path[edge_yellow] (u4) to[bend right=-12] (C3);

        % Parallel edges from C4
        \path[edge_green] (C4) to[bend right=5] (u4);
        \path[edge_green] (u4) to[bend right=-30] (C4);
        \draw[very thick] (C4) to[bend right=15] (u5);
        \draw[very thick] (u5) to[bend right=-37] (C4);
        
        % Parallel edges to B'
        \draw[very thick] (v1) to[bend right=-60] (Bprime);
        \draw[very thick] (Bprime) to[bend right=30] (v1);
        \draw[very thick] (v2) to[bend right=-30] (Bprime);
        \draw[very thick] (Bprime) to[bend right=5] (v2);
        \draw[very thick] (v3) to[bend right=-15] (Bprime);
        \draw[very thick] (Bprime) to[bend right=-15] (v3);
        \draw[very thick] (v4) to[bend right=5] (Bprime);
        \draw[very thick] (Bprime) to[bend right=-30] (v4);
        \draw[very thick] (v5) to[bend right=30] (Bprime);
        \draw[very thick] (Bprime) to[bend right=-60] (v5);
        \end{scope}

        \begin{scope}[shift={(8.5, 0)}]
        % Vertices
        \node[vertex, black, label={left:$s$}] (B) at (-6, 0) {};
        \node[vertex, black, label={below:$w_1$}] (C1) at (-3, 3) {};
        \node[vertex, black, label={below:$w_2$}] (C2) at (-3, 1) {};
        \node[vertex, black, label={below:$w_3$}] (C3) at (-3, -1) {};
        \node[vertex, black, label={above:$w_4$}] (C4) at (-3, -3) {};
        \node[vertex, black, label={above:$u_1$}] (u1) at (0, 4) {};
        \node[vertex, black, label={above:$v_1$}] (v1) at (2, 4) {};
        \node[vertex, black, label={above:$u_2$}] (u2) at (0, 2) {};
        \node[vertex, black, label={above:$v_2$}] (v2) at (2, 2) {};
        \node[vertex, black, label={below:$u_3$}] (u3) at (0, 0) {};
        \node[vertex, black, label={below:$v_3$}] (v3) at (2, 0) {};
        \node[vertex, black, label={above:$u_4$}] (u4) at (0, -2) {};
        \node[vertex, black, label={above:$v_4$}] (v4) at (2, -2) {};
        \node[vertex, black, label={below:$u_5$}] (u5) at (0, -4) {};
        \node[vertex, black, label={below:$v_5$}] (v5) at (2, -4) {};
        \node[vertex, black, label={right:$s'$}] (Bprime) at (5.5, 0) {};

        % Edges
        % uv
        \draw[very thick] (u1) to (v1);
        \draw[very thick] (v2) to (u2);
        \draw[very thick] (u3) to (v3);
        \draw[very thick] (u4) to (v4);
        \draw[very thick] (v5) to (u5);

        % BC1, BC2, BC3
        \path[edge_red] (B) to[bend right=-60] (C1);
        \path[edge_red] (C1) to[bend right=30] (B);
        \path[edge_green] (B) to[bend right=5] (C3);
        \path[edge_green] (C3) to[bend right=-30] (B);
        \path[edge_yellow] (B) to[bend right=30] (C4);
        \path[edge_yellow] (C4) to[bend right=-60] (B);
        
        % Parallel edges from C1 
        \draw[very thick] (C1) to[bend right=-37] (u1);
        \draw[very thick] (u1) to[bend right=15] (C1);
        
        % Parallel edges from C2
        \path[edge_red_and_green] (C2) to[bend right=-12] (u2);
        \path[edge_red_and_green] (u2) to[bend right=-12] (C2);
        \draw[very thick] (C2) to[bend right=-10] (u3);
        \draw[very thick] (u3) to[bend right=-10] (C2);
        
        % Parallel edges from C3
        \path[edge_yellow] (C3) to[bend right=-12] (u4);
        \path[edge_yellow] (u4) to[bend right=-12] (C3);

        % Parallel edges from C4
        \path[edge_green] (C4) to[bend right=5] (u4);
        \path[edge_green] (u4) to[bend right=-30] (C4);
        \draw[very thick] (C4) to[bend right=15] (u5);
        \draw[very thick] (u5) to[bend right=-37] (C4);
        
        % Parallel edges to B'
        \draw[very thick] (v1) to[bend right=-60] (Bprime);
        \draw[very thick] (Bprime) to[bend right=30] (v1);
        \draw[very thick] (v2) to[bend right=-30] (Bprime);
        \draw[very thick] (Bprime) to[bend right=5] (v2);
        \draw[very thick] (v3) to[bend right=-15] (Bprime);
        \draw[very thick] (Bprime) to[bend right=-15] (v3);
        \draw[very thick] (v4) to[bend right=5] (Bprime);
        \draw[very thick] (Bprime) to[bend right=-30] (v4);
        \draw[very thick] (v5) to[bend right=30] (Bprime);
        \draw[very thick] (Bprime) to[bend right=-60] (v5);
        \end{scope}

        \begin{scope}[shift={(-8.5, -11)}]
        % Vertices
        \node[vertex, black, label={left:$s$}] (B) at (-6, 0) {};
        \node[vertex, black, label={below:$w_1$}] (C1) at (-3, 3) {};
        \node[vertex, black, label={below:$w_2$}] (C2) at (-3, 1) {};
        \node[vertex, black, label={below:$w_3$}] (C3) at (-3, -1) {};
        \node[vertex, black, label={above:$w_4$}] (C4) at (-3, -3) {};
        \node[vertex, black, label={above:$u_1$}] (u1) at (0, 4) {};
        \node[vertex, black, label={above:$v_1$}] (v1) at (2, 4) {};
        \node[vertex, black, label={above:$u_2$}] (u2) at (0, 2) {};
        \node[vertex, black, label={above:$v_2$}] (v2) at (2, 2) {};
        \node[vertex, black, label={below:$u_3$}] (u3) at (0, 0) {};
        \node[vertex, black, label={below:$v_3$}] (v3) at (2, 0) {};
        \node[vertex, black, label={above:$u_4$}] (u4) at (0, -2) {};
        \node[vertex, black, label={above:$v_4$}] (v4) at (2, -2) {};
        \node[vertex, black, label={below:$u_5$}] (u5) at (0, -4) {};
        \node[vertex, black, label={below:$v_5$}] (v5) at (2, -4) {};
        \node[vertex, black, label={right:$s'$}] (Bprime) at (5.5, 0) {};

        % Edges
        % uv
        \draw[very thick] (u1) to (v1);
        \draw[very thick] (v2) to (u2);
        \draw[very thick] (u3) to (v3);
        \draw[very thick] (u4) to (v4);
        \draw[very thick] (v5) to (u5);

        % BC1, BC2, BC3
        \path[edge_red] (B) to[bend right=-60] (C1);
        \path[edge_red] (C1) to[bend right=30] (B);
        \path[edge_blue] (B) to[bend right=-30] (C2);
        \path[edge_blue] (C2) to[bend right=5] (B);
        \path[edge_yellow] (B) to[bend right=30] (C4);
        \path[edge_yellow] (C4) to[bend right=-60] (B);
        
        % Parallel edges from C1 
        \draw[very thick] (C1) to[bend right=-37] (u1);
        \draw[very thick] (u1) to[bend right=15] (C1);
        
        % Parallel edges from C2
        \draw[very thick] (C2) to[bend right=-10] (u3);
        \draw[very thick] (u3) to[bend right=-10] (C2);
        
        % Parallel edges from C3
        \path[edge_red_and_blue] (C3) to[bend right=10] (u2);
        \path[edge_red_and_blue] (u2) to[bend right=-30] (C3);
        \path[edge_yellow] (C3) to[bend right=-12] (u4);
        \path[edge_yellow] (u4) to[bend right=-12] (C3);

        % Parallel edges from C4
        \draw[very thick] (C4) to[bend right=15] (u5);
        \draw[very thick] (u5) to[bend right=-37] (C4);
        
        % Parallel edges to B'
        \draw[very thick] (v1) to[bend right=-60] (Bprime);
        \draw[very thick] (Bprime) to[bend right=30] (v1);
        \draw[very thick] (v2) to[bend right=-30] (Bprime);
        \draw[very thick] (Bprime) to[bend right=5] (v2);
        \draw[very thick] (v3) to[bend right=-15] (Bprime);
        \draw[very thick] (Bprime) to[bend right=-15] (v3);
        \draw[very thick] (v4) to[bend right=5] (Bprime);
        \draw[very thick] (Bprime) to[bend right=-30] (v4);
        \draw[very thick] (v5) to[bend right=30] (Bprime);
        \draw[very thick] (Bprime) to[bend right=-60] (v5);
        \end{scope}

        \begin{scope}[shift={(8.5, -11)}]
        % Vertices
        \node[vertex, black, label={left:$s$}] (B) at (-6, 0) {};
        \node[vertex, black, label={below:$w_1$}] (C1) at (-3, 3) {};
        \node[vertex, black, label={below:$w_2$}] (C2) at (-3, 1) {};
        \node[vertex, black, label={below:$w_3$}] (C3) at (-3, -1) {};
        \node[vertex, black, label={above:$w_4$}] (C4) at (-3, -3) {};
        \node[vertex, black, label={above:$u_1$}] (u1) at (0, 4) {};
        \node[vertex, black, label={above:$v_1$}] (v1) at (2, 4) {};
        \node[vertex, black, label={above:$u_2$}] (u2) at (0, 2) {};
        \node[vertex, black, label={above:$v_2$}] (v2) at (2, 2) {};
        \node[vertex, black, label={below:$u_3$}] (u3) at (0, 0) {};
        \node[vertex, black, label={below:$v_3$}] (v3) at (2, 0) {};
        \node[vertex, black, label={above:$u_4$}] (u4) at (0, -2) {};
        \node[vertex, black, label={above:$v_4$}] (v4) at (2, -2) {};
        \node[vertex, black, label={below:$u_5$}] (u5) at (0, -4) {};
        \node[vertex, black, label={below:$v_5$}] (v5) at (2, -4) {};
        \node[vertex, black, label={right:$s'$}] (Bprime) at (5.5, 0) {};

        % Edges
        % uv
        \draw[very thick] (u1) to (v1);
        \draw[very thick] (v2) to (u2);
        \draw[very thick] (u3) to (v3);
        \draw[very thick] (u4) to (v4);
        \draw[very thick] (v5) to (u5);

        % BC1, BC2, BC3
        \path[edge_red] (B) to[bend right=-60] (C1);
        \path[edge_red] (C1) to[bend right=30] (B);
        \path[edge_blue] (B) to[bend right=-30] (C2);
        \path[edge_blue] (C2) to[bend right=5] (B);
        \path[edge_green] (B) to[bend right=5] (C3);
        \path[edge_green] (C3) to[bend right=-30] (B);
        
        % Parallel edges from C1 
        \draw[very thick] (C1) to[bend right=-37] (u1);
        \draw[very thick] (u1) to[bend right=15] (C1);
        \path[edge_blue_and_green] (C1) to[bend right=-30] (u2);
        \path[edge_blue_and_green] (u2) to[bend right=5] (C1);
        
        % Parallel edges from C2
        \path[edge_red_and_green] (C2) to[bend right=-12] (u2);
        \path[edge_red_and_green] (u2) to[bend right=-12] (C2);
        \draw[very thick] (C2) to[bend right=-10] (u3);
        \draw[very thick] (u3) to[bend right=-10] (C2);
        
        % Parallel edges from C3
        \path[edge_red_and_blue] (C3) to[bend right=10] (u2);
        \path[edge_red_and_blue] (u2) to[bend right=-30] (C3);

        % Parallel edges from C4
        \path[edge_green] (C4) to[bend right=5] (u4);
        \path[edge_green] (u4) to[bend right=-30] (C4);
        \draw[very thick] (C4) to[bend right=15] (u5);
        \draw[very thick] (u5) to[bend right=-37] (C4);
        
        % Parallel edges to B'
        \draw[very thick] (v1) to[bend right=-60] (Bprime);
        \draw[very thick] (Bprime) to[bend right=30] (v1);
        \draw[very thick] (v2) to[bend right=-30] (Bprime);
        \draw[very thick] (Bprime) to[bend right=5] (v2);
        \draw[very thick] (v3) to[bend right=-15] (Bprime);
        \draw[very thick] (Bprime) to[bend right=-15] (v3);
        \draw[very thick] (v4) to[bend right=5] (Bprime);
        \draw[very thick] (Bprime) to[bend right=-30] (v4);
        \draw[very thick] (v5) to[bend right=30] (Bprime);
        \draw[very thick] (Bprime) to[bend right=-60] (v5);
        \end{scope}
    \end{scope}
 \end{tikzpicture}
 \caption{An example for the constructions used in the proof of Theorem~\ref{thm:ECA_orientability_strong}. In the first row, we can see the graph~$G$ constructed from the Boolean formula $\varphi = (x_1 \lor x_2) \land (x_2 \lor x_3) \land (x_2 \lor x_4) \land (x_4 \lor x_5)$. In the second and third rows, there are the graphs obtained from $G$ by deleting the first, the second, the third, and the fourth color, respectively. The colors red, blue, green, and yellow are denoted by (red) rectangles, (blue) triangles, (green) deltoids, and (yellow) pentagons respectively. The arcs marked with alternating forms are assigned multiple colors. The arcs drawn with a single line are assigned no colors.}
 \label{fig:ECA_orientability_1}
\end{figure}

 Now we show that $G$ admits an arc-1-color-avoiding strongly 1-connected orientation if and only if $\varphi$ is satisfiable in such a way that for each clause, not all literals have the same truth value.

 \medskip
 
 Assume first that $\varphi$ is satisfiable in such a way that for each clause, not all literals have the same truth value. 
 For each $i \in [n]$, let us orient the edge $u_i v_i$ from $u_i$ to $v_i$ if $x_i = \texttt{true}$, and let us orient this edge from $v_i$ to $u_i$ if $x_i = \texttt{false}$. Finally, orient all the parallel edges in opposite directions.
 
 Now we prove that with this orientation, $G$ is arc-1-color-avoiding strongly 1-connected. Note that it is enough to show that for any $j \in [m]$, after removing the color $j$, there still exists a dipath from $s$ to any other vertex and a dipath from any vertex to $s$. Let $j \in [m]$ be an arbitrary color, and consider the digraph obtained after removing the color $j$.

 Since there are at least 11 clauses and every variable is contained in at most 4 clauses, there exists $j' \in [m] \setminus \{ j \}$ such that the clauses $C_{j'}$ and $C_j$ do not share any variables. Let $i'_1, i'_2 \in [n]$ be indices such that $x_{i'_1}$ and $x_{i'_2}$ are a true and a false variable, respectively, in the clause $C_{j'}$. Then $s w_{j'} u_{i'_1} v_{i'_1} s'$ and $s' v_{i'_2} u_{i'_2} w_{j'} s$ are $s$-$s'$ and $s'$-$s$ dipath in the remaining digraph.
 
 Clearly, for any $j'' \in [m] \setminus \{ j \}$, the arcs $sw_{j''}$ and $w_{j''}s$ ensure the existence of dipaths between $s$ and $w_{j''}$ in both directions in the remaining digraph.
 
 Let $i_1, i_2 \in [n]$ be indices such that $x_{i_1}$ and $x_{i_2}$ are a true and a false variable, respectively, in the clause $C_j$. Then $w_j u_{i_1} v_{i_1} s'$ and $s' v_{i_2} u_{i_2} w_j$ are $w_j$-$s'$- and $s'$-$w_j$ dipath, respectively, which can be extended with an $s'$-$s$ or an $s$-$s'$ dipath to a $w_j$-$s$ and a $s$-$w_j$ dipath in the remaining digraph.
 
 Let $i \in [n]$ be arbitrary and let $j''' \in [m]$ be such an index for which the variable $x_i$ is contained in the clause $C_{j'''}$. If $x_i$ is not contained in $C_j$, then the arcs $w_{j'''} u_i$ and $u_i w_{j'''}$ can be extended with an $s$-$w_{j'''}$ or an $w_{j'''}$-$s$ dipath to an $s$-$u_i$ and an $u_i$-$s$ dipath in the remaining graph. If $x_i$ is contained in $C_j$, then the arcs $w_j u_i$ and $u_i w_j$ can be extended with an $s$-$w_j$ or an $w_j$-$s$ dipath to an $s$-$u_i$ and an $u_i$-$s$ dipath in the remaining graph.

 For any $i \in [n]$, the arcs $v_i s'$ and $s' v_i$ can be extended with an $s'$-$s$ or an $s$-$s'$ dipath to a $v_i$-$s$ and a $s$-$v_i$ dipath in the remaining digraph.
 
 Thus, with the given orientation, $G$ is indeed arc-1-color-avoiding strongly 1-connected. For an example, see Figure~\ref{fig:ECA_orientability_1_b}.

 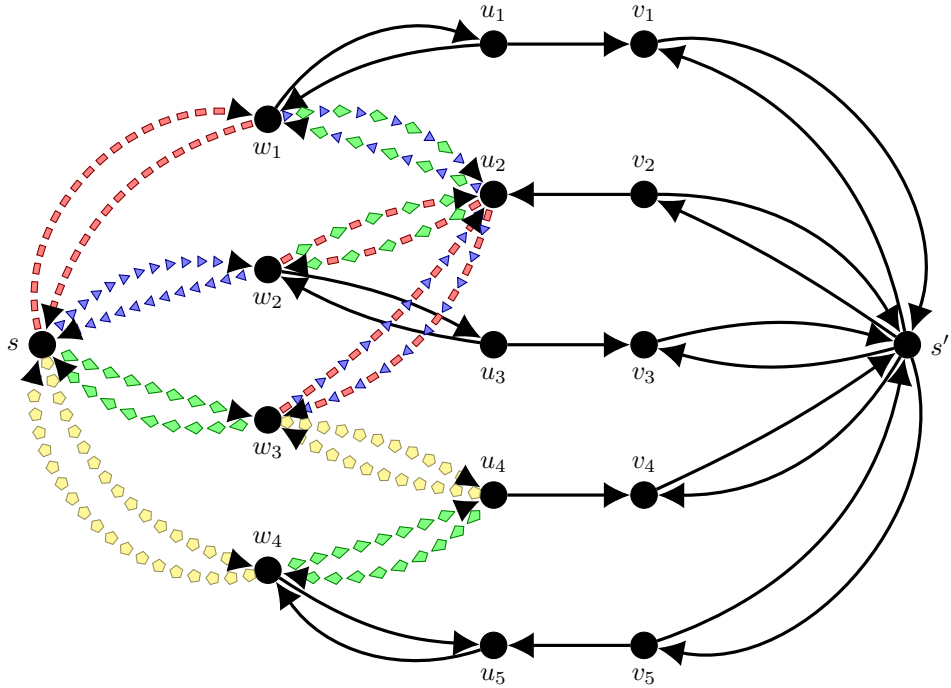
\begin{figure}[H]
 \centering
 \begin{tikzpicture}
    \begin{scope}[shift={(0,0)}]
        \tikzstyle{vertex}=[draw,circle,fill,minimum size=10,inner sep=0]
        \tikzset{paint/.style={draw=#1!50!black, fill=#1!50}, decorate with/.style = {decorate, decoration={shape backgrounds, shape=#1, shape size = 5pt, shape sep = 6pt}}}

        \tikzset{edge_red/.style={postaction={decorate, decoration={markings, mark=between positions 3pt and 1-8pt step 7pt with {\draw[red!50!black,thin,fill=red!50] (30:0.08) -- (150:0.08) -- (210:0.08) -- (330:0.08) -- (30:0.08);}}}}}

        \tikzset{edge_blue/.style={postaction={decorate, decoration={markings, mark=between positions 5pt and 1-8pt step 6pt with {\draw[blue!50!black,thin,fill=blue!50] (0:0.08) -- (120:0.08) -- (240:0.08) -- (0:0.08);}}}}}
 
        \tikzset{edge_green/.style={postaction={decorate, decoration={markings, mark=between positions 5pt and 1-8pt step 8.5pt with {\draw[green!50!black,thin,fill=green!50] (0:0.16) -- (90:0.08) -- (180:0.08) -- (270:0.08) -- (0:0.16);}}}}}

        \tikzset{edge_yellow/.style={postaction={decorate, decoration={markings, mark=between positions 2pt and 1-8pt step 7pt with {\draw[yellow!50!black,thin,fill=yellow!50] (0:0.09) -- (72:0.09) -- (144:0.09) -- (216:0.09) -- (288:0.09) -- (0:0.09);}}}}}

        \tikzset{edge_red_and_blue/.style={postaction={decorate, decoration={markings, mark=between positions 3pt and 1-8pt step 11pt with {\draw[red!50!black,thin,fill=red!50] (30:0.08) -- (150:0.08) -- (210:0.08) -- (330:0.08) -- (30:0.08);}}}, postaction={decorate, decoration={markings, mark=between positions 8.5pt and 1-2pt step 11pt with {\draw[blue!50!black,thin,fill=blue!50] (0:0.08) -- (120:0.08) -- (240:0.08) -- (0:0.08);}}}}}

        \tikzset{edge_red_and_green/.style={postaction={decorate, decoration={markings, mark=between positions 3pt and 1-8pt step 14.5pt with {\draw[red!50!black,thin,fill=red!50] (30:0.08) -- (150:0.08) -- (210:0.08) -- (330:0.08) -- (30:0.08);}}}, postaction={decorate, decoration={markings, mark=between positions 9.5pt and 1-8pt step 14.5pt with {\draw[green!50!black,thin,fill=green!50] (0:0.16) -- (90:0.08) -- (180:0.08) -- (270:0.08) -- (0:0.16);}}}}}

        \tikzset{edge_blue_and_green/.style={postaction={decorate, decoration={markings, mark=between positions 2pt and 1-8pt step 14pt with {\draw[blue!50!black,thin,fill=blue!50] (0:0.08) -- (120:0.08) -- (240:0.08) -- (0:0.08);}}}, postaction={decorate, decoration={markings, mark=between positions 8pt and 1-8pt step 14pt with {\draw[green!50!black,thin,fill=green!50] (0:0.16) -- (90:0.08) -- (180:0.08) -- (270:0.08) -- (0:0.16);}}}}}
        
        \node[vertex, black, label={left:$s$}] (B) at (-6, 0) {};
        \node[vertex, black, label={below:$w_1$}] (C1) at (-3, 3) {};
        \node[vertex, black, label={below:$w_2$}] (C2) at (-3, 1) {};
        \node[vertex, black, label={below:$w_3$}] (C3) at (-3, -1) {};
        \node[vertex, black, label={above:$w_4$}] (C4) at (-3, -3) {};
        \node[vertex, black, label={above:$u_1$}] (u1) at (0, 4) {};
        \node[vertex, black, label={above:$v_1$}] (v1) at (2, 4) {};
        \node[vertex, black, label={above:$u_2$}] (u2) at (0, 2) {};
        \node[vertex, black, label={above:$v_2$}] (v2) at (2, 2) {};
        \node[vertex, black, label={below:$u_3$}] (u3) at (0, 0) {};
        \node[vertex, black, label={below:$v_3$}] (v3) at (2, 0) {};
        \node[vertex, black, label={above:$u_4$}] (u4) at (0, -2) {};
        \node[vertex, black, label={above:$v_4$}] (v4) at (2, -2) {};
        \node[vertex, black, label={below:$u_5$}] (u5) at (0, -4) {};
        \node[vertex, black, label={below:$v_5$}] (v5) at (2, -4) {};
        \node[vertex, black, label={right:$s'$}] (Bprime) at (5.5, 0) {};

        \draw[very thick,->,>={LaTeX[black,length=10pt, width=10pt]}] (u1) to (v1);
        \draw[very thick,->,>={LaTeX[black,length=10pt, width=10pt]}] (v2) to (u2);
        \draw[very thick,->,>={LaTeX[black,length=10pt, width=10pt]}] (u3) to (v3);
        \draw[very thick,->,>={LaTeX[black,length=10pt, width=10pt]}] (u4) to (v4);
        \draw[very thick,->,>={LaTeX[black,length=10pt, width=10pt]}] (v5) to (u5);

        \path[edge_red] (B) to[bend right=-60] (C1);
        \draw[->,>={LaTeX[black,length=10pt, width=10pt]},draw opacity=0] (B) to[bend right=-60] (C1);
        \path[edge_red] (C1) to[bend right=30] (B);
        \draw[->,>={LaTeX[black,length=10pt, width=10pt]},draw opacity=0] (C1) to[bend right=30] (B);
        \path[edge_blue] (B) to[bend right=-30] (C2);
        \draw[->,>={LaTeX[black,length=10pt, width=10pt]},draw opacity=0] (B) to[bend right=-30] (C2);
        \path[edge_blue] (C2) to[bend right=5] (B);
        \draw[->,>={LaTeX[black,length=10pt, width=10pt]},draw opacity=0] (C2) to[bend right=5] (B);
        \path[edge_green] (B) to[bend right=5] (C3);
        \draw[->,>={LaTeX[black,length=10pt, width=10pt]},draw opacity=0] (B) to[bend right=5] (C3);
        \path[edge_green] (C3) to[bend right=-30] (B);
        \draw[->,>={LaTeX[black,length=10pt, width=10pt]},draw opacity=0] (C3) to[bend right=-30] (B);
        \path[edge_yellow] (B) to[bend right=30] (C4);
        \draw[->,>={LaTeX[black,length=10pt, width=10pt]},draw opacity=0] (B) to[bend right=30] (C4);
        \path[edge_yellow] (C4) to[bend right=-60] (B);
        \draw[->,>={LaTeX[black,length=10pt, width=10pt]},draw opacity=0] (C4) to[bend right=-60] (B);
        
        \draw[very thick,->,>={LaTeX[black,length=10pt, width=10pt]}] (C1) to[bend right=-37] (u1);
        \draw[very thick,->,>={LaTeX[black,length=10pt, width=10pt]}] (u1) to[bend right=15] (C1);
        \path[edge_blue_and_green] (C1) to[bend right=-30] (u2);
        \draw[->,>={LaTeX[black,length=10pt, width=10pt]},draw opacity=0] (C1) to[bend right=-30] (u2);
        \path[edge_blue_and_green] (u2) to[bend right=5] (C1);
        \draw[->,>={LaTeX[black,length=10pt, width=10pt]},draw opacity=0] (u2) to[bend right=5] (C1);
        
        \path[edge_red_and_green] (C2) to[bend right=-12] (u2);
        \draw[->,>={LaTeX[black,length=10pt, width=10pt]},draw opacity=0] (C2) to[bend right=-12] (u2);
        \path[edge_red_and_green] (u2) to[bend right=-12] (C2);
        \draw[->,>={LaTeX[black,length=10pt, width=10pt]},draw opacity=0] (u2) to[bend right=-12] (C2);
        \draw[very thick,->,>={LaTeX[black,length=10pt, width=10pt]}] (C2) to[bend right=-10] (u3);
        \draw[very thick,->,>={LaTeX[black,length=10pt, width=10pt]}] (u3) to[bend right=-10] (C2);
        
        \path[edge_red_and_blue] (C3) to[bend right=10] (u2);
        \draw[->,>={LaTeX[black,length=10pt, width=10pt]},draw opacity=0] (C3) to[bend right=10] (u2);
        \path[edge_red_and_blue] (u2) to[bend right=-30] (C3);
        \draw[->,>={LaTeX[black,length=10pt, width=10pt]},draw opacity=0] (u2) to[bend right=-30] (C3);
        \path[edge_yellow] (C3) to[bend right=-12] (u4);
        \draw[->,>={LaTeX[black,length=10pt, width=10pt]},draw opacity=0] (C3) to[bend right=-12] (u4);
        \path[edge_yellow] (u4) to[bend right=-12] (C3);
        \draw[->,>={LaTeX[black,length=10pt, width=10pt]},draw opacity=0] (u4) to[bend right=-12] (C3);

        \path[edge_green] (C4) to[bend right=5] (u4);
        \draw[->,>={LaTeX[black,length=10pt, width=10pt]},draw opacity=0] (C4) to[bend right=5] (u4);
        \path[edge_green] (u4) to[bend right=-30] (C4);
        \draw[->,>={LaTeX[black,length=10pt, width=10pt]},draw opacity=0] (u4) to[bend right=-30] (C4);
        \draw[very thick,->,>={LaTeX[black,length=10pt, width=10pt]}] (C4) to[bend right=15] (u5);
        \draw[very thick,->,>={LaTeX[black,length=10pt, width=10pt]}] (u5) to[bend right=-37] (C4);
        
        \draw[very thick,->,>={LaTeX[black,length=10pt, width=10pt]}] (v1) to[bend right=-60] (Bprime);
        \draw[very thick,->,>={LaTeX[black,length=10pt, width=10pt]}] (Bprime) to[bend right=30] (v1);
        \draw[very thick,->,>={LaTeX[black,length=10pt, width=10pt]}] (v2) to[bend right=-30] (Bprime);
        \draw[very thick,->,>={LaTeX[black,length=10pt, width=10pt]}] (Bprime) to[bend right=5] (v2);
        \draw[very thick,->,>={LaTeX[black,length=10pt, width=10pt]}] (v3) to[bend right=-15] (Bprime);
        \draw[very thick,->,>={LaTeX[black,length=10pt, width=10pt]}] (Bprime) to[bend right=-15] (v3);
        \draw[very thick,->,>={LaTeX[black,length=10pt, width=10pt]}] (v4) to[bend right=5] (Bprime);
        \draw[very thick,->,>={LaTeX[black,length=10pt, width=10pt]}] (Bprime) to[bend right=-30] (v4);
        \draw[very thick,->,>={LaTeX[black,length=10pt, width=10pt]}] (v5) to[bend right=30] (Bprime);
        \draw[very thick,->,>={LaTeX[black,length=10pt, width=10pt]}] (Bprime) to[bend right=-60] (v5);
    \end{scope}
 \end{tikzpicture}
 \caption{Given the edge-colored graph in Figure~\ref{fig:ECA_orientability_1}, here we can see its orientation corresponding to the assignment $x_1 = x_3 = x_4 = \texttt{true}$ and $x_2 = x_5 = \texttt{false}$.}
 \label{fig:ECA_orientability_1_b}
 \end{figure}

 Now assume that $G$ admits an arc-1-color-avoiding strongly 1-connected orientation, and let us consider such an orientation. For any $i \in [n]$, let $x_i = \texttt{true}$ if the edge $u_i v_i$ is oriented from $u_i$ to $v_i$, and let $x_i = \texttt{false}$ if the edge $u_i v_i$ is oriented from $v_i$ to $u_i$.
 
 First, we show that this truth assignment satisfy $\varphi$. Suppose to the contrary that there exists $j \in [m]$ such that the clause $C_j$ is not satisfied under the constructed truth assignment. Let $B_j$ denote the set of indices of the variables that are contained in the clause $C_j$. Then for each $i \in B_j$, we have $x_i = \texttt{false}$, thus the edge $u_i v_i$ is oriented from $v_i$ to $u_i$. Therefore, after the removal of the color $j$ from the digraph, the arc set $\{ v_i u_i \mid i \in B_j \}$ form a directed cut, contradicting that the given orientation is an arc-1-color-avoiding strongly 1-connected orientation.

 Now we show that for each clause, not all literals have the same truth value. Suppose to the contrary that there exists $j \in [m]$ such that all the variables in the clause $C_j$ have the same truth value. Let $B_j$ denote the set of indices of the variables that are contained in the clause $C_j$. Since $\varphi$ is positive and is satisfied under the constructed truth assignment, this is only possible if $x_i = \texttt{true}$ holds for all $i \in B_j$, that is, the edge $u_i v_i$ is oriented from $u_i$ to $v_i$ for all $i \in B_j$. Therefore, after the removal of the color $j$ from the digraph, the arc set $\{ u_i v_i \mid i \in B_j \}$ form a directed cut, contradicting that the given orientation is an arc-1-color-avoiding strongly 1-connected orientation.
 
 Therefore, $\varphi$ can be satisfied in such a way that for each clause, not all literals have the same truth value.

 \medskip
 
 To assign precisely one color to each edge, we do the following modifications. Let us replace each pair of parallel edges whose color list has length $L \ge 2$ with a path of $L$ pairs of parallel edges, where each pair is assigned exactly one of the colors from the original color list. Let us replace each edge whose color list is empty with another edge of color $m+1$. Add a new vertex $s''$ to the graph and connect it to $s$ with two parallel edges of color $m+1$, and also to each vertex of the original graph $G$ with a path of $m$ parallel edges, where each pair is assigned exactly one of the colors from $[m]$. It is not difficult to show that this modified graph admits an arc-1-color-avoiding strongly 1-connected coloring if and only if $\varphi$ can be satisfied in such a way that for each clause, not all literals have the same truth value.
\end{proof}

Note that if the edges of the graph are colored with only two colors, say red and blue, then we can decide in polynomial time whether $G$ has an arc-$\ell$-color-avoiding strongly $k$-arc-connected orientation. Such an orientation exists if and only if both the subgraphs formed by the red edges and by the blue edges admit a strongly $k$-arc-connected orientation. By the classical result of Nash-Williams~\cite{article:strongly_k-arc-connected_orientations}, this can be decided in polynomial time, and such an orientation --- if exists --- can also be found in polynomial time. Determining the complexity of the same problem when the graph is colored with a constant number of colors~---~in particular, with three~---~remains open.

As for the arc-$\ell$-color-avoiding strongly $k$-vertex-connected variant, the situation is somewhat different. If the edges of~$G$ are colored with two colors and we seek an arc-$\ell$-color-avoiding strongly $k$-vertex-connected orientation, then the problem is polynomial-time solvable for~$k = 1$ by the above. For~$k = 2$, the decision version is also solvable in polynomial time by the theorem of Thomassen~\cite{article:thomassen}, but the complexity of finding an actual orientation remains open. In contrast, for~$k \ge 3$, even the decision version of this problem becomes NP-hard, as shown by Durand de Gevigney~\cite{article:durand_de_gevigney}.

Now let us consider the case when we want to find an arc-color-avoiding rooted connected orientation.

\begin{theorem} \label{thm:ECA_orientability_rooted}
 Given an edge-colored graph $G$ and numbers $k, \ell \in \mathbb{Z}_+$, deciding whether $G$ has an arc-$\ell$-color-avoiding rooted $k$-arc- or $k$-vertex-connected orientation is NP-complete. This problem remains NP-complete even if $k = \ell = 1$.
\end{theorem}
\begin{proof}
 This problem is obviously in NP, and to show that it is NP-hard even for the case $k = \ell = 1$, we reduce the NP-complete problem \textsc{Positive-Linear-NAE-3SAT-Exact4}~\cite{Darmann2020Simple} to it. Our goal in the problem \textsc{Positive-Linear-NAE-3SAT-Exact4} is to decide whether a Boolean formula in conjunctive normal form where each clause contains 3 literals, each literal is not negated and appears in exactly four distinct clauses, each pair of distinct clauses shares at most one variable can be satisfied in such a way that for each clause, not all literals have the same truth value.
 
 Similarly as in the previous proof, first, we assign a list of colors to each edge, and then show how to modify the construction so that every edge is assigned a single color.

 Let $\varphi$ be an an instance of \textsc{Positive-Linear-NAE-3SAT-Exact4} with $n$ variables $x_1, \ldots, x_n$ and with $m \ge 11$ clauses $C_1, \ldots, C_m$. Let $G$ be defined as follows. Let
 \[ V(G) \colonequals \{ r \} \cup \big\{ u_i, v_i \bigm|i \in [n] \big\} \cup \big\{ w_j, w'_j \bigm| j \in [m] \big\} \text{.} \]
 For each $j \in [m]$, connect the vertices $r$ and $w_j$ with an edge having a list of colors $\{j\}$, and connect the vertices $r$ and $w'_j$ with an edge having a list of colors $\{ j+m \}$. For each $i \in [n]$, connect the vertices $u_i$ and $v_i$ with an edge whose list of colors is $\emptyset$. For each $i \in [n]$, let $A_i$ denote the set of indices of the clauses containing the variable $x_i$. For each $i\in[n]$ and $j \in [m]$, if the variable $x_i$ is contained in the clause $C_j$, then connect $w_j$ and $u_i$ with two parallel edges, both having a list of colors $A_i \setminus \{j\}$, and connect $w'_j$ and $v_i$ with two parallel edges, both having a list of colors $\big\{ k+m \bigm| k \in A_i \setminus \{ j \} \big\}$. Clearly, $G$ can be constructed in polynomial time. For an example, see Figure~\ref{fig:arc_orientation_rooted}. 

 \begin{figure}[H]
 \centering
 \begin{tikzpicture}
  \begin{scope}
    \tikzstyle{vertex}=[draw,circle,fill,minimum size=10,inner sep=0]
    \tikzset{paint/.style={draw=#1!50!black, fill=#1!50}, decorate with/.style = {decorate, decoration={shape backgrounds, shape=#1, shape size = 5pt, shape sep = 6pt}}}

    \tikzset{edge_red/.style={postaction={decorate, decoration={markings, mark=between positions 3pt and 1-3pt step 7pt with {\draw[red!50!black,thin,fill=red!50] (30:0.08) -- (150:0.08) -- (210:0.08) -- (330:0.08) -- (30:0.08);}}}}}
    
    \tikzset{edge_blue/.style={postaction={decorate, decoration={markings, mark=between positions 5pt and 1-3pt step 6pt with {\draw[blue!50!black,thin,fill=blue!50] (0:0.08) -- (120:0.08) -- (240:0.08) -- (0:0.08);}}}}}
 
    \tikzset{edge_green/.style={postaction={decorate, decoration={markings, mark=between positions 5pt and 1-3pt step 8.5pt with {\draw[green!50!black,thin,fill=green!50] (0:0.16) -- (90:0.08) -- (180:0.08) -- (270:0.08) -- (0:0.16);}}}}}

    \tikzset{edge_yellow/.style={postaction={decorate, decoration={markings, mark=between positions 2pt and 1-3pt step 7pt with {\draw[yellow!50!black,thin,fill=yellow!50] (0:0.09) -- (72:0.09) -- (144:0.09) -- (216:0.09) -- (288:0.09) -- (0:0.09);}}}}}

    \node[vertex, label={above:$r$}] (r) at (0, 4) {};
    \node[vertex, label={below:$w_1$}] (C1) at (-4, 1) {}; 
    \node[vertex, label={below:$w'_1$}] (Cprime1) at (4, 1) {};
    \node[vertex, label={above:$w_2$}] (C2) at (-4, -1) {};
    \node[vertex, label={above:$w'_2$}] (Cprime2) at (4, -1) {};
    \node[vertex, label={above:$u_1$}] (u1) at (-1, 2) {};
    \node[vertex, label={above:$v_1$}] (v1) at (1, 2) {};
    \node[vertex, label={above:$u_2$}] (u2) at (-1, 0) {};
    \node[vertex, label={above:$v_2$}] (v2) at (1, 0) {};
    \node[vertex, label={below:$u_3$}] (u3) at (-1, -2) {};
    \node[vertex, label={below:$v_3$}] (v3) at (1, -2) {};

    \path[edge_red] (r) .. controls (-1.75, 3.5) and (-3.75,3) .. (C1);
    \path[edge_blue] (r) .. controls (-5.5, 4) and (-6.5, -0.5) .. (C2);
    \path[edge_green] (r) .. controls (1.75, 3.5) and (3.75,3) .. (Cprime1);
    \path[edge_yellow] (r) .. controls (5.5, 4) and (6.5, 0.5) .. (Cprime2);

    \draw[very thick] (C1) to[bend right=-40] (u1);
    \draw[very thick] (u1) to[bend right=15] (C1);
    
    \path[edge_blue] (C1) to[bend right=-25] (u2);
    \path[edge_blue] (u2) to[bend right=0] (C1);

    \path[edge_red] (C2) to[bend right=0] (u2);
    \path[edge_red] (u2) to[bend right=-25] (C2);
    
    \draw[very thick] (C2) to[bend right=15] (u3);
    \draw[very thick] (u3) to[bend right=-45] (C2);
    
    \draw[very thick] (Cprime1) to[bend right=15] (v1);
    \draw[very thick] (v1) to[bend right=-40] (Cprime1);
    
    \path[edge_yellow] (Cprime1) to[bend right=0] (v2);
    \path[edge_yellow] (v2) to[bend right=-25] (Cprime1);
    
    \path[edge_green] (Cprime2) to[bend right=-25] (v2);
    \path[edge_green] (v2) to[bend right=0] (Cprime2);
    
    \draw[very thick] (Cprime2) to[bend right=-45] (v3);
    \draw[very thick] (v3) to[bend right=15] (Cprime2);
    
    \draw[very thick] (u1) -- (v1);
    \draw[very thick] (v2) -- (u2);
    \draw[very thick] (u3) -- (v3);
  \end{scope}

  \begin{scope}[shift={(0,-8.5)}]
    \tikzstyle{vertex}=[draw,circle,fill,minimum size=10,inner sep=0]
    \tikzset{paint/.style={draw=#1!50!black, fill=#1!50}, decorate with/.style = {decorate, decoration={shape backgrounds, shape=#1, shape size = 5pt, shape sep = 6pt}}}

    \tikzset{edge_red/.style={postaction={decorate, decoration={markings, mark=between positions 3pt and 1-8pt step 7pt with {\draw[red!50!black,thin,fill=red!50] (30:0.08) -- (150:0.08) -- (210:0.08) -- (330:0.08) -- (30:0.08);}}}}}
    
    \tikzset{edge_blue/.style={postaction={decorate, decoration={markings, mark=between positions 5pt and 1-8pt step 6pt with {\draw[blue!50!black,thin,fill=blue!50] (0:0.08) -- (120:0.08) -- (240:0.08) -- (0:0.08);}}}}}
 
    \tikzset{edge_green/.style={postaction={decorate, decoration={markings, mark=between positions 5pt and 1-8pt step 8.5pt with {\draw[green!50!black,thin,fill=green!50] (0:0.16) -- (90:0.08) -- (180:0.08) -- (270:0.08) -- (0:0.16);}}}}}

    \tikzset{edge_yellow/.style={postaction={decorate, decoration={markings, mark=between positions 2pt and 1-8pt step 7pt with {\draw[yellow!50!black,thin,fill=yellow!50] (0:0.09) -- (72:0.09) -- (144:0.09) -- (216:0.09) -- (288:0.09) -- (0:0.09);}}}}}

    \node[vertex, label={above:$r$}] (r) at (0, 4) {};
    \node[vertex, label={below:$w_1$}] (C1) at (-4, 1) {}; 
    \node[vertex, label={below:$w'_1$}] (Cprime1) at (4, 1) {};
    \node[vertex, label={above:$w_2$}] (C2) at (-4, -1) {};
    \node[vertex, label={above:$w'_2$}] (Cprime2) at (4, -1) {};
    \node[vertex, label={above:$u_1$}] (u1) at (-1, 2) {};
    \node[vertex, label={above:$v_1$}] (v1) at (1, 2) {};
    \node[vertex, label={above:$u_2$}] (u2) at (-1, 0) {};
    \node[vertex, label={above:$v_2$}] (v2) at (1, 0) {};
    \node[vertex, label={below:$u_3$}] (u3) at (-1, -2) {};
    \node[vertex, label={below:$v_3$}] (v3) at (1, -2) {};

    \path[edge_red] (r) .. controls (-1.75, 3.5) and (-3.75,3) .. (C1);
    \draw[->,>={LaTeX[black,length=10pt, width=10pt]},draw opacity=0] (r) .. controls (-1.75, 3.5) and (-3.75,3) .. (C1);
    \path[edge_blue] (r) .. controls (-5.5, 4) and (-6.5, -0.5) .. (C2);
    \draw[->,>={LaTeX[black,length=10pt, width=10pt]},draw opacity=0] (r) .. controls (-5, 4) and (-6, -0.5) .. (C2);
    \path[edge_green] (r) .. controls (1.75, 3.5) and (3.75,3) .. (Cprime1);
    \draw[->,>={LaTeX[black,length=10pt, width=10pt]},draw opacity=0] (r) .. controls (1.75, 3.5) and (3.75,3) .. (Cprime1);
    \path[edge_yellow] (r) .. controls (5.5, 4) and (6.5, 0.5) .. (Cprime2);
    \draw[->,>={LaTeX[black,length=10pt, width=10pt]},draw opacity=0] (r) .. controls (5, 4) and (6, 0.5) .. (Cprime2);

    \draw[very thick,->,>={LaTeX[black,length=10pt, width=10pt]}] (C1) to[bend right=-40] (u1);
    \draw[very thick,->,>={LaTeX[black,length=10pt, width=10pt]}] (u1) to[bend right=15] (C1);
    
    \path[edge_blue] (C1) to[bend right=-25] (u2);
    \draw[->,>={LaTeX[black,length=10pt, width=10pt]},draw opacity=0] (C1) to[bend right=-25] (u2);
    \path[edge_blue] (u2) to[bend right=0] (C1);
    \draw[->,>={LaTeX[black,length=10pt, width=10pt]},draw opacity=0] (u2) to[bend right=0] (C1);

    \path[edge_red] (C2) to[bend right=0] (u2);
    \draw[->,>={LaTeX[black,length=10pt, width=10pt]},draw opacity=0] (C2) to[bend right=0] (u2);
    \path[edge_red] (u2) to[bend right=-25] (C2);
    \draw[->,>={LaTeX[black,length=10pt, width=10pt]},draw opacity=0] (u2) to[bend right=-25] (C2);
    
    \draw[very thick,->,>={LaTeX[black,length=10pt, width=10pt]}] (C2) to[bend right=15] (u3);
    \draw[very thick,->,>={LaTeX[black,length=10pt, width=10pt]}] (u3) to[bend right=-45] (C2);
    
    \draw[very thick,->,>={LaTeX[black,length=10pt, width=10pt]}] (Cprime1) to[bend right=15] (v1);
    \draw[very thick,->,>={LaTeX[black,length=10pt, width=10pt]}] (v1) to[bend right=-40] (Cprime1);
    
    \path[edge_yellow] (Cprime1) to[bend right=0] (v2);
    \draw[->,>={LaTeX[black,length=10pt, width=10pt]},draw opacity=0] (Cprime1) to[bend right=0] (v2);
    \path[edge_yellow] (v2) to[bend right=-25] (Cprime1);
    \draw[->,>={LaTeX[black,length=10pt, width=10pt]},draw opacity=0] (v2) to[bend right=-25] (Cprime1);
    
    \path[edge_green] (Cprime2) to[bend right=-25] (v2);
    \draw[->,>={LaTeX[black,length=10pt, width=10pt]},draw opacity=0] (Cprime2) to[bend right=-25] (v2);
    \path[edge_green] (v2) to[bend right=0] (Cprime2);
    \draw[->,>={LaTeX[black,length=10pt, width=10pt]},draw opacity=0] (v2) to[bend right=0] (Cprime2);
    
    \draw[very thick,->,>={LaTeX[black,length=10pt, width=10pt]}] (Cprime2) to[bend right=-45] (v3);
    \draw[very thick,->,>={LaTeX[black,length=10pt, width=10pt]}] (v3) to[bend right=15] (Cprime2);
    
    \draw[very thick,->,>={LaTeX[black,length=10pt, width=10pt]}] (u1) -- (v1);
    \draw[very thick,->,>={LaTeX[black,length=10pt, width=10pt]}] (v2) -- (u2);
    \draw[very thick,->,>={LaTeX[black,length=10pt, width=10pt]}] (u3) -- (v3);
  \end{scope}

\end{tikzpicture}
\caption{An example for the constructions used in the proof of Theorem~\ref{thm:ECA_orientability_rooted}. Above, we can see the graph $G$ constructed from the Boolean formula $\varphi = (x_1 \lor x_2) \land (x_2 \lor x_3)$. Below, we can see the orientation corresponding to the assignment $x_1 = x_3 = \texttt{true}$ and $x_2 = \texttt{false}$.}
\label{fig:arc_orientation_rooted}
\end{figure}
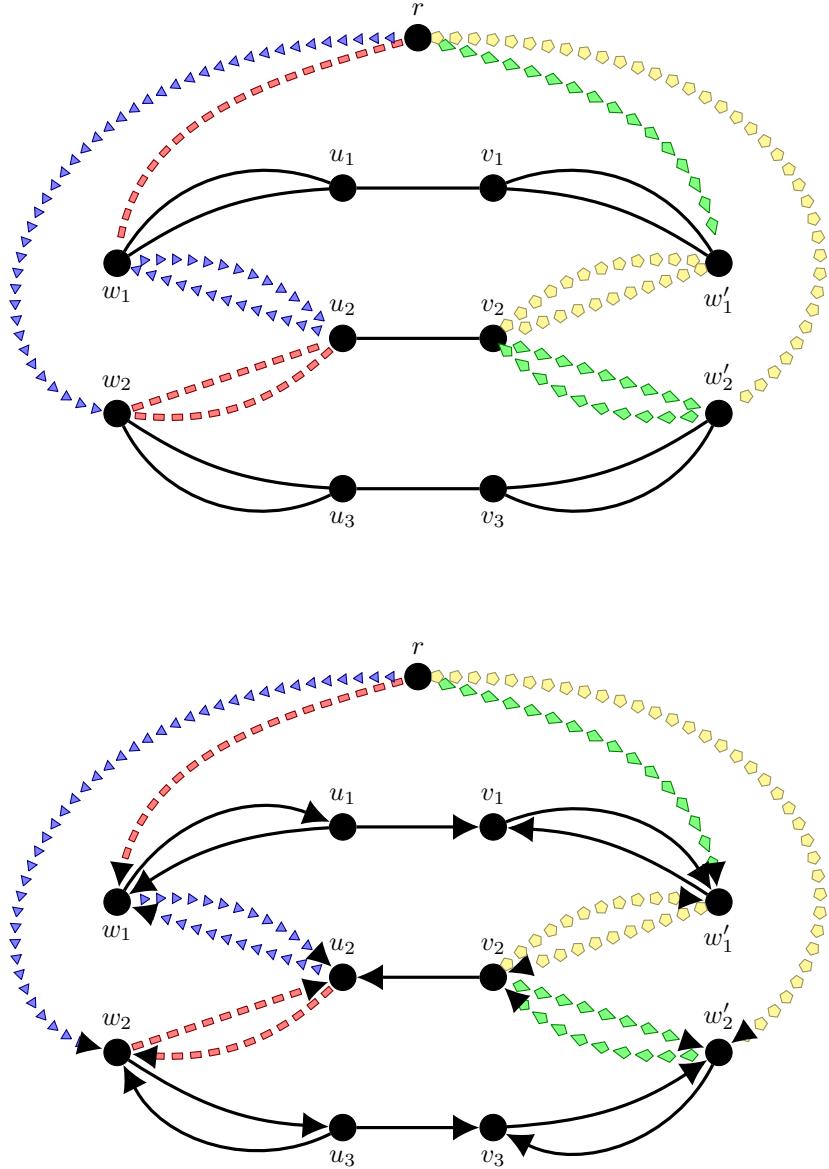

 Now we show that $G$ admits an arc-1-color-avoiding rooted 1-connected orientation with root $r$ if and only if the formula $\varphi$ is satisfiable in such a way that for each clause, not all literals have the same truth value.

 Assume first that $\varphi$ is satisfiable in such a way that for each clause, not all literals have the same truth value. For each $i \in [n]$, let us orient the edge $u_i v_i$ from $u_i$ to $v_i$ if $x_i = \texttt{true}$, and let us orient this edge from $v_i$ to $u_i$ if $x_i = \texttt{false}$. For each $j \in [m]$, let us orient the edge $r w_j$ from $r$ to $w_j$, and let us orient the edge $r w'_j$ from $r$ to $w'_j$. Finally, orient all the parallel edges in opposite directions. Similarly to the proof of Theorem~\ref{thm:ECA_orientability_strong}, it can be shown that with this orientation, $G$ is arc-1-color-avoiding rooted 1-connected orientation with root $r$.

 Now assume that $G$ admits an arc-1-color-avoiding rooted 1-connected orientation with root $r$, and let us consider such an orientation. For any $i \in [n]$, let $x_i = \texttt{true}$ if the edge $u_i v_i$ is oriented from $u_i$ to $v_i$, and let $x_i = \texttt{false}$ if the edge $u_i v_i$ is oriented from $v_i$ to $u_i$. Similarly to the proof of Theorem~\ref{thm:ECA_orientability_strong}, it can be shown that this truth assignment satisfies $\varphi$ in such a way that for each clause, not all literals have the same truth value.

 To assign precisely one color to each edge, we do the following modifications. Let us replace each pair of parallel edges whose color list has length $L \ge 2$ with a path of $L$ pairs of parallel edges, where each pair is assigned exactly one of the colors from the original color list. Let us replace each edge whose color list is empty with another edge of color $2m+1$. Using that each variable appears in exactly four distinct clauses, it is not difficult to show that this modified graph admits an arc-1-color-avoiding strongly 1-connected coloring if and only if $\varphi$ can be satisfied in such a way that for each clause, not all literals have the same truth value.
\end{proof}
 
\subsection{Vertex-color-avoiding connected orientations}

Here we prove the analogues of Theorems~\ref{thm:ECA_orientability_strong}--\ref{thm:ECA_orientability_rooted} for vertex-color-avoiding connectivity. 

\begin{theorem} \label{thm:VCA_orientability_strong}
 Given a vertex-colored graph $G$ and numbers $k, \ell \in \mathbb{Z}_+$, deciding whether $G$ has a vertex-$\ell$-color-avoiding strongly $k$-arc- or $k$-vertex-connected orientation is NP-complete. This problem remains NP-complete even if $k = \ell = 1$.
\end{theorem}

\begin{proof}
 This problem is obviously in NP, and to show that it is NP-hard even for the case $k = \ell = 1$, we reduce the problem of Theorem \ref{thm:ECA_orientability_strong} to it. Our goal in this problem is to decide if a given edge-colored graph admits an arc-1-color-avoiding strongly 1-connected orientation.
   
 Let $G=(V,E)$ be such an arbitrary edge-colored graph on at least two vertices, whose edges are colored with color set $[m]$. Let $G'$ be defined as follows. For each $e = \{ u, v \} \in E(G)$, let $u'_e$ and $v'_e$ be vertices of the same color in $G'$ as that of $e$ in $G$. Let
 \[ V' \colonequals \{ u'_e, v'_e \mid e \in E\}, \]
 let
 \[ V'' \colonequals \{ v'' \mid v \in V \} \]
 be the set of vertices of color $m+1$,
 \[ V''' \colonequals \{ v''' \mid v \in V \} \]
 be the set of vertices of color $m+2$, and let $V(G') \colonequals V' \cup V'' \cup V'''$. For each $e = \{ u, v \} \in E(G)$, connect the vertices $u'_e$ and $v'_e$ with an edge in $G'$, and connect $u'_e$ with two parallel edges both to $u''$ and to $u'''$, and also connect $v'_e$ with two parallel edges both to $v''$ and to $v'''$. Clearly, $G'$ can be constructed in polynomial time. For an example, see Figure~\ref{fig:VCA-orientation-strong}.
   
 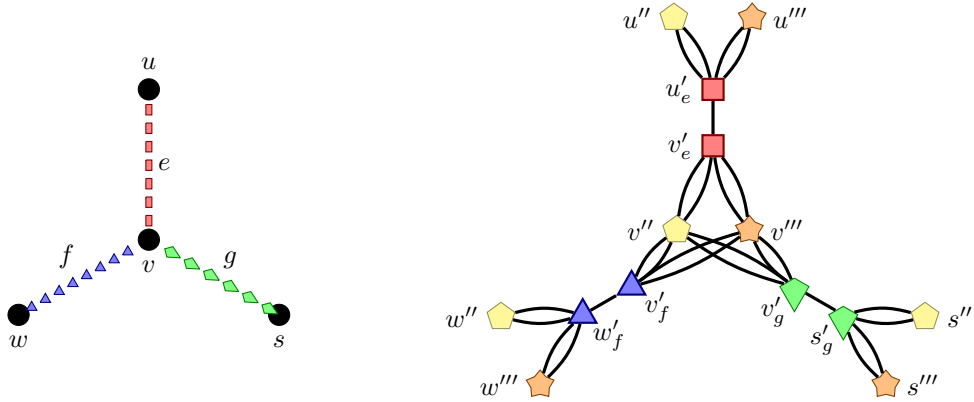
\begin{figure}[!ht]
 \centering
 \begin{tikzpicture}
  \tikzset{edge_red/.style={postaction={decorate, decoration={markings, mark=between positions 3pt and 1-2pt step 7pt with {\draw[red!50!black,thin,fill=red!50] (30:0.08) -- (150:0.08) -- (210:0.08) -- (330:0.08) -- (30:0.08);}}}}}
  \tikzset{edge_blue/.style={postaction={decorate, decoration={markings, mark=between positions 5pt and 1-2pt step 6pt with {\draw[blue!50!black,thin,fill=blue!50] (0:0.08) -- (120:0.08) -- (240:0.08) -- (0:0.08);}}}}}
  \tikzset{edge_green/.style={postaction={decorate, decoration={markings, mark=between positions 5pt and 1-2pt step 8.5pt with {\draw[green!50!black,thin,fill=green!50] (0:0.16) -- (90:0.08) -- (180:0.08) -- (270:0.08) -- (0:0.16);}}}}}

  \tikzstyle{vertex}=[draw,circle,fill,minimum size=8,inner sep=0]
  
  \tikzstyle{vertex_red}=[draw, thick, red!50!black, shape=rectangle, fill=red!50, minimum size=8pt, inner sep=0]
  \tikzstyle{vertex_blue}=[draw, thick, blue!50!black, shape=regular polygon, regular polygon sides=3, fill=blue!50, minimum size=12pt, inner sep=0]
  \tikzstyle{vertex_green}=[draw, green!50!black, kite, fill=green!50, minimum size=11pt, inner sep=0.1pt]
  \tikzstyle{vertex_yellow}=[draw, yellow!50!black, shape=regular polygon, regular polygon sides=5, fill=yellow!50, minimum size=11pt, inner sep=0]
  \tikzstyle{vertex_orange}=[draw, orange!50!black, shape=star, star points=5, fill=orange!50, minimum size=11pt, inner sep=0]

    \tikzset{vtx/.style={circle, inner sep=2.5pt, draw=black, thin}}
    \tikzset{v_black/.style ={vtx, fill=black}}
    \tikzset{v_red/.style   ={vtx, fill=red}}
    \tikzset{v_green/.style ={vtx, fill=green}}
    \tikzset{v_blue/.style  ={vtx, fill=blue}}
    \tikzset{v_purple/.style={vtx, fill=purple}}
    \tikzset{v_brown/.style ={vtx, fill=brown}}
   
   \node[vertex] (u) at (90:2) [label=above:$u$] {};
   \node[vertex] (v) at (0,0) [label=below:$v$] {};
   \node[vertex] (w) at (210:2) [label=below:$w$] {};
   \node[vertex] (s) at (330:2) [label=below:$s$] {};
   
   \path[edge_red] (v) -- (u) node[pos=0.5, right] {$e$};
   \path[edge_blue] (v) -- (w) node[pos=0.5, above left] {$f$};
   \path[edge_green] (v) -- (s) node[pos=0.5, above right] {$g$};
  
  \begin{scope}[shift={(7.5,0)}]
   \coordinate (center) at (0,0);
   \coordinate (top_tip) at (90:3);
   \coordinate (left_tip) at (210:3);
   \coordinate (right_tip) at (330:3);

   \node[vertex_red] (ve) at (90:1.25) [label=left:$v'_e$] {};
   \node[vertex_red] (ue) at (90:2) [label=left:$u'_e$] {};
   \node[vertex_blue] (vf) at (210:1.25) [label={[xshift=10pt, yshift=-24pt]: $v'_f$}] {};
   \node[vertex_blue] (wf) at (210:2) [label={[xshift=10pt, yshift=-24pt]: $w'_f$}] {};
   \node[vertex_green] (vg) at (330:1.25) [label={[xshift=-8pt, yshift=-22pt]: $v'_g$}] {};
   \node[vertex_green] (sg) at (330:2) [label={[xshift=-8pt, yshift=-22pt]: $s'_g$}] {};

   \node[vertex_yellow] (v'') at (165:0.5) [label=left:$v''$] {};
   \node[vertex_orange] (v''') at (15:0.5) [label=right:$v'''$] {};
   \node[vertex_yellow] (u'') at (100:3) [label=left:$u''$] {};
   \node[vertex_orange] (u''') at (80:3) [label=right:$u'''$] {};
   \node[vertex_yellow] (w'') at (200:3) [label=left:$w''$] {};
   \node[vertex_orange] (w''') at (220:3) [label=left:$w'''$] {};
   \node[vertex_yellow] (s'') at (340:3) [label=right:$s''$] {};
   \node[vertex_orange] (s''') at (320:3) [label=right:$s'''$] {};

   \draw[very thick] (ue) -- (ve);
   \draw[very thick] (wf) -- (vf);
   \draw[very thick] (sg) -- (vg);

   \draw[very thick] (ue) to[bend right=15] (u'');
   \draw[very thick] (ue) to[bend right=-15] (u'');
   \draw[very thick] (ue) to[bend right=15] (u''');
   \draw[very thick] (ue) to[bend right=-15] (u''');
   \draw[very thick] (wf) to[bend right=15] (w'');
   \draw[very thick] (wf) to[bend right=-15] (w'');
   \draw[very thick] (wf) to[bend right=15] (w''');
   \draw[very thick] (wf) to[bend right=-15] (w''');
   \draw[very thick] (sg) to[bend right=15](s'');
   \draw[very thick] (sg) to[bend right=-15](s'');
   \draw[very thick] (sg) to[bend right=15] (s''');
   \draw[very thick] (sg) to[bend right=-15] (s''');
        
   \draw[very thick] (ve) to[bend right=15] (v'');
   \draw[very thick] (ve) to[bend right=-15] (v'');
   \draw[very thick] (ve) to[bend right=15] (v''');
   \draw[very thick] (ve) to[bend right=-15] (v''');
   \draw[very thick] (vf) to[bend right=15] (v'');
   \draw[very thick] (vf) to[bend right=-15] (v'');
   \draw[very thick] (vf) to[bend right=10] (v''');
   \draw[very thick] (vf) to[bend right=-10] (v''');
   \draw[very thick] (vg) to[bend right=10] (v'');
   \draw[very thick] (vg) to[bend right=-10] (v'');
   \draw[very thick] (vg) to[bend right=15] (v''');
   \draw[very thick] (vg) to[bend right=-15] (v''');
  \end{scope}
 \end{tikzpicture}
 \caption{An example for the construction used in the proof of Theorem~\ref{thm:VCA_orientability_strong}. On the left, we can see an edge-colored graph $G$, and on the right, we can see the vertex-colored graph $G'$ constructed from $G$.}
 \label{fig:VCA-orientation-strong}
 \end{figure}
    
 Now we show that $G$ admits an arc-1-color-avoiding strongly 1-connected orientation if and only if $G'$ admits a vertex-1-color-avoiding strongly 1-connected orientation.

 \smallskip
 
 First, assume that $G$ admits an arc-1-color-avoiding strongly 1-connected orientation, and consider such an orientation of $G$. Let us orient the parallel pairs of edges of $G'$ forward and backward. All the other edges of $G'$ are of the form $u'_e v'_e$ for some vertices $u, v \in V(G)$ with $e = \{ u, v \} \in E(G)$. Let us orient such an edge $\{ u'_e, v'_e \}$ from $u'_e$ to $v'_e$ in $G'$ if $e$ is oriented from $u$ to $v$ in $G$; otherwise let the edge $\{ u'_e, v'_e \}$ be oriented from $v'_e$ to $u'_e$. It is not difficult to show that this is a vertex-1-color-avoiding strongly 1-connected orientation of $G'$.

 \smallskip

 Now assume that $G'$ admits a vertex-1-color-avoiding strongly 1-connected orientation, and consider such an orientation of $G$. For any $e = \{ u, v \} \in E(G)$, let us orient the edge $e$ from $u$ to $v$ in $G$ if the edge $\{ u'_e, v'_e \}$ is oriented from $u'_e$ to $v'_e$ in $G'$; otherwise let the edge $e$ be oriented from $v$ to $u$. It is not difficult to show that this is an arc-1-color-avoiding strongly 1-connected orientation of $G$. 
\end{proof}

\begin{theorem} \label{thm:VCA_orientability_rooted}
 Given a vertex-colored graph $G$ and numbers $k, \ell \in \mathbb{Z}_+$, deciding whether $G$ has a vertex-$\ell$-color-avoiding rooted $k$-arc- or $k$-vertex-connected orientation is NP-complete. This problem remains NP-complete even if $k = \ell = 1$.
\end{theorem}
\begin{proof}
 This problem is obviously in NP, and to show that it is NP-hard even for the case $k = \ell = 1$, we reduce the problem of Theorem~\ref{thm:ECA_orientability_rooted} to it. Our goal in this problem is to decide if a given edge-colored graph with a given root $r$ admits an arc-1-color-avoiding $r$-rooted 1-connected orientation.

 Let $G=(V,E)$ be an arbitrary edge-colored graph on at least two vertices with a distinguished vertex $r$. Let $G'$ be the same vertex-colored graph defined in the proof of Theorem~\ref{thm:VCA_orientability_strong}, and consider the graph $G' - r'''$, that is, the graph obtained from $G'$ by removing the vertex $r'''$.

 Analogously to the proof of Theorem~\ref{thm:VCA_orientability_strong}, it can be shown that $G$ admits an arc-1-color-avoiding $r$-rooted 1-connected orientation if and only if $G' - r'''$ admits a vertex-1-color-avoiding $r''$-rooted 1-connected orientation.
\end{proof}

\subsection{Internally vertex-color-avoiding connected orientations}

Now let us consider the case when we want to find an internally vertex-color-avoiding strongly connected orientation.

\begin{theorem} \label{thm:IVCA_orientability_strong}
 Given a vertex-colored graph $G$ and numbers $k, \ell \in \mathbb{Z}_+$, deciding whether $G$ has an internally vertex-$\ell$-color-avoiding strongly $k$-arc- or $k$-vertex-connected orientation is NP-complete. This problem remains NP-complete even if $k = \ell = 1$.
\end{theorem}
\begin{proof}
 This problem is obviously in NP, and to show that it is NP-hard even for the case $k = \ell = 1$, we reduce the problem of Theorem~\ref{thm:ECA_orientability_strong} to it. Our goal in this problem is to decide if a given dge-colored graph admits an arc-1-color-avoiding strongly 1-connected orientation.

 Let $G = (V,E)$ be an arbitrary edge-colored graph on at least two vertices, and let $G'$ be the same graph defined in the proof of Theorem~\ref{thm:VCA_orientability_strong}. Then by Lemma~\hyperlink{item:ii_of_lemma_differently_colored_neighbor}{\ref*{lemma:differently_colored_neighbor}.\ref*{item:ii_of_differently_colored_neighbor}} and by the proof of Theorem~\ref{thm:VCA_orientability_strong}, the graph $G$ admits an arc-1-color-avoiding strongly 1-connected orientation if and only if $G'$ admits an internally vertex-1-color-avoiding 1-connected orientation.
\end{proof}

\begin{theorem} \label{thm:IVCA_orientability_rooted}
 Given a vertex-colored graph $G$ and numbers $k, \ell \in \mathbb{Z}_+$, deciding whether $G$ has an internally vertex-$\ell$-color-avoiding rooted $k$-arc- or $k$-vertex-connected orientation is NP-complete. This problem remains NP-complete even if $k = \ell = 1$.
\end{theorem}
\begin{proof}
 This problem is obviously in NP, and to show that it is NP-hard even for the case $k = \ell = 1$, we reduce the problem of Theorem~\ref{thm:ECA_orientability_rooted} to it. Our goal in this problem is to decide if a given edge-colored graph with a given root $r$ admits an arc-1-color-avoiding $r$-rooted 1-connected orientation.

 Let $G=(V,E)$ be an arbitrary edge-colored graph on at least two vertices with a distinguished vertex $r$, whose edges are colored with color set $[m]$. Let $G'$ be the same vertex-colored graph defined in the proof of Theorem~\ref{thm:VCA_orientability_strong}, and consider the vertex-colored graph $G''$ that can be obtained from $G' - r'''$ by recoloring $r''$ to color $m+3$. 
 
 Then by Lemma~\hyperlink{item:iii_of_lemma_differently_colored_neighbor}{\ref*{lemma:differently_colored_neighbor}.\ref*{item:iii_of_differently_colored_neighbor}} and by the proof of Theorem~\ref{thm:VCA_orientability_rooted}, the graph $G$ admits an arc-1-color-avoiding $r$-rooted 1-connected orientation if and only if $G''$ admits an internally vertex-1-color-avoiding $r''$-rooted 1-connected orientation.
\end{proof}

\subsection{Simultaneous orientation and color-avoiding connected coloring}

Finally, we consider the problem of deciding whether the edges of a given graph $G$ can be oriented so that the obtained digraph has a color-avoiding strongly or rooted connected coloring with a given number of colors.

We start with the case of finding an orientation with an arc-$\ell$-color-avoiding strongly $k$-arc- or $k$-vertex-connected coloring.

\begin{theorem}
 Let $G$ be a graph and let $c, k, \ell \in \mathbb{Z}_+$. Deciding whether the edges of $G$ can be oriented so that the obtained digraph has an arc-$\ell$-color-avoiding strongly $k$-arc- or $k$-vertex-connected coloring with at most $c$ colors is NP-complete. This problem remains NP-complete even if $G$ is 4-regular, $k = \ell = 1$ and $c = 2$.
 \end{theorem}
\begin{proof}
 This problem is obviously in NP, and by Theorem~\ref{thm:Kotzig} it is also NP-hard: for $k = \ell = 1$ and $c = 2$ and for a 4-regular graph $G$, our problem becomes equivalent to deciding whether $G$ can be decomposed into two edge-disjoint Hamiltonian cycles.
\end{proof}

Now let us consider the case of finding an orientation with an arc-$\ell$-color-avoiding $r$-rooted 1-connected coloring.

\begin{theorem}
 Let $G$ be a graph with a given vertex $r \in V(G)$, and let $c, \ell \in \mathbb{Z}_+$. Then the edges of $G$ can be oriented so that the obtained digraph has an arc-$\ell$-color-avoiding $r$-rooted 1-connected coloring with at most $c$ colors if and only if $c \ge \ell+1$ and $G$ has an $r$-rooted $(\ell+1)$-connected orientation. Moreover, we can find such an orientation along with an arc-$\ell$-color-avoiding $r$-rooted 1-connected coloring (if exists) in polynomial time.
\end{theorem}
\begin{proof}
 The first statement directly follows from Theorem~\ref{thm:rooted_mincolors}. Since we can find an $r$-rooted $(\ell+1)$-connected orientation of $G$ (or show that it does not have any) in polynomial time~\cites{frank2,gabow} and by Theorem~\ref{thm:rooted_arc_mincoloring} we can find an arc-$\ell$-color-avoiding $r$-rooted 1-connected coloring of the obtained digraph with $\ell+1$ colors in polynomial time, we conclude the proof of the theorem.
\end{proof}

The problem of finding an orientation along with an arc-$\ell$-color-avoiding $r$-rooted $k$-arc- or $k$-vertex-connected coloring with a given number of colors for $k \ge 2$ remains open.

Let us now turn to the case of finding an orientation with a vertex-$\ell$-color-avoiding strongly $k$-arc-connected coloring.

\begin{theorem} \label{thm:col&orient_vertex_strongly_kedge}
 Let $G$ be a graph, and let $c, k, \ell \in \mathbb{Z}_+$. Then the edges of $G$ can be oriented so that the obtained digraph has a vertex-$\ell$-color-avoiding strongly $k$-arc-connected coloring with at most $c$ colors if and only if $G$ is $2k$-edge-connected. Moreover, we can find such an orientation along with a vertex-$\ell$-color-avoiding $r$-rooted $k$-arc-connected coloring (if exists) in polynomial time.
\end{theorem}
\begin{proof}
 The theorem directly follows from the theorem of Nash-Williams~\cite{article:strongly_k-arc-connected_orientations}, from Corollary~\ref{cor:vertexmincolor} and from the fact that a strongly $k$-arc-connected orientation can be found in polynomial time~\cites{frank3,gabow2,iwata}.
\end{proof}

Now we consider the case of finding an orientation with a vertex-$\ell$-color-avoiding strongly $k$-vertex-connected coloring.

\begin{theorem} \label{thm:col&orient_vertex_rooted_kedge}
 Let $G$ be a graph, and let $c, k, \ell \in \mathbb{Z}_+$. Then the edges of $G$ can be oriented so that the obtained digraph has a vertex-$\ell$-color-avoiding strongly $k$-vertex-connected coloring with at most $c$ colors if and only if $G$ has a strongly $k$-vertex-connected orientation. If $k = 1$, then we can find such an orientation along with a vertex-$\ell$-color-avoiding strongly $k$-vertex-connected coloring (if exists) in polynomial time. If, however, $k \ge 3$, then finding such an orientation along with a vertex-$\ell$-color-avoiding strongly $k$-vertex-connected coloring (if exists) is NP-hard.
\end{theorem}
\begin{proof}
 The first statement directly follows from Corollary~\ref{cor:vertexmincolor}.
 Observe, that the case $k=1$ was already settled in Theorem~\ref{thm:col&orient_vertex_strongly_kedge}. Finally, the statement about the case $k \ge 3$ follows from the fact that finding a strongly $k$-vertex-connected orientation for $k \ge 3$ is NP-hard~\cite{article:durand_de_gevigney}.
\end{proof}

The problem of finding an orientation along with a vertex-$\ell$-color-avoiding strongly 2-vertex-connected coloring with a given number of colors remains open.

Next, let us consider the case of finding an orientation with a vertex-$\ell$-color-avoiding $r$-rooted $k$-arc- or $k$-vertex-connected coloring.

\begin{theorem}
 Let $G$ be a graph with a given vertex $r \in V(G)$, and let $c, k, \ell \in \mathbb{Z}_+$. Then the edges of $G$ can be oriented so that the obtained digraph has a vertex-$\ell$-color-avoiding $r$-rooted $k$-arc-connected coloring with at most $c$ colors if and only if $G$ has an $r$-rooted $k$-arc-connected orientation. Moreover, we can find such an orientation along with a vertex-$\ell$-color-avoiding $r$-rooted $k$-arc-connected coloring (if exists) in polynomial time.
\end{theorem}
\begin{proof}
 The theorem directly follows from Corollary~\ref{cor:vertexmincolor} and from the fact that an $r$-rooted $k$-arc-connected orientation can be found in polynomial time~\cites{frank2,gabow}.
\end{proof}

\begin{theorem}
 Let $G$ be a graph, and let $c, k, \ell \in \mathbb{Z}_+$. Then the edges of $G$ can be oriented so that the obtained digraph has a vertex-$\ell$-color-avoiding $r$-rooted $k$-vertex-connected coloring with at most $c$ colors if and only if $G$ has an $r$-rooted $k$-vertex-connected orientation. If $k \le 2$, then we can find such an orientation along with a vertex-$\ell$-color-avoiding $r$-rooted $k$-vertex-connected coloring (if exists) in polynomial time. If, however, $k \ge 3$, then finding such an orientation along with a vertex-$\ell$-color-avoiding $r$-rooted $k$-vertex-connected coloring (if exists) is NP-hard.
\end{theorem}
\begin{proof}
 The first statement directly follows from Corollary~\ref{cor:vertexmincolor}.
 Observe, that the case $k=1$ was already settled in Theorem~\ref{thm:col&orient_vertex_rooted_kedge}. The statement about the case $k=2$ directly follows from Corollary~\ref{cor:vertexmincolor} and from the fact that an $r$-rooted 2-vertex-connected orientation can be found in polynomial time~\cites{article:whitty, thesis:plehn, article:cheriyan}. Finally, the statement about the case $k \ge 3$ follows from the fact that finding an $r$-rooted $k$-vertex-connected orientation for $k \ge 3$ is NP-hard~\cite{thesis:de_gevigney}.
\end{proof}

Now let us consider the case of finding an orientation with an internally vertex-$\ell$-color-avoiding strongly $k$-arc- or $k$-vertex-connected coloring.

\begin{theorem}
 Let $G$ be a graph and let $c, k, \ell \in \mathbb{Z}_+$. Deciding whether the edges of $G$ can be oriented so that the obtained digraph has an internally vertex-$\ell$-color-avoiding strongly $k$-arc- or $k$-vertex-connected coloring with at most $c$ colors is NP-complete. This problem remains NP-complete even if $k = \ell = 1$ and $c = 2$.
\end{theorem}
\begin{proof}
 The problem is clearly in NP, and to show that it is NP-hard even if $k = \ell = 1$ and $c = 2$, we reduce the NP-complete problem presented in Theorem~\ref{thm:ivckecvc} to it. Let $G$ be an arbitrary graph. Let $G'$ be the graph that can be obtained from $G$ by adding a parallel edge to each of its edges. Then it is not difficult to see that $G$ has an internally vertex-$\ell$-color-avoiding strongly $k$-arc- or $k$-vertex-connected coloring with at most $c$ colors if and only if the edges of $G'$ can be oriented so that the obtained digraph has an internally vertex-$\ell$-color-avoiding strongly $k$-arc- or $k$-vertex-connected coloring with at most $c$ colors.
\end{proof}

Finally, let us study the case of finding an orientation with an internally vertex-$\ell$-color-avoiding $r$-rooted $k$-arc- or $k$-vertex-connected coloring.

\begin{theorem}
 Let $G$ be a graph with a given vertex $r \in V(G)$ and let $c, k, \ell \in \mathbb{Z}_+$. Deciding whether the edges of $G$ can be oriented so that the obtained digraph has an internally vertex-$\ell$-color-avoiding $r$-rooted $k$-arc- or $k$-vertex-connected coloring with at most $c$ colors is NP-complete. This problem remains NP-complete even if $k = \ell = 1$ and $c = 2$.
\end{theorem}
\begin{proof}
 The problem is clearly in NP. Now we show that the problem is NP-hard even for $k = \ell = 1$ and $c=2$ by reducing the NP-complete problem of recognizing 2-colorable (loopless) hypergraphs~\cite{article:hypergraph_2-colorability} to it.
 
 Let $\mathcal{H}$ be an arbitrary hypergraph without loops and let $G'$ be the same graph defined in the proof of Theorem~\ref{thm:ivckecvc}. Let $G$ be the graph that can be obtained from $G'$ as follows. Add a parallel edge to each edge of $G$, and then add a new vertex $r$ to so obtained graph and connect it to every vertex $v \in V(G')$ corresponding to a vertex of $V(\mathcal{H})$.
 
 Assume first that $\mathcal{H}$ has a proper 2-vertex-coloring with red and blue colors. Let us orient the parallel pairs of edges forward and backward, and let us orient every edge incident to $r$ away from $r$. Also, let us use the 2-vertex-coloring of $\mathcal{H}$ for the corresponding vertices in $G$, and color all the other vertices of $G$ red. Then by Proposition~\hyperlink{item:iii_of_prop_IVCA}{\ref*{proposition:equiv_def_of_IVCA}.\ref*{item:iii_of_IVCA}}, it is not difficult to see that this yields an internally vertex-1-color-avoiding $r$-rooted 1-connected digraph.

 Now assume that the edges of $G$ can be oriented so that the obtained digraph has an internally vertex-1-color-avoiding $r$-rooted 1-connected coloring with red and blue colors. We can assume that the parallel pairs of edges are oriented forward and backward and that every edge incident to $r$ is oriented away from~$r$ (otherwise we can reorient the edges). Then assigning the same colors of the corresponding vertices in $\mathcal{H}$ (regardless the colors of the remaining vertices of $G$) yields a proper 2-vertex-coloring of $\mathcal{H}$.
\end{proof}

\section{Conclusions}

In this work, we characterized the graphs that admit color-avoiding connected colorings, and studied the problem of determining such colorings using as few colors as possible. We further examined the problem of orienting the edges of a colored graph --- or simultaneously orienting and coloring an uncolored one --- so as to achieve color-avoiding strong or rooted connectivity. Our results also extend to $k$-edge- and $k$-vertex-connectivity, as well as to settings where multiple colors must be simultaneously avoided. In some cases, we generalized the framework to matroids.

Several cases remain open. In particular, the complexity of finding an edge-$\ell$-color-avoiding connected coloring with minimum number of colors is unresolved for $\ell \ge 2$. Similarly, determining the minimum number of colors needed for arc-$\ell$-color-avoiding $r$-rooted $k$-arc- or $k$-vertex-connected colorings remains open for $k \ge 2$ and $\ell \in \mathbb{Z}_+$. It is also open whether an orientation along with a vertex-$\ell$-color-avoiding strongly 2-vertex-connected coloring can be efficiently found for $\ell \in \mathbb{Z}_+$. Finally, it also remains open whether the problems in Theorems~\ref{thm:ECA_orientability_strong}--\ref{thm:IVCA_orientability_rooted} are still NP-complete when the number of colors is constant, in particular, three.

\paragraph{Acknowledgement.} The authors would like to express their gratitude to Roland Molontay for useful conversations and for his support during the research process. The authors are also grateful to Krist\'{o}f B\'{e}rczi for his comments regarding the manuscript. 
J\'{o}zsef Pint\'{e}r was supported by the European Union project RRF-2.3.1–21-2022–00004 within the framework of the Artificial Intelligence National Laboratory. Pint\'{e}r was also funded by Project No.\ KDP-IKT-2023-900-I1-00000957/0000003 with support provided by the Ministry of Culture and Innovation of Hungary from the National Research, Development and Innovation Fund, financed under the KDP-2023 funding scheme.
The research of Kitti Varga was supported by the Ministry of Innovation and Technology of Hungary from the National Research, Development and Innovation Fund -- grant ADVANCED 150556.

\end{document}